\documentclass[11pt]{amsart}
\usepackage{pst-node}
\usepackage{pst-intersect}

\usepackage{etex}
\usepackage{amsfonts}
\usepackage{amsmath}
\usepackage[latin1]{inputenc}
\usepackage{setspace}
\usepackage{mathrsfs}
\usepackage{graphicx}
\graphicspath{ {./pics/} }
\usepackage{caption}
\usepackage{subcaption}
\usepackage{stmaryrd}
\usepackage[english]{babel}
\usepackage{enumitem}
\usepackage{amssymb}
\usepackage{amsthm}
\usepackage{bbding}
\usepackage[all]{xy}
\usepackage[a4paper]{geometry} 

\usepackage{setspace}
\usepackage[bottom]{footmisc} 				

\usepackage{tikz-cd}
\usepackage{pstricks}
\usepackage{pst-slpe}
\usepackage{multido}							
\usepackage{pstricks-add}
\usepackage{pst-plot}

\theoremstyle{plain}
\newtheorem{theorem}{Theorem}[section]
\newtheorem{lemma}[theorem]{Lemma}
\newtheorem{proposition}[theorem]{Proposition}
\newtheorem{corollary}[theorem]{Corollary}
\newtheorem{definition}[theorem]{Definition}

\theoremstyle{definition}

\newtheorem{example}[theorem]{Example}
\newtheorem{remark}[theorem]{Remark}

\theoremstyle{remark}

\newcommand{\C}{\mathbb{C}}

\newcommand{\cS}{\mathcal{S}}

\numberwithin{equation}{section}

\definecolor{diskmid}{gray}{0.8}
	\definecolor{line}{rgb}{0,0,1}
	\definecolor{diskin}{gray}{0.8}%
	\definecolor{diskout}{gray}{0.8}
\psset{unit=0.8cm}

\def\cross{
\psset{unit=0.25cm}
\begin{pspicture}[shift=-0.5](-1,-1)(1,1)
			\pscircle[fillstyle=solid,fillcolor=diskin,linewidth=0.8pt,linestyle=dotted](0,0){1}
			\degrees[8]
			\psline[linecolor= line,linewidth=1pt](1;1)(1;5)
			\psline[linecolor= blue,linewidth=1pt](1;3)(0.25;3)
			\psline[linecolor= blue,linewidth=1pt](1;7)(0.25;7)	
		\end{pspicture}
		\psset{unit=0.5cm}
}

\def\monoid{
\psset{unit=0.25cm}
\begin{pspicture}[shift=-0.5](-1,-1)(1,1)
			\pscircle[fillstyle=solid,fillcolor=diskin,linewidth=0.8pt,linestyle=dotted ](0,0){1}
			\degrees[8]
			\pscurve[linecolor= line,linewidth=1.2pt](1;3)(0.9;3)(0.4;2)(0.9;1)(1;1)
			\pscurve[linecolor= line,linewidth=1.2pt](1;5)(0.9;5)(0.4;6)(0.9;7)(1;7)
		\end{pspicture}  
		\psset{unit=0.5cm}
	}

\def\id{	
\psset{unit=0.25cm}
\begin{pspicture}[shift=-0.5](-1,-1)(1,1)
			\pscircle[fillstyle=solid,fillcolor=diskin,linewidth=0.8pt,linestyle=dotted](0,0){1}
			\degrees[8]
			\pscurve[linecolor= line,linewidth=1.2pt](1;5)(0.9;5)(0.4;4)(0.9;3)(1;3)
			\pscurve[linecolor= line,linewidth=1.2pt](1;1)(0.9;1)(0.4;0)(0.9;7)(1;7)
		\end{pspicture} 
		\psset{unit=0.5cm}
		}

\title{The skein category of the annulus}
\date{\today}                                          
\author{K. Al Qasimi}
\address{KdV Institute for Mathematics, University of Amsterdam, Science Park 105-107, 1098 XG Amsterdam, The Netherlands.}
\email{s.k.s.k.s.alqasemi@uva.nl}

\author{J.V. Stokman}
\address{KdV Institute for Mathematics, University of Amsterdam, Science Park 105-107, 1098 XG Amsterdam, The Netherlands.}
\email{j.v.stokman@uva.nl}

\begin{document}
\maketitle

\begin{abstract}
We construct the skein category $\mathcal{S}$ of the annulus and show that it is equivalent to the affine Temperley-Lieb category of Graham and Lehrer. It leads to a skein theoretic description of the extended affine Temperley-Lieb algebras. We construct an endofunctor of $\mathcal{S}$ that corresponds, on the level of tangle diagrams, to the insertion of an arc connecting the inner and outer boundary of the annulus. We use it to define and construct towers of extended affine Temperley-Lieb algebra modules. It allows us to construct a tower of modules acting on spaces of link patterns on the punctured disc which play an important role in the study of loop models. In case of trivial Dehn twist we show that the direct sum of the representation spaces of the link pattern tower defines a graded algebra that may be regarded as a relative version of the Roger-Yang skein algebra of arcs and links on the punctured disc. We also describe the link pattern tower in terms of fused extended affine Temperley-Lieb algebra modules.
\end{abstract}

\section{Introduction}

In \cite{Kau, Kau2} Kauffman constructed a knot invariant based on an elementary combinatorial rule
for eliminating crossings in the associated knot diagram. In doing so he introduced the two skein relations, the Kauffman skein relation \eqref{kauffman} and the loop removal relation \eqref{loopremoval}.  It has led to skein theory, the study of knots and links in 3-manifolds modulo the Kauffman skein relation and the loop removal relation, see, e.g., \cite{Tur,Prz2, Mor, BW}. The current paper deals with the 3-manifold $A\times [0,1]$, with $A$ the annulus in the complex plane.

We introduce and study the skein category $\mathcal{S}$ of the annulus $A$. It is the linear category with objects the nonnegative integers and morphisms $\textup{Hom}_{\mathcal{S}}(m,n)$ the linear skein of the annulus with a.
In other words, $\textup{Hom}_{\mathcal{S}}(m,n)$ consists of 
the ambient isotopy classes of $(m,n)$-tangle diagrams on the annulus modulo the equivalence relation generated by the  Kauffman 
skein relation and the loop removal relation. The skein category $\mathcal{S}$ is in fact a strict monoidal category with the tensor product obtained from the relative Kauffman bracket skein product introduced in Przytycki and Sikora \cite{PrzSik}.

Following closely Przytycki \cite[\S 3]{Prz}, we construct a relative version of the Kauffman bracket  
to prove that the skein category $\mathcal{S}$ is equivalent to Graham and Lehrer's \cite{GL}  affine Temperley-Lieb category. As a consequence it follows that the endomorphism algebra
$\textup{End}_{\mathcal{S}}(n):=\textup{Hom}_{\mathcal{S}}(n,n)$ is
isomorphic to Green's \cite{Gre} $n$-affine diagram algebra, also known as the (extended) affine Temperley-Lieb algebra. 

We will define an endofunctor $\mathcal{I}$ of $\mathcal{S}$ called the arc-insertion functor, which on the level of morphisms inserts a new arc connecting the inner and outer boundary of the annulus in a particular way while undercrossing all arcs it meets along the way. On the level of endomorphisms it provides a tower of algebras 
\[
\textup{End}_{\mathcal{S}}(0)\overset{\mathcal{I}_0}{\longrightarrow}\textup{End}_{\mathcal{S}}(1)\overset{\mathcal{I}_1}{\longrightarrow}
\textup{End}_{\mathcal{S}}(2)\overset{\mathcal{I}_2}{\longrightarrow}\cdots
\] 
with connecting maps $\mathcal{I}_n$ the algebra maps 
$\mathcal{I}|_{\textup{End}_{\mathcal{S}}(n)}: \textup{End}_{\mathcal{S}}(n)\rightarrow
\textup{End}_{\mathcal{S}}(n+1)$. This tower 
was considered before in the context of knot theory \cite{AlH} and in the context of fusion of extended affine Temperley-Lieb algebra modules \cite{GS} respectively.
It differs from the arc-tower from e.g. \cite{GS,CMPX}, 
which is defined with respect to the two-step algebra embedding $\textup{End}_{\mathcal{S}}(n)\rightarrow\textup{End}_{\mathcal{S}}(n+2)$ that corresponds to the identification of an idempotent subalgebra of $\textup{End}_{\mathcal{S}}(n+2)$ with $\textup{End}_{\mathcal{S}}(n)$. 

We introduce and study towers 
\[
V_0\overset{\phi_0}{\longrightarrow}V_1\overset{\phi_1}{\longrightarrow}V_2
\overset{\phi_2}{\longrightarrow}V_3\overset{\phi_3}{\longrightarrow}\cdots
\]
of extended affine Temperley-Lieb algebra modules. These are chains of
left $\textup{End}_{\mathcal{S}}(n)$-modules $V_n$ ($n\in\mathbb{Z}_{\geq 0}$) connected by
morphisms $\phi_n: V_n\rightarrow\textup{Res}^{\mathcal{I}_n}(V_{n+1})$ of 
$\textup{End}_{\mathcal{S}}(n)$-modules, where $\textup{Res}^{\mathcal{I}_n}(V_{n+1})$
is the $\textup{End}_{\mathcal{S}}(n+1)$-module $V_{n+1}$ viewed as 
$\textup{End}_{\mathcal{S}}(n)$-module via the
algebra map $\mathcal{I}_n: \textup{End}_{\mathcal{S}}(n)\rightarrow\textup{End}_{\mathcal{S}}(n+1)$. 

Our motivation for studying such towers stems from integrable models in statistical physics with extended affine Temperley-Lieb algebra symmetry. Examples are inhomogeneous dense loop models and inhomogeneous XXZ spin-$\frac{1}{2}$ chains with quasi-periodic boundary conditions, see, e.g.,
\cite{KP,FZJZ, dGP} and references therein. In this context the representation space $V_n$ of the tower represents the state space of the model at system size $n$ and the connecting maps relate the models of different system sizes.

We introduce a special tower of extended affine Temperley-Lieb algebra modules, which we will
call the {\it link pattern tower}. It depends on a free parameter $v$, called the {\it twist weight} of the
tower. 
For even $n$ the representation space is spanned by ambient isotopy classes of $(0,n)$-tangle diagrams
in $A$ without crossings and without loops, connecting $n$ marked points on the outer boundary
of $A$.
For odd $n$ the tangle diagrams include a defect line
connecting the outer boundary to the inner boundary, 
 and we add the rule that Dehn twists of the defect line may be removed by $v$
(see \eqref{Dehnremoval}).
The $\textup{End}_{\mathcal{S}}(n)$-action is described as follows. The skein class of an $(n,n)$-tangle diagram on the annulus acts on a diagram $D\in V_n$ by placing $D$
inside the $(n,n)$-tangle diagram, removing crossings
and contractible loops by the skein relations, and removing  
noncontractible loops by a particular weight factor depending on $v$ (see \eqref{nonloopremoval}).

For even $n$ the connecting maps $\phi_n: V_n\rightarrow V_{n+1}$ of the link pattern tower correspond, from the skein
theoretic perspective,
 to the insertion of a defect line, undercrossing all arcs it meets along the way. These maps were considered before in \cite{FZJZ} in the study of the inhomogeneous dense loop model on the half-infinite cylinder. The connecting maps $\phi_n: V_n\rightarrow V_{n+1}$ for odd $n$ are more subtle.
{}From a skein theoretic perspective they can be described as follows. The connecting
map $\phi_n$ acts on a diagram by detaching the defect line from the inner boundary and reconnecting it to the outer boundary in two different ways, either encircling the hole of the annulus before reattaching it to the outer boundary, or not. These two contributions are given explicit weights depending on the twist weight $v$ and on the Temperley-Lieb algebra parameter, see Theorem \ref{intertwinerstower}. 

We show that the link pattern tower is nondegenerate for generic parameter values, in the sense that the induced morphisms
$\widehat{\phi}_n: \textup{Ind}^{\mathcal{I}_n}(V_n)\rightarrow V_{n+1}$ of $\textup{End}_{\mathcal{S}}(n+1)$-modules are surjective, where $\textup{Ind}^{\mathcal{I}_n}(V_n)$ is the 
$\textup{End}_{\mathcal{S}}(n+1)$-module obtained by inducing $V_n$ along the algebra map $\mathcal{I}_n: \textup{End}_{\mathcal{S}}(n)\rightarrow
\textup{End}_{\mathcal{S}}(n+1)$. We relate the link pattern tower to the recently introduced fusion \cite{GS} of extended affine Temperley-Lieb algebra modules. We construct for
each $n\in\mathbb{Z}_{\geq 0}$ a fused $\textup{End}_{\mathcal{S}}(n+1)$-module $W_{n+1}$ and a morphism
$\psi_{n}: W_{n+1}\rightarrow V_{n+1}$ of $\textup{End}_{\mathcal{S}}(n+1)$-modules such that $\widehat{\phi}_n$ factorizes through $\psi_n$. The $\textup{End}_{\mathcal{S}}(n+1)$-module $W_{n+1}$ is obtained
by fusing the 
$\textup{End}_{\mathcal{S}}(n)$-module $V_n$ with an one-dimensional 
$\textup{End}_{\mathcal{S}}(1)$-module.

For twist weight $v=1$ 
the representation spaces of the link pattern tower may be naturally identified with
spaces of link patterns on the punctured disc $\mathbb{D}^\ast=\mathbb{D}\setminus\{0\}$ by shrinking the hole of the annulus
to a point. The resulting modules play an important role in the description of the inhomogeneous dense  loop models on the half-infinite cylinder (see, e.g., \cite{KP,FZJZ}). We show that in this case the direct sum of the representation spaces of the link pattern tower is a graded algebra, and as such may be viewed as a relative version of Roger's and Yang's \cite[Def. 2.3]{RY}
skein algebra of arcs and links on the punctured disc $\mathbb{D}^*$. In this skein algebra perspective multiple endpoints of arcs in $\mathbb{D}^*\times [0,1]$ may connect to the pole $\{0\}\times [0,1]$ but each line segment $\{\xi\}\times [0,1]$ above the marked points $\xi$ on the outer boundary
of $\mathbb{D^*}$ is met by only one endpoint. The number of endpoints on $\partial\mathbb{D}\times[0,1]$ is the grading of the associated element in the algebra.
In this identification our connecting maps $\phi_n$ for $n$ odd relate to the {\it puncture-skein relation} in \cite{RY}, which is the skein theoretic reduction rule when multiple arcs connect to the centre pole $\{0\}\times [0,1]$. In fact, the connecting map $\phi_n$, for each $n\in\mathbb{Z}_{\geq 0}$, just becomes right multiplication by the class of the identity of $\textup{End}_{\mathcal{S}}(1)$ on the $n$th graded piece of the algebra.

In our future work \cite{ANS} we use the link pattern tower to construct a tower of solutions to quantum Knizhnik-Zamolodchikov (qKZ) equations. The tower consists of $V_n$-valued solutions of qKZ equations ($n\in\mathbb{Z}_{\geq 0}$) which are compatible with respect to the connecting maps $\phi_n$. At the stochastic/combinatorial value of the extended affine Temperley-Lieb parameter, the $V_n$-valued solution in the tower reduces to the ground state of the dense loop model on the half-infinite cylinder with perimeter $n$. In that case the tower structure gives explicit recursion relations of the ground states with respect to the system size, leading to a refinement of the results in \cite{FZJZ}.

The structure of the paper is as follows.
In Sections \ref{section2} and \ref{section3} we define the category of tangle diagrams, $\mathcal{T}$ and the skein category of the annulus, $\mathcal{S}$, respectively. We explain in Section \ref{section3} that $\mathcal{S}$ is a monoidal category, with the tensor product obtained from the relative Kauffman bracket skein product from \cite{PrzSik}. 
In Section \ref{section4} we introduce Graham's and Lehrer's \cite{GL} affine Temperley-Lieb category $\mathcal{TL}$, whose morphisms are defined in terms of affine diagrams, and we show that the affine
Temperley-Lieb category is equivalent to $\mathcal{S}$.
In Section \ref{section5} we define the extended affine Temperley-Lieb algebra, $\textup{TL}_n$, algebraically and we recall Green's \cite{Gre} result that $\textup{TL}_n$ is equivalent to the endomorphism algebra $\textup{End}_{\mathcal{TL}}(n)$ of $\mathcal{TL}$. Combined with the result from Section \ref{section4} it leads to three different realizations of the extended affine Temperley-Lieb algebra (skein theoretic, combinatorial and algebraic).
In Section \ref{section6} we define the arc insertion functor $\mathcal{I}: \mathcal{S} \rightarrow \mathcal{S}$. 
In Section \ref{section7} we introduce the notion of towers of extended affine Temperley-Lieb algebra modules. In Section \ref{section8} we construct the link pattern tower and we explain how it gives rise to a relative version of the Roger-Yang \cite{RY} skein algebra on the punctured disc. We show in Section \ref{section9} how the link pattern tower is related to fusion. Finally 
in the appendix we discuss how the resulting tower of extended affine Temperley-Lieb algebras lifts to extended affine braid groups and extended affine Hecke algebras, and we discuss a $B$-type  presentation of the extended affine Temperley-Lieb algebra.

\subsection{Acknowledgments}
The authors would like to thank Bernard Nienhuis as motivation for this paper came from the joint work \cite{ANS}. 
We thank the two referees of the paper for their comments which led to substantial improvements of the paper. 
The work by Kayed Al Qasimi is supported by the Ministry of Education of the United Arab Emirates.
Diagrams were coded using PSTricks.

\section{The category of tangle diagrams}\label{section2}
Consider the three-manifold $\Sigma:=A\times [0,1]$ with $A$ the annulus 
\[
A:=\{z\in\mathbb{C} \,\, | \,\, 1\leq |z|\leq 2\}
\]
in the complex plane. We think of $\Sigma$ as a thickened cylinder in $\mathbb{R}^3$,
$$\Sigma = \psset{unit=0.5cm}
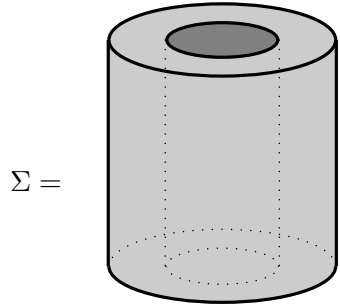
\begin{pspicture}[shift=-4](8,10)
	\psframe*[linecolor=diskin](1,2)(7,8)
	\psellipse*[linecolor=diskin](4,8)(3,1)
	\psellipse*[linecolor=diskin](4,2)(3,1)
	\psellipse*[linecolor=gray](4,8)(1.5,0.5)
	\psellipse[linewidth=1.2pt](4,8)(1.5,0.5)
	\psline[linewidth = 1.2pt](1,2)(1,8)
	\pssavepath[linewidth=1.2pt]{top}{\psellipse[linewidth=1.2pt](4,8)(3.025,1)}
	\pssavepath[linewidth=1.2pt]{bottom}{\psellipticarc[linewidth=1.2pt](4,2)(3.025,1){180}{0}}
	\psellipticarc[linestyle=dotted,linewidth=0.6pt](4,2)(3,1){0}{180}
	\psline[linewidth=1.2pt](7,2)(7,8)
	\psline[linewidth=0.6pt,linestyle=dotted](2.5,2)(2.5,8)
	\psline[linewidth=0.6pt,linestyle=dotted](5.5,2)(5.5,8)
	\psellipse[linewidth=0.6pt,linestyle=dotted](4,2)(1.5,0.5)
	\end{pspicture}
	$$

Write $\partial A=C_i\cup C_o$ for the boundary of $A$ with $C_i:=S^1=
\{z\in\mathbb{C} \,\, | \,\, |z|=1\}$ and $C_o=\{z\in\mathbb{C} \,\, | \,\, |z|=2\}$ (the indices ``$i$" and ``$o$" stand for inner and outer, respectively). Give $A$ the counterclockwise orientation. 

Set $\zeta_n:=\exp(2\pi i/n)$ for $n\in\mathbb{Z}_{>0}$.  
Let $m,n\in\mathbb{Z}_{\geq 0}$ with $m+n$ even. A (framed) $(m,n)$-tangle $T$ in $\Sigma$ 
is a disjoint union of smooth framed loops and $\frac{1}{2}(m+n)$ framed arcs in $\Sigma$
satisfying:
\begin{enumerate}
\item[{\bf a.}] The loops are in the interior of $\Sigma$.
\item[{\bf b.}]  The $m+n$ marked points $(2\xi_m^{j-1},1)$ ($1\leq j\leq m$) and 
$(2\xi_n^{i-1},0)$ ($1\leq i\leq n$), framed along $\partial A\times\{1\}$ and $\partial A\times \{0\}$ with the orientation induced from $A$, are the endpoints of the framed arcs.
\end{enumerate}
Let $\textup{proj}: \Sigma\rightarrow A$ be the map obtained by projecting radially on the outer wall $C_o\times [0,1]$ of $\Sigma$ and identifying $C_o\times [0,1]\simeq A$ by collapsing the wall $C_o\times[0,1]$ inwards onto the floor $A\times\{0\}$ of $\Sigma$.
The projection $D:=\textup{proj}(T)$ of an $(m,n)$-tangle $T$ in general position with respect to $\textup{proj}$, together with the crossing data at the crossing points in the diagram, is called an $(m,n)$-tangle diagram in $A$.

If we draw a picture of an $(m,n)$-tangle diagram then we label the inner points $\xi_m^{i-1}$ on the diagram by $i$ ($i=1,\ldots,m$) and the outer points $2\xi_n^{j-1}$ by $j$ ($j=1,\ldots,n$).
An example of a tangle diagram is given in Figure \ref{example}.
\begin{figure}[htbp]
\begin{center}
\begin{pspicture}(-1.8,-1.8)(1.8,1.8)
    	\SpecialCoor
    	\pscircle[fillstyle=solid, fillcolor=diskin,linewidth=1.2pt](0,0){1.5}
    	\pscircle[fillstyle=solid, fillcolor=white, linewidth=1.2pt](0,0){0.5}
    	\degrees[15]
     	\rput{0}(1.7;0){\tiny $1$}
    	\rput{0}(1.7;3){\tiny $2$}
    	\rput{0}(1.7;6){\tiny $3$}
	\rput{0}(1.7;9){\tiny $4$}
	\rput{0}(1.7;12){\tiny $5$}
	\rput{0}(0.25;0){\tiny $1$}
    	\rput{0}(0.25;5){\tiny $2$}
    	\rput{0}(0.25;10){\tiny $3$}
	\pscurve[linecolor=line,linewidth=1.2pt](1.5;3)(1.4;3)(1.1;0)(1.4;12)(1.5;12)
	\pscurve[linecolor=line,linewidth=1.2pt](0.5;0)(0.6;0)(0.7;2.5)(0.6;5)(0.5;5)
	\pscurve[linecolor=line,linewidth=1.2pt](1.5;6)(1.4;6)(1;7.5)(1.4;9)(1.5;9)
	\pscircle[linewidth=1.2pt,linecolor=line](1;5){0.25}
	\psdot[linecolor=diskin,dotsize=0.3](1.1;0)
	\pscurve[linecolor=line,linewidth=1.2pt](1.5;0)(1.4;0)(1;0)(0.8;12)(0.6;10)(0.5;10)
   \end{pspicture} 
\caption{An example of a (3,5)-tangle diagram in $A$.}
\label{example}
\end{center}
\end{figure}
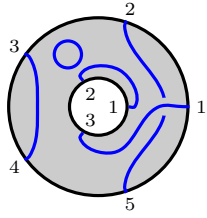

We say that two $(m,n)$-tangle diagrams $D$ and $D^\prime$ in $A$ are equivalent if we can transform $D$ to $D'$ by a planar isotopy of the annulus that fixes the boundary.
That is, there exists a
smooth ambient isotopy $h: A\times [0,1]\rightarrow A$ fixing $\partial A$ pointwise, satisfying
$h(D,1)=D^\prime$ and respecting the crossing data.  If $D$ is an $(m,n)$-tangle diagram we write $\overline{D}\in\textup{Hom}_{\mathcal{T}}(m,n)$ for its equivalence class.

\begin{definition}\label{tanglecat}
The category $\mathcal{T}$ of tangle diagrams
in $A$ is the category with objects $\mathbb{Z}_{\geq 0}$ and
morphisms $\textup{Hom}_{\mathcal{T}}(m,n)$ the equivalence classes of $(m,n)$-tangle diagrams in $A$ if $m+n$ is even, and the empty set if $m+n$ is
odd. 
The composition map
\[
\textup{Hom}_{\mathcal{T}}(k,m)\times
\textup{Hom}_{\mathcal{T}}(m,n)\rightarrow
\textup{Hom}_{\mathcal{T}}(k,n),\qquad (\overline{D},\overline{D^\prime})\mapsto 
\overline{D^\prime}\circ\overline{D}
\]
is defined as follows: $\overline{D^\prime}\circ \overline{D}:=\overline{D^\prime\circ D}$ with
$D^\prime\circ D$ the $(k,n)$-tangle diagram obtained 
by rescaling $D$ to $\{z\in\mathbb{C}\,\, | \,\, 1\leq |z|\leq \frac{3}{2}\}$, rescaling $D^\prime$
to $\{z\in\mathbb{C}\,\, | \,\, \frac{3}{2}\leq |z|\leq 2\}$ and placing $D$ inside $D^\prime$.
The identity morphism $\textup{Id}_n\in\textup{End}_{\mathcal{T}}(n)$ is the equivalence class of the tangle diagram with straight line arcs from $\xi_n^{j-1}$ to $2\xi_n^{j-1}$ for $j=1,\ldots,n$ and no loops (it is the empty diagram for $n=0$). 
\end{definition}
An example of the composition of two tangle diagrams is given in \eqref{compositionexample}.
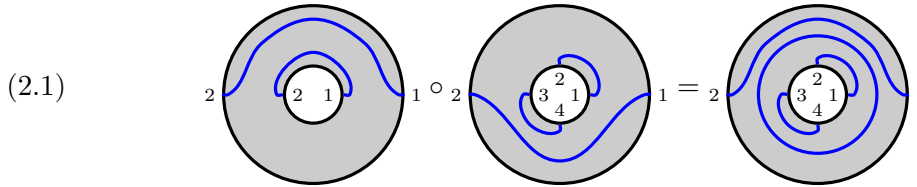
\begin{equation}\label{compositionexample} 
\begin{pspicture}[shift=-1.8](-1.8,-1.8)(1.8,1.8)
      	\SpecialCoor
    	\pscircle[fillstyle=solid, fillcolor=diskin,linewidth=1.2pt](0,0){1.5}
    	\pscircle[fillstyle=solid, fillcolor=white, linewidth=1.2pt](0,0){0.5}
    	\degrees[8]
     	\rput{0}(1.7;0){\tiny $1$}
    	\rput{0}(1.7;4){\tiny $2$}
	\rput{0}(0.25;0){\tiny $1$}
    	\rput{0}(0.25;4){\tiny $2$}
	\pscurve[linecolor=line,linewidth=1.2pt](1.5;0)(1.4;0)(1.2;1)(1.25;2)(1.2;3)(1.4;4)(1.5;4)
	\pscurve[linecolor=line,linewidth=1.2pt](0.5;0)(0.6;0)(0.7;2)(0.6;4)(0.5;4)
   \end{pspicture}
   \circ
        \begin{pspicture}[shift=-1.8](-1.8,-1.8)(1.8,1.8)
    	\SpecialCoor
    	\pscircle[fillstyle=solid, fillcolor=diskin,linewidth=1.2pt](0,0){1.5}
    	\pscircle[fillstyle=solid, fillcolor=white, linewidth=1.2pt](0,0){0.5}
    	\degrees[8]
     	\rput{0}(1.7;0){\tiny $1$}
    	\rput{0}(1.7;4){\tiny $2$}
	\rput{0}(0.26;0){\tiny $1$}
    	\rput{0}(0.26;2){\tiny $2$}
    	\rput{0}(0.26;4){\tiny $3$}
    	\rput{0}(0.26;6){\tiny $4$}
	\pscurve[linecolor=line,linewidth=1.2pt](1.5;0)(1.4;0)(1.1;6)(1.4;4)(1.5;4)
	\pscurve[linecolor=line,linewidth=1.2pt](0.5;0)(0.6;0)(0.7;1)(0.6;2)(0.5;2)
	\pscurve[linecolor=line,linewidth=1.2pt](0.5;4)(0.6;4)(0.7;5)(0.6;6)(0.5;6)
   \end{pspicture}   
    =
    \begin{pspicture}[shift=-1.8](-1.8,-1.8)(1.8,1.8)
    	\SpecialCoor
    	\pscircle[fillstyle=solid, fillcolor=diskin,linewidth=1.2pt](0,0){1.5}
    	\pscircle[fillstyle=solid, fillcolor=white, linewidth=1.2pt](0,0){0.5}
    	\degrees[8]
     	\rput{0}(1.7;0){\tiny $1$}
    	\rput{0}(1.7;4){\tiny $2$}
	\rput{0}(0.26;0){\tiny $1$}
    	\rput{0}(0.26;2){\tiny $2$}
    	\rput{0}(0.26;4){\tiny $3$}
    	\rput{0}(0.26;6){\tiny $4$}
	\pscurve[linecolor=line,linewidth=1.2pt](1.5;0)(1.4;0)(1.2;1)(1.25;2)(1.2;3)(1.4;4)(1.5;4)
	\pscurve[linecolor=line,linewidth=1.2pt](0.5;0)(0.6;0)(0.7;1)(0.6;2)(0.5;2)
	\pscurve[linecolor=line,linewidth=1.2pt](0.5;4)(0.6;4)(0.7;5)(0.6;6)(0.5;6)
	\pscircle[linecolor=line,linewidth=1.2pt](0,0){1}
   \end{pspicture}   
\end{equation}


\section{The skein category of the annulus}\label{section3}
It is well known that skein modules on the strip $\mathbb{R}\times [0,1]$ form the morphisms
of a strict monoidal, linear category called the skein category, see, e.g., \cite[Chpt. XII]{Tur2}. In this section we extend this result to skein modules on the annulus.

Write $\mathbb{C}[\textup{Hom}_{\mathcal{T}}(m,n)]$ for the complex vector space with
linear basis the equivalence classes of $(m,n)$-tangle diagrams in $A$. We take it to be 
$\{0\}$ if $m+n$ is odd. Extend the category
$\mathcal{T}$ of tangle diagrams in $A$ to a linear category $\textup{Lin}(\mathcal{T})$
with objects $\mathbb{Z}_{\geq 0}$, morphisms
$\textup{Hom}_{\textup{Lin}(\mathcal{T})}(m,n):=\mathbb{C}[\textup{Hom}_{\mathcal{T}}(m,n)]$, and composition map the complex bilinear extension of the composition map of $\mathcal{T}$.
The skein category on the annulus is now defined as the quotient category
obtained from $\textup{Lin}(\mathcal{T})$ 
by modding out the {\it Kauffman skein relations} \cite{Kau,Kau2}:

\begin{definition}\label{congruence}
Let $t^{\frac{1}{4}}$ be a nonzero complex number. The skein category $\mathcal{S}=\mathcal{S}(t^{\frac{1}{4}})$ of the annulus $A$ is the quotient of 
$\textup{Lin}(\mathcal{T})$ by the equivalence relation obtained by taking the linear and transitive closure of the following local relations on tangle diagrams:
\begin{enumerate}
\item[{\bf a.}] The \emph{Kauffman skein relation} $\overline{D}\sim t^{\frac{1}{4}}\overline{D^\prime}+
t^{-\frac{1}{4}}\overline{D^{\prime\prime}}$ with $D,D^\prime,D^{\prime\prime}$ three tangle diagrams that are identical except in a small open disc in $A$ where they are as shown
	\begin{align*}
		& 
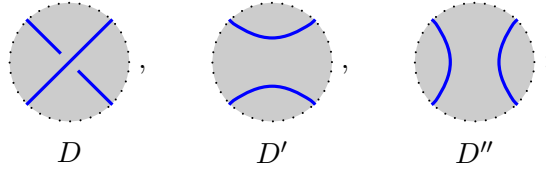
\begin{pspicture}[shift=-1.9](-1,-2)(1,1)
			\pscircle[fillstyle=solid,fillcolor=diskin,linewidth=0.8pt,linestyle=dotted](0,0){1}
			\degrees[8]
			\psline[linecolor= line,linewidth=1.2pt](1;1)(1;5)
			\psline[linecolor= blue,linewidth=1.2pt](1;3)(0.2;3)
			\psline[linecolor= blue,linewidth=1.2pt](1;7)(0.2;7)	
			\rput(0,-1.5){$D$}
		\end{pspicture}\; ,\qquad
		 \begin{pspicture}[shift=-1.9](-1,-2)(1,1)
			\pscircle[fillstyle=solid,fillcolor=diskin,linewidth=0.8pt,linestyle=dotted ](0,0){1}
			\degrees[8]
			\pscurve[linecolor= line,linewidth=1.2pt](1;3)(0.9;3)(0.4;2)(0.9;1)(1;1)
			\pscurve[linecolor= line,linewidth=1.2pt](1;5)(0.9;5)(0.4;6)(0.9;7)(1;7)
			\rput(0,-1.5){$D'$}
		\end{pspicture}\; ,\qquad  
		 \begin{pspicture}[shift=-1.9](-1,-2)(1,1)
			\pscircle[fillstyle=solid,fillcolor=diskin,linewidth=0.8pt,linestyle=dotted](0,0){1}
			\degrees[8]
			\pscurve[linecolor= line,linewidth=1.2pt](1;5)(0.9;5)(0.4;4)(0.9;3)(1;3)
			\pscurve[linecolor= line,linewidth=1.2pt](1;1)(0.9;1)(0.4;0)(0.9;7)(1;7)
			\rput(0,-1.5){$D''$}
		\end{pspicture} \;; 
	\end{align*}

\item[{\bf b.}] The \emph{loop removal relation} $\overline{D}\sim -(t^{\frac{1}{2}}+t^{-\frac{1}{2}})
\overline{D^\prime}$
with $D,D^\prime$ two tangle diagrams that are identical 
except in a small open disc in $A$ where they are as shown
		 \begin{align*}
		&
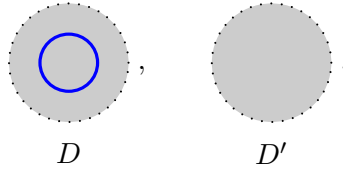
\begin{pspicture}[shift=-1.9](-1,-2)(1,1)
			\pscircle[fillstyle=solid,fillcolor=diskin,linewidth=0.8pt,linestyle=dotted](0,0){1}
			\pscircle[linecolor= line,linewidth=1.2pt](0;0){0.5} 
			\rput(0,-1.5){$D$}
		\end{pspicture}\; ,\qquad
		\begin{pspicture}[shift=-1.9](-1,-2)(1,1)
			\pscircle[fillstyle=solid,fillcolor=diskin,linewidth=0.8pt,linestyle=dotted](0,0){1}
			\rput(0,-1.5){$D'$}
		\end{pspicture}\;.
		\end{align*}

\end{enumerate}
\end{definition}

Note that if $m+n$ is odd then $\textup{Hom}_{\mathcal{S}}(m,n)=\{0\}$.
If $D$ is a tangle diagram in $A$ then we will write 
$[D]$ for the corresponding element in $\textup{Hom}_{\mathcal{S}}(m,n)$. 
We write $\mathbf{1}_n=[\textup{Id}_n]\in\textup{End}_{\mathcal{S}}(n)$
for the identity morphism ($n\in\mathbb{Z}_{\geq 0}$).

As is customary in skein theory, we write the Kauffman skein relation 
in $\textup{Hom}_{\mathcal{S}}(m,n)$ as
	\begin{align}
		&\begin{pspicture}[shift=-0.9](-1,-1)(1,1)
			\pscircle[fillstyle=solid,fillcolor=diskin,linewidth=0.8pt,linestyle=dotted](0,0){1}
			\degrees[8]
			\psline[linecolor= line,linewidth=1.2pt](1;1)(1;5)
			\psline[linecolor= blue,linewidth=1.2pt](1;3)(0.2;3)
			\psline[linecolor= blue,linewidth=1.2pt](1;7)(0.2;7)	
		\end{pspicture}
 = t^{\frac{1}{4}}\; 
		\begin{pspicture}[shift=-0.9](-1,-1)(1,1)
			\pscircle[fillstyle=solid,fillcolor=diskin,linewidth=0.8pt,linestyle=dotted ](0,0){1}
			\degrees[8]
			\pscurve[linecolor= line,linewidth=1.2pt](1;3)(0.9;3)(0.4;2)(0.9;1)(1;1)
			\pscurve[linecolor= line,linewidth=1.2pt](1;5)(0.9;5)(0.4;6)(0.9;7)(1;7)
		\end{pspicture}
		+ t^{-\frac{1}{4}} \;
		\begin{pspicture}[shift=-0.9](-1,-1)(1,1)
			\pscircle[fillstyle=solid,fillcolor=diskin,linewidth=0.8pt,linestyle=dotted](0,0){1}
			\degrees[8]
			\pscurve[linecolor= line,linewidth=1.2pt](1;5)(0.9;5)(0.4;4)(0.9;3)(1;3)
			\pscurve[linecolor= line,linewidth=1.2pt](1;1)(0.9;1)(0.4;0)(0.9;7)(1;7)
		\end{pspicture} \; \label{kauffman}
	\end{align}
and the loop removal relation in the skein module $\textup{Hom}_{\mathcal{S}}(m,n)$ 
as
		\begin{align}
		&\begin{pspicture}[shift=-0.9](-1,-1)(1,1)
			\pscircle[fillstyle=solid,fillcolor=diskin,linewidth=0.8pt,linestyle=dotted](0,0){1}
			\pscircle[linecolor= line,linewidth=1.2pt](0;0){0.5} 
		\end{pspicture}
	=	-(t^{\frac{1}{2}} +t^{-\frac{1}{2}}) \;
		\begin{pspicture}[shift=-0.9](-1,-1)(1,1)
			\pscircle[fillstyle=solid,fillcolor=diskin,linewidth=0.8pt,linestyle=dotted](0,0){1}
		\end{pspicture}\;,\label{loopremoval}
		\end{align} 
\noindent
with the disc showing the local neighbourhood in $A$ where the tangle diagrams differ.
We will also write down identities in skein modules by depicting both sides of the equation as linear combinations of the tangle diagrams $D$ representing $[D]$.

\begin{remark}
The important observation, due to Kauffman \cite{Kau,Kau2}, is that $[D]\in\textup{Hom}_{\mathcal{S}}(m,n)$ is invariant under the Reidemeister moves R1', R2 and R3 
(see Figure \ref{Reidemeister}) and their mirror versions, applied to the $(m,n)$-tangle diagram $D$ in $A$. Hence $[D]$ represents the ambient isotopy class of the associated framed $(m,n)$-tangle in
$\Sigma$.
\end{remark}
\begin{figure}[htbp]
\psset{unit=0.7cm}
	\begin{subfigure}[b]{0.4\textwidth}
		\centering
			\begin{pspicture}[shift=-0.9](-1,-1)(1,1)
			\pscircle[fillstyle=solid,fillcolor=diskin,linewidth=0.8pt,linestyle=dotted, ](0,0){1}
			\degrees[8]
			\psline[linecolor= line,linewidth=1.2pt](1;2)(1;6)
			\end{pspicture}
		$=$
			\begin{pspicture}[shift=-0.9](-1,-1)(1,1)
			\pscircle[fillstyle=solid,fillcolor=diskin,linewidth=0.8pt,linestyle=dotted, ](0,0){1}
			\degrees[8]
			\psecurve[linecolor= line,linewidth=1.2pt](0,1.1)(0,1)(0,0.5)(0.2,0.2)(0.5,0.4)(0.3,0.65)(-0.2,0.4)
			\psecurve[linecolor= line,linewidth=1.2pt](0.3,-0.65)(0.5,-0.4)(0.2,-0.2)(0,-0.5)(0,-1)(0,-1.1)
			\psdot[linecolor=diskin,dotsize=0.3](0,0.55)
			\psdot[linecolor=diskin,dotsize=0.3](0,-0.55)
			\psecurve[linecolor= line,linewidth=1.2pt](0.5,0.4)(0.3,0.65)(-0.2,0.4)(-0.3,0.1)(-0.3,0)(-0.3,-0.1)(-0.2,-0.4)(0.3,-0.65)(0.5,-0.4)(0.2,-0.2)
			\end{pspicture}
			\subcaption*{R1'}
	 \end{subfigure}
	   \begin{subfigure}[b]{0.4\textwidth}
	 \centering
	 \begin{pspicture}[shift=-0.9](-1,-1)(1,1)
			\pscircle[fillstyle=solid,fillcolor=diskin,linewidth=0.8pt,linestyle=dotted, ](0,0){1}
			\degrees[8]
			\psline[linecolor= line,linewidth=1.2pt](1;2.5)(1;5.5)
			\psline[linecolor= line,linewidth=1.2pt](1;1.5)(1;6.5)
			\end{pspicture}
			$=$
			\begin{pspicture}[shift=-0.9](-1,-1)(1,1)
			\pscircle[fillstyle=solid,fillcolor=diskin,linewidth=0.8pt,linestyle=dotted, ](0,0){1}
			\degrees[8]
			\pscurve[linecolor= line,linewidth=1.2pt](1;2.5)(0.7;2.5)(0.4;0)(0.7;5.5)(1;5.5)
			\psdot[linecolor=diskin,dotsize=0.3](0.45;2)
			\psdot[linecolor=diskin,dotsize=0.3](0.45;6)
			\pscurve[linecolor= line,linewidth=1.2pt](1;1.5)(0.7;1.5)(0.4;4)(0.7;6.5)(1;6.5)
			\end{pspicture}
			\subcaption*{R2}
	  \end{subfigure}
	 \begin{subfigure}[b]{0.4\textwidth}
	 \centering
	  \begin{pspicture}[shift=-0.9](-1,-1)(1,1)
			\pscircle[fillstyle=solid,fillcolor=diskin,linewidth=0.8pt,linestyle=dotted, ](0,0){1}
			\degrees[8]
			\psline[linecolor= line,linewidth=1.2pt](1;2.5)(1;6.5)
			\psdot[linecolor=diskin,dotsize=0.3](0;0)
			\psdot[linecolor=diskin,dotsize=0.3](0.5;6.5)
			\pscurve[linecolor= line,linewidth=1.2pt](1;2)(0.3;0)(1;6)
			\psdot[linecolor=diskin,dotsize=0.3](0.5;1.5)
			\psline[linecolor= line,linewidth=1.2pt](1;1.5)(1;5.5)
		\end{pspicture}
		$=$
		\begin{pspicture}[shift=-0.9](-1,-1)(1,1)
			\pscircle[fillstyle=solid,fillcolor=diskin,linewidth=0.8pt,linestyle=dotted, ](0,0){1}
			\degrees[8]
			\psline[linecolor= line,linewidth=1.2pt](1;2.5)(1;6.5)
			\psdot[linecolor=diskin,dotsize=0.3](0;0)
			\psdot[linecolor=diskin,dotsize=0.3](0.5;2.5)
			\pscurve[linecolor= line,linewidth=1.2pt](1;2)(0.3;4)(1;6)
			\psdot[linecolor=diskin,dotsize=0.3](0.5;5.5)
			\psline[linecolor= line,linewidth=1.2pt](1;1.5)(1;5.5)
		\end{pspicture}
		\subcaption*{R3}
			 \end{subfigure}
\caption{Reidemeister moves.}
\label{Reidemeister}
\end{figure}
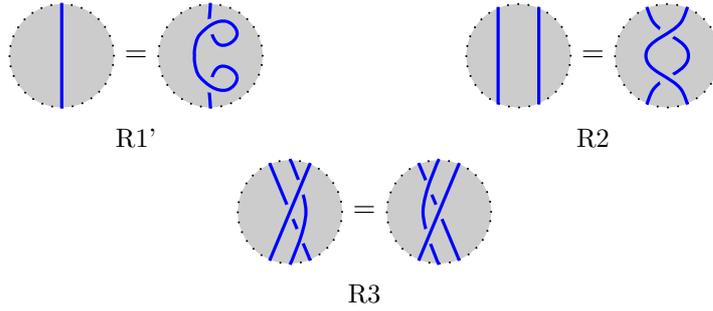
Note that the Reidemeister move R1 is only satisfied up to a scalar multiple,
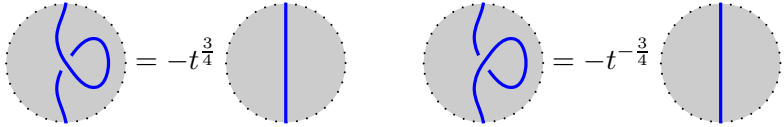
\begin{equation*}
	\begin{pspicture}[shift=-0.9](-1,-1)(1,1)
			\pscircle[fillstyle=solid,fillcolor=diskin,linewidth=0.8pt,linestyle=dotted, ](0,0){1}
			\degrees[8]
			\psecurve[linecolor= line,linewidth=1.2pt](0,-1.1)(0,-1)(-0.15,-0.4)(0,0)(0.5,0.4)(0.7,0)(0.5,-0.4)
			\psdot[linecolor=diskin,dotsize=0.3](0,0)
			\psecurve[linecolor= line,linewidth=1.2pt](0,1.1)(0,1)(-0.15,0.4)(0,0)(0.5,-0.4)(0.7,0)(0.5,0.4)
			\end{pspicture}
			= -t^{\frac{3}{4}}\;
			 \begin{pspicture}[shift=-0.9](-1,-1)(1,1)
			\pscircle[fillstyle=solid,fillcolor=diskin,linewidth=0.8pt,linestyle=dotted, ](0,0){1}
			\degrees[8]
			\psline[linecolor= line,linewidth=1.2pt](1;2)(1;6)
		\end{pspicture}
		\hspace{1cm}
		\begin{pspicture}[shift=-0.9](-1,-1)(1,1)
			\pscircle[fillstyle=solid,fillcolor=diskin,linewidth=0.8pt,linestyle=dotted, ](0,0){1}
			\degrees[8]
			\psecurve[linecolor= line,linewidth=1.2pt](0,1.1)(0,1)(-0.15,0.4)(0,0)(0.5,-0.4)(0.7,0)(0.5,0.4)
			\psdot[linecolor=diskin,dotsize=0.3](0,0)
			\psecurve[linecolor= line,linewidth=1.2pt](0,-1.1)(0,-1)(-0.15,-0.4)(0,0)(0.5,0.4)(0.7,0)(0.5,-0.4)
			\end{pspicture}
			= -t^{-\frac{3}{4}} \;
			 \begin{pspicture}[shift=-0.9](-1,-1)(1,1)
			\pscircle[fillstyle=solid,fillcolor=diskin,linewidth=0.8pt,linestyle=dotted, ](0,0){1}
			\degrees[8]
			\psline[linecolor= line,linewidth=1.2pt](1;2)(1;6)
		\end{pspicture} 
\end{equation*}

\begin{remark}\label{RKBSM}
The morphism space $\textup{Hom}_{\mathcal{S}}(m,n)$ can be identified with a relative Kauffman bracket skein module on the thickened cylinder $\Sigma$ with (framed) marked points $(2\xi_m^{i-1},1)$ ($1\leq i\leq m$) and $(2\xi_n^{j-1},0)$ ($1\leq j\leq n$), cf. \cite{Prz}. The identification goes through the projection map $\textup{proj}$. In this $3$-dimensional description of the hom-spaces the composition rule turns into the vertically stacking of the thickened cylinders.
\end{remark}

We now show that the skein category $\mathcal{S}$ is a strict monoidal, linear category.\footnote{We thank an anonymous referee for this observation.} The tensor functor $\times_{\mathcal{S}}: \mathcal{S}\times\mathcal{S}\rightarrow\mathcal{S}$ on objects
$m,n\in\mathbb{Z}_{\geq 0}$ is given by $m\times_{\mathcal{S}} n:=m+n$. On morphisms the tensor product is defined through Przytycki's and Sikora's \cite[\S 3]{PrzSik} relative version of the skein algebra multiplication on the associated relative Kauffman bracket skein modules from Remark \ref{RKBSM}.
On the level of tangles $T, T^\prime$ on the thickened cylinder $\Sigma$, the Kauffman bracket skein product $T\cdot T^\prime$ amounts to placing $T^\prime$ inside the solid cylindrical hole of the thickened cylinder of $T$ and moving the endpoints 
of the arcs to the marked points on $C_o\times\{1\}$ and $C_o\times\{0\}$ in a specific way. The exact rule regarding the repositioning of the endpoints is determined as follows. 

Before putting $T^\prime$ inside $T$, fix the parametrizations $\gamma_\ell(s):=(2\exp(2\pi is),\ell)$ ($s\in [0,1]$) of $C_o\times\{\ell\}\subset A\times\{\ell\}$ ($\ell=0,1$). Place the endpoints of the two tangles $T$ and $T^\prime$ on the line segments $(1,2]\times \{\ell\}\subset A\times \{\ell\}$ ($\ell=0,1$)
using an isotopy of $\Sigma$ which, for $\ell\in\{0,1\}$, stabilizes $A\times \{\ell\}$, fixes the endpoint
$(2,\ell)\in
\partial\Sigma$ and pushes, for $\epsilon>0$ sufficiently small, the boundary arc $\gamma_\ell([0,1-\epsilon])$ into the line segment $(1,2]\times \{\ell\}$. The skein algebra multiplication rule then produces a new tangle with endpoints on the two line segments $(1,2]\times\{\ell\}$ ($\ell=0,1$), which is converted back to a tangle with endpoints on the marked points on $C_o\times\{\ell\}$ ($\ell=0,1$) by applying a reverse isotopy of the type as described above
(see \cite[\S 3]{PrzSik} for further details).

Through the projection map $\textup{proj}$ the relative skein algebra multiplication rule as described in the previous paragraph gives bilinear operations
\[
\textup{Hom}_{\mathcal{S}}(k,\ell)\times\textup{Hom}_{\mathcal{S}}(m,n)
\overset{\times_{\mathcal{S}}}{\longrightarrow}\textup{Hom}_{\mathcal{S}}(k+m,\ell+n),
\quad ([D], [D^\prime])\mapsto [D]\times_{\mathcal{S}} [D^\prime]
\]
for $k,\ell,m,n\in\mathbb{Z}_{\geq 0}$. They are explicitly described as follows.
Let $D$ be a $(k,\ell)$-tangle diagram on $A$ and $D^\prime$ an $(m,n)$-tangle diagram on $A$. 
Then $[D]\times_{\mathcal{S}} [D^\prime]=[D\ast D^\prime]$ with $D\ast D^\prime$ the following
$(k+m,\ell+n)$-tangle diagram. 

Let $D_{\curvearrowright}$ be a diagram on $A$ obtained from $D$ by applying 
a planar isotopy of $A$ which
\begin{enumerate}
\item[{\bf 1.}] rotates the endpoints $\xi_k^{i-1}\in C_i$ clockwise to $\xi_{k+m}^{i-1}$
($1\leq i\leq k$),
\item[{\bf 2.}] rotates the endpoints $2\xi_\ell^{i-1}\in C_o$ clockwise to 
$\xi_{\ell+n}^{i-1}$ ($1\leq i\leq \ell$),
\item[{\bf 3.}] fixes some straight line segment between the inner and outer boundary of $A$.
\end{enumerate}
Similarly, let $D^\prime_{\curvearrowleft}$ be the diagram on $A$ obtained 
from $D^\prime$ by applying a planar isotopy of $A$ which
\begin{enumerate}
\item[{\bf 1.}] rotates the endpoints $\xi_m^{i-1}\in C_i$ counterclockwise to 
$\xi_{k+m}^{k+i-1}$  ($1\leq i\leq m$),
\item[{\bf 2.}] rotates the endpoints $2\xi_n^{i-1}\in C_o$ counterclockwise to 
$2\xi_{\ell+n}^{n+i-1}$ ($1\leq i\leq n$),
\item[{\bf 3.}] fixes some straight line segment between the inner and outer boundary of $A$.
\end{enumerate} 
Then $D\ast D^\prime$ is the $(k+m,\ell+n)$-tangle diagram obtained by placing 
$D_{\curvearrowright}$ on top of $D^\prime_{\curvearrowleft}$.

In the following picture we give an example of the $\ast$-product of two tangle diagrams on
$A$. We use a different colour for the (2,2)-tangle diagram to assist comprehension.
$$
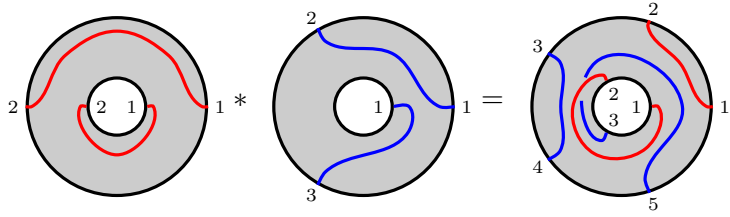
\begin{pspicture}[shift=-1.8](-1.8,-1.8)(1.8,1.8)
    	\SpecialCoor
    	\pscircle[fillstyle=solid, fillcolor=diskin,linewidth=1.2pt](0,0){1.5}
    	\pscircle[fillstyle=solid, fillcolor=white, linewidth=1.2pt](0,0){0.5}
    	\degrees[8]
     	\rput{0}(1.7;0){\tiny $1$}
    	\rput{0}(1.7;4){\tiny $2$}
	\rput{0}(0.25;0){\tiny $1$}
    	\rput{0}(0.25;4){\tiny $2$}
	\pscurve[linecolor=red,linewidth=1.2pt](1.5;0)(1.4;0)(1.2;1)(1.25;2)(1.2;3)(1.4;4)(1.5;4)
	\pscurve[linecolor=red,linewidth=1.2pt](0.5;0)(0.6;0)(0.8;6)(0.6;4)(0.5;4)
   \end{pspicture}
	\ast
\begin{pspicture}[shift=-1.8](-1.8,-1.8)(1.8,1.8)
    	\SpecialCoor
    	\pscircle[fillstyle=solid, fillcolor=diskin,linewidth=1.2pt](0,0){1.5}
    	\pscircle[fillstyle=solid, fillcolor=white, linewidth=1.2pt](0,0){0.5}
    	\degrees[3]
     	\rput{0}(1.7;0){\tiny $1$}
    	\rput{0}(1.7;1){\tiny $2$}
    	\rput{0}(1.7;2){\tiny $3$}
	\rput{0}(0.25;0){\tiny $1$}
	\pscurve[linecolor=line,linewidth=1.2pt](1.5;2)(1.25;2)(0.85;2.75)(0.75;0)(0.5;0)
	\pscurve[linecolor=line,linewidth=1.2pt](1.5;0)(1.25;0)(1;0.5)(1.25;1)(1.5;1)
   \end{pspicture} 
=
\begin{pspicture}[shift=-1.8](-1.8,-1.8)(1.8,1.8)
    	\SpecialCoor
    	\pscircle[fillstyle=solid, fillcolor=diskin,linewidth=1.2pt](0,0){1.5}
    	\pscircle[fillstyle=solid, fillcolor=white, linewidth=1.2pt](0,0){0.5}
    	\degrees[15]
     	\rput{0}(1.7;0){\tiny $1$}
    	\rput{0}(1.7;3){\tiny $2$}
	\rput{0}(1.7;6){\tiny $3$}
    	\rput{0}(1.7;9){\tiny $4$}
	\rput{0}(1.7;12){\tiny $5$}
	\rput{0}(0.25;0){\tiny $1$}
    	\rput{0}(0.25;5){\tiny $2$}
	\rput{0}(0.25;10){\tiny $3$}
	\pscurve[linecolor=line,linewidth=1.2pt](1.5;12)(1.25;12)(1;0)(0.85;3)(0.85;5)(0.65;8)(0.65;10)(0.5;10)
	\pscurve[linecolor=line,linewidth=1.2pt](1.5;6)(1.25;6)(1;7.5)(1.25;9)(1.5;9)
	\psdot[linecolor=diskin,dotsize=0.4](0.725;6.5)
	\pscurve[linecolor=red,linewidth=1.2pt](1.5;0)(1.4;0)(1.2;1.5)(1.4;3)(1.5;3)
	\pscurve[linecolor=red,linewidth=1.2pt](0.5;0)(0.6;0)(0.8;12.5)(0.9;10)(0.8;7.5)(0.6;5)(0.5;5)
   \end{pspicture}
$$

\begin{example}
The tensor product maps $\textup{End}_{\mathcal{S}}(0)\times\textup{Hom}_{\mathcal{S}}(m,n)
\rightarrow \textup{Hom}_{\mathcal{S}}(m,n)$ and $\textup{Hom}_{\mathcal{S}}(m,n)\times
\textup{End}_{\mathcal{S}}(0)\rightarrow \textup{Hom}_{\mathcal{S}}(m,n)$ correspond to 
placing knot diagrams on top or below tangle diagrams within $\textup{Hom}_{\mathcal{S}}(m,n)$. The resulting $\textup{End}_{\mathcal{S}}(0)$-bimodule structure on $\textup{Hom}_{\mathcal{S}}(m,n)$ has been described and studied in the more general context of relative Kauffman skein modules over surfaces, see, e.g., \cite{Prz, Prz2}. See also \cite[\S 4.1]{Le} for a discussion of $\textup{Hom}_{\mathcal{S}}(0,2)$ as $\textup{End}_{\mathcal{S}}(0)$-bimodule.
\end{example}

\begin{proposition}
The skein category $\mathcal{S}$ of the annulus is a strict monoidal linear category with tensor functor $\times_{\mathcal{S}}: \mathcal{S}\times\mathcal{S}\rightarrow\mathcal{S}$ as defined above, and 
unit object $0$.
\end{proposition}
\begin{proof}
By the remarks preceding the proposition, the only nontrivial check is the compatibility of $\times_{\mathcal{S}}$
with composition of morphisms. For the first tensor component this follows from the fact that
all the endpoints of $D$ in 
$D\ast D^\prime$ are rotated clockwise, while over-rotation by angles $\geq 2\pi$ cannot occur due to the third property of the planar isotopy transforming $D$ into
$D_{\curvearrowright}$. A similar remark applies for the second tensor component.
\end{proof}

We write $\otimes$ for the usual tensor product of complex vector spaces.
\begin{corollary}\label{algemaps}
For $m,n\in\mathbb{Z}_{\geq 0}$ we have algebra morphisms
\[
\epsilon_{m,n}: \textup{End}_{\mathcal{S}}(m)\otimes\textup{End}_{\mathcal{S}}(n)
\rightarrow \textup{End}_{\mathcal{S}}(m+n)
\]
defined by $\epsilon_{m,n}\bigl([D]\otimes [D^\prime]\bigr):=[D]\times_{\mathcal{S}}[D^\prime]=
[D\ast D^\prime]$.
\end{corollary}
As we shall see in Remark \ref{isoRem}, the algebra $\textup{End}_{\mathcal{S}}(m)$ is isomorphic to the $m$th extended affine Temperley-Lieb algebra. Under this identification, the algebra maps $\epsilon_{m,n}$ were considered before in \cite[\S 3.3]{GS}.

\section{Equivalence with the affine Temperley-Lieb category}\label{section4}

The affine Temperley-Lieb category was introduced by Graham and Lehrer \cite{GL}. 
In this category the morphisms are affine diagrams, which are defined as follows.

\begin{definition}
Let $m,n\in\mathbb{Z}_{\geq 0}$. An affine $(m,n)$-diagram is an $(m,n)$-tangle diagram 
in $A$ with no crossings and without contractible loops in $A$. 
We write $\mathcal{D}_{m,n}$ for the subclass of $\textup{Hom}_{\mathcal{T}}(m,n)$
consisting of equivalence classes $\overline{D}$ of affine $(m,n)$-diagrams $D$.
 \end{definition}
\begin{remark}
An affine diagram on the annulus can be viewed as a periodic diagram on the infinite horizontal strip
by cutting the annulus open along a line segment connecting the inner and outer boundary of $A$ and extending the resulting diagram periodically.
This is how affine diagrams were originally considered in \cite{GL, Gre}.
\end{remark}

Let $\textup{Lin}_c(\mathcal{T})$ be the quotient of the linear category $\textup{Lin}(\mathcal{T})$ by
the loop removal relation \eqref{loopremoval} (compare with Definition \ref{congruence}). The sublabel ``$c$" stands for contractible, signifying that in $\textup{Lin}_c(\mathcal{T})$ contractible loops in tangle diagrams may be removed by the multiplicative factor $-(t^{\frac{1}{2}}+t^{-\frac{1}{2}})$.
If $D$ is an $(m,n)$-tangle diagram then we write $\langle D\rangle$ for its equivalence class in 
$\textup{Hom}_{\textup{Lin}_c(\mathcal{T})}(m,n)$.

Note that the skein category $\mathcal{S}$ is the quotient of $\textup{Lin}_c(\mathcal{T})$ by
the Kauffman skein relation \eqref{kauffman}. Graham's and Lehrer's \cite{GL} affine Temperley-Lieb category, which is closely related to Jones' \cite{Jo} annular Temperley-Lieb category,
is the following subcategory of $\textup{Lin}_c(\mathcal{T})$.

\begin{definition}[\cite{GL}]
The affine Temperley-Lieb category $\mathcal{TL}=\mathcal{TL}(t^{\frac{1}{2}})$ is the linear subcategory of
$\textup{Lin}_c(\mathcal{T})$ with objects $\mathbb{Z}_{\geq 0}$ and morphisms
$\textup{Hom}_{\mathcal{TL}}(m,n)$
the subspace of
$\textup{Hom}_{\textup{Lin}_c(\mathcal{T})}(m,n)$
spanned by the equivalence classes $\langle D\rangle$ of affine $(m,n)$-diagrams $D$.
\end{definition}

If $D$ is an affine $(m,n)$-diagram and $D^\prime$ is an affine $(k,m)$-diagram then
 \[
\langle D\rangle \circ \langle D^\prime\rangle=\bigl(-t^{\frac{1}{2}}-t^{-\frac{1}{2}}\bigr)^{l(D^{\prime\prime})}
\langle D^{\prime\prime}_c\rangle
 \]
 in $\textup{Hom}_{\mathcal{TL}}(k,n)$, with $D^{\prime\prime}$ the $(k,n)$-tangle diagram in $A$ obtained by inserting $D^\prime$ inside $D$ (in the same way as in Definition \ref{tanglecat}), with $l(D^{\prime\prime})$
 the number of loops in $D^{\prime\prime}$ contractible in $A$,
 and with $D^{\prime\prime}_c\in\mathcal{D}_{k,n}$ the affine $(k,n)$-diagram obtained from 
 $D^{\prime\prime}$ by removing the contractible loops.

Note that 
$\textup{Hom}_{\mathcal{TL}}(m,n)=\{0\}$ if $m+n$ is odd, and 
\[
\{ \langle D\rangle \,\, | \,\, D \,\,\,
\textup{affine } (m,n)\textup{-diagram}\}
\]
is a linear basis of $\textup{Hom}_{\mathcal{TL}}(m,n)$.

Next we show that the linear categories $\cS$ and
$\mathcal{TL}$ are equivalent. The subtle point is to show that the
obvious linear functor from $\mathcal{TL}$ to $\cS$ is faithful. The proof uses a relative version of the Kauffman bracket for $(m,n)$-tangle diagrams in $A$, compare with the proof of \cite[Thm. 3.1]{Prz}.
\begin{theorem}\label{equivThm}
The linear categories $\cS$ and $\mathcal{TL}$ are equivalent.
\end{theorem}
\begin{proof}
Consider the essentially surjective linear functor $\mathcal{F}: \mathcal{TL}\rightarrow \cS$
which is the identity on objects and maps $\langle D\rangle$ to $[D]$ for an affine $(m,n)$-diagram $D$.
It is clearly well defined since the loop removal relation
holds in $\cS$ as well as in $\mathcal{TL}$. 

Let $D$ be an $(m,n)$-tangle diagram in $A$. The Kauffman skein relation and the loop removal relation allow us to write $[D]$ as a linear combination of classes 
$[D^\prime]\in\textup{Hom}_{\mathcal{S}}(m,n)$
with the $D^\prime$'s being affine $(m,n)$-diagrams. It follows that the functor $\mathcal{F}$ is full.
It remains to show that $\mathcal{F}$ is faithful.

Suppose that $m+n$ is even and let $D$ be an $(m,n)$-tangle diagram in $A$ with $k$ crossing points. Let 
$\mathcal{S}_D$ be the set of cardinality $2^k$ containing the $(m,n)$-tangle diagrams $S$ without crossings that are obtained from $D$ by removing each crossing \cross
in $D$ by either \monoid or \id.  For $S\in\mathcal{S}_D$ let $h_D(S)$ (respectively $v_D(S)$) be the number of crossing
points at which \cross is replaced by \monoid (respectively \id).
Let $c_D(S)$ be the
number of loops in $S$ that are contractible in $A$, and write $\widetilde{S}$ for the affine 
$(m,n)$-diagram obtained from $S$ by removing these contractible loops.

It is easy to see that there exists a well defined linear map
\[
\widehat{\psi}: \mathbb{C}[\textup{Hom}_{\mathcal{T}}(m,n)]\rightarrow 
\textup{Hom}_{\mathcal{TL}}(m,n)
\]
satisfying
\begin{equation}\label{relKauffman}
\widehat{\psi}(\overline{D}):=\sum_{S\in\mathcal{S}_D}\bigl(-t^{\frac{1}{2}}-t^{-\frac{1}{2}}\bigr)^{c_D(S)}
t^{(h_D(S)-v_D(S))/4}\langle\widetilde{S}\rangle
\end{equation}
for all $(m,n)$-tangle diagrams $D$ in $A$. Direct computations show that 
the map $\widehat{\psi}$ respects the Kauffman skein relation \eqref{kauffman}
and the loop removal relation \eqref{loopremoval}, so it gives rise to a linear map 
\[
\psi: \textup{Hom}_{\cS}(m,n)\rightarrow\textup{Hom}_{\mathcal{TL}}(m,n)
\]
satisfying $\psi([D])=\widehat{\psi}(\overline{D})$ for $(m,n)$-tangle diagrams $D$ in $A$.

By the Kaufmann skein relation \eqref{kauffman} and the loop removal relation
\eqref{loopremoval}, the linear map $\psi$ is the inverse of the linear map $\mathcal{F}: \textup{Hom}_{\mathcal{TL}}(m,n)\rightarrow 
\textup{Hom}_{\mathcal{S}}(m,n)$. This shows that $\mathcal{F}$ is faithful.
\end{proof}
\begin{remark}
The special case $\textup{End}_{\mathcal{S}}(0)\simeq\textup{End}_{\mathcal{TL}}(0)$
was established for general surfaces in \cite[Lem. 3.3]{Prz}.
\end{remark}
\begin{definition}
We call $\psi([D])=\widehat{\psi}(\overline{D})\in\textup{Hom}_{\mathcal{TL}}(m,n)$ (see \eqref{relKauffman}) the {\it relative Kauffman bracket} of the $(m,n)$-tangle diagram $D$ in $A$.
\end{definition}
\begin{remark}
Note that for $(m,n)=(0,0)$, the relative Kauffman bracket $\psi([D])$ of a link diagram $D$ in $A$
lands in the algebra 
$\textup{End}_{\mathcal{TL}}(0)$, which is isomorphic to the algebra of polynomials in one variable
(the variable corresponds to the equivalence class of a noncontractible loop in $A$). Evaluating
the resulting polynomial at $-t^{\frac{1}{2}}-t^{-\frac{1}{2}}$ can be thought of as closing the hole of the annulus and viewing the link diagram as an element in the skein module of the disc (or equivalently, of the plane). As a result one obtains the usual Kauffman \cite{Kau} bracket of $D$, viewed as a link diagram in the plane (see \cite{Lic} and \cite[\S 1.7]{Mor}).
\end{remark}

\section{The extended affine Temperley-Lieb algebra}\label{section5}

Write $\textup{TL}_0:=\mathbb{C}[X]$ for the algebra of complex polynomials in one variable $X$ and $\textup{TL}_1:=\mathbb{C}[\rho,\rho^{-1}]$ for the algebra of complex Laurent polynomials in 
the variable $\rho$. Let $\textup{TL}_2$ be the complex associative unital algebra with generators
$e_1,e_2,\rho,\rho^{-1}$ and defining relations
	\begin{align*}
	 &e_i^2 = \bigl(-t^{\frac{1}{2}} - t^{-\frac{1}{2}}\bigr) e_i,   \\
		 	&\rho e_i =  e_{i+1} \rho, \\
		 & \rho\rho^{-1}=1=\rho^{-1}\rho,\\
		 & \rho^2 e_1 = e_1,
	\end{align*}
where the indices are taken modulo two. Finally, for $n\geq 3$ let $\textup{TL}_n$ be the
complex associative unital algebra with generators $e_1,e_2,\ldots,e_{n},\rho,\rho^{-1}$ and
defining relations
\begin{equation}\label{relTL}
\begin{split}
	 &e_i^2 = \bigl(-t^{\frac{1}{2}} - t^{-\frac{1}{2}}\bigr) e_i,   \\
	 &e_ie_j = e_je_i\qquad\qquad\quad \hbox{ if }\,\, i-j\not=\pm 1,\\
	 &e_ie_{i\pm 1}e_i = e_i,\\
		 	&\rho e_i =  e_{i+1} \rho, \\
		 & \rho\rho^{-1}=1=\rho^{-1}\rho,\\
		 & \bigl(\rho e_1\bigr)^{n-1}  = \rho^n (\rho e_1), 		 
\end{split}
\end{equation}
where the indices are taken modulo $n$. Observe that the last defining relation $(\rho e_1)^{n-1}=\rho^n(\rho e_1)$
in \eqref{relTL} can be replaced by
\[
\rho^2e_{n-1}=e_1e_2\cdots e_{n-1}.
\]
Note that $\textup{TL}_n=\textup{TL}_n(t^{\frac{1}{2}})$ for $n\geq 2$ depends on the nonzero complex parameter $t^{\frac{1}{2}}$, which we omit from the notations if no confusion can arise. 
\begin{remark}\label{restrictedgenerators}
The definition for $n=2$ and $n\geq 3$ can be placed at the same footing by describing 
$\textup{TL}_n$ in terms of the smaller set $e_1,e_2,\ldots,e_{n-1},\rho,\rho^{-1}$ of algebraic generators. The defining relations then are 
\begin{equation}\label{relTLalt}
\begin{split}
	 &e_i^2 = \bigl(-t^{\frac{1}{2}} - t^{-\frac{1}{2}}\bigr) e_i, ,\qquad 1\leq i<n,\\
	 &e_ie_j = e_je_i\qquad\,\,\,\quad\quad\qquad\, 1\leq i,j<n \,\,\hbox{ and }
	 \,\, i-j\not=\pm 1,\\
	 &e_ie_{i\pm 1}e_i = e_i,\,\,\quad\qquad\qquad\, 1\leq i,i\pm 1<n,\\
		 	&\rho e_i =  e_{i+1} \rho,\qquad\qquad\qquad 1\leq i<n-1,\\
			&\rho^2e_{n-1}=e_1\rho^2,\\
		 & \rho\rho^{-1}=1=\rho^{-1}\rho,\\
		 & \rho^2e_{n-1}=e_1e_2\cdots e_{n-1}.		 
\end{split}
\end{equation}
 
\end{remark}

\begin{definition}[\cite{Gre}]
$\textup{TL}_n$ is called the {\it ($n$th) extended affine Temperley-Lieb
algebra}. 
\end{definition}
Denote $\mathcal{TL}_n:=\textup{End}_{\mathcal{TL}}(n)$ for the algebra of endomorphisms
of $n$ in the affine Tem\-per\-ley-Lieb category $\mathcal{TL}$. The following result is essentially due to Green \cite{Gre}.

\begin{theorem}\label{presentationThm}
{\bf a.} $\textup{TL}_0\simeq \mathcal{TL}_0$ with the algebra isomorphism
$\textup{TL}_0\rightarrow\mathcal{TL}_0$ defined by
  $$X \mapsto  
	\begin{pspicture}[shift=-0.8](-1.5,-1)(1.5,1)
	\SpecialCoor
	\definecolor{line}{rgb}{0,0,1}
	\pscircle[fillstyle=ccslope,slopebegin=diskout,slopeend=diskin,linewidth=1.2pt](0,0){1}
	\pscircle[fillstyle=solid, fillcolor=white, linewidth=1.2pt](0,0){0.25}
	\pscircle[linecolor=line,linewidth=1.2pt](0,0){0.7}
	\end{pspicture}$$

{\bf b.} $\textup{TL}_1\simeq\mathcal{TL}_1$ with the algebra isomorphism
$\textup{TL}_1\rightarrow\mathcal{TL}_1$ defined by
\begin{equation}
\rho\mapsto
\begin{pspicture}[shift=-0.8](-1.5,-1)(1.5,1)
	\SpecialCoor
	\definecolor{line}{rgb}{0,0,1}
	\pscircle[fillstyle=ccslope,slopebegin=diskin,slopeend=diskout,linewidth=1.2pt](0,0){1}
	\pscircle[fillstyle=solid, fillcolor=white, linewidth=1.2pt](0,0){0.25}
	\degrees[8]
	\pscurve[linecolor=line,linewidth=1.2pt](0.25;0)(0.35;0)(0.5;1)(0.6;2)(0.6;3)(0.6;4)(0.6;5)(0.6;6)(0.6;7)(0.8;8)(1;8)
	\rput{0}(1.2;0){\tiny $1$}
\end{pspicture}\label{Dehntwistrho}
\end{equation}

{\bf c.} $\textup{TL}_2\simeq\mathcal{TL}_2$ with the algebra isomorphism $\textup{TL}_2\rightarrow
\mathcal{TL}_2$ defined by

\begin{equation}\label{map2}
	\rho \mapsto 
	\begin{pspicture}[shift=-0.8](-1.5,-1)(1.5,1)
                     	\SpecialCoor
                    	\definecolor{line}{rgb}{0,0,1}
                    	\pscircle[fillstyle=ccslope,slopebegin=diskin,slopeend=diskout,linewidth=1.2pt](0,0){1}
                    	\pscircle[fillstyle=solid, fillcolor=white, linewidth=1.2pt](0,0){0.25}
                    	\degrees[8]
                    	\rput{0}(1.2;0){\tiny $1$}
                    	\rput{0}(1.2;4){\tiny $2$}
			\psplot[algebraic, plotpoints=400,linecolor=line,polarplot=true,linewidth=1.2pt]{0}{Pi }{(2/Pi)*ASIN(2*x/Pi-1) +(5*Pi-5*x)/(4*Pi)}
			\rput{4}(0,0){\psplot[algebraic, plotpoints=400,linecolor=line,polarplot=true,linewidth=1.2pt]{0}{Pi }{(2/Pi)*ASIN(2*x/Pi-1) +(5*Pi-5*x)/(4*Pi)}}
  		\end{pspicture}	
,
 \qquad
\qquad
	e_1 \mapsto 
	\begin{pspicture}[shift=-0.8](-1.5,-1)(1.5,1)
	\SpecialCoor
	\definecolor{line}{rgb}{0,0,1}
	\pscircle[fillstyle=ccslope,slopebegin=diskin,slopeend=diskout,linewidth=1.2pt](0,0){1}
	\pscircle[fillstyle=solid, fillcolor=white, linewidth=1.2pt](0,0){0.25}
	\degrees[8]
	\pscurve[linecolor=line,linewidth=1.2pt](0.25;0)(0.35;0)(0.4;2)(0.35;4)(0.25;4)
	\pscurve[linecolor=line,linewidth=1.2pt](1;0)(0.9;0)(0.7;1.5)(0.7;2.5)(0.9;4)(1;4)
	\rput{0}(1.2;0){\tiny $\,1$}
	\rput{0}(1.2;4){\tiny $2\,$}
\end{pspicture},
\qquad
	e_2 \mapsto  \begin{pspicture}[shift=-0.8](-1.5,-1)(1.5,1)
	\SpecialCoor
	\definecolor{line}{rgb}{0,0,1}
	\pscircle[fillstyle=ccslope,slopebegin=diskin,slopeend=diskout,linewidth=1.2pt](0,0){1}
	\pscircle[fillstyle=solid, fillcolor=white, linewidth=1.2pt](0,0){0.25}
	\degrees[8]
	\pscurve[linecolor=line,linewidth=1.2pt](0.25;0)(0.35;0)(0.4;6)(0.35;4)(0.25;4)
	\pscurve[linecolor=line,linewidth=1.2pt](1;0)(0.9;0)(0.7;6.5)(0.7;5.5)(0.9;4)(1;4)
	\rput{0}(1.2;0){\tiny $\,1$}
	\rput{0}(1.2;4){\tiny $2\,$}
\end{pspicture}
\end{equation}
{\bf d.} If $n\geq 3$ then $\textup{TL}_n\simeq\mathcal{TL}_n$ with the algebra isomorphism 
$\textup{TL}_n\rightarrow\mathcal{TL}_n$ defined by
\begin{equation}\label{mapn}
			\rho \mapsto    \begin{pspicture}[shift=-1.4](-1.5,-1.5)(1.5,1.5)
    	\pscircle[fillstyle=ccslope,slopebegin=diskin,slopeend=diskout,linewidth=1.2pt](0,0){1.25}
    	\pscircle[fillstyle=solid, fillcolor=white, linewidth=1.2pt](0,0){0.5}
    	\degrees[16]
     	\rput{0}(1.5;0){\tiny $1$}
	\rput{0}(1.5;1){\tiny $2$}
     	\rput{0}(0.25;0.5){\tiny $1$}
	\rput{0}(0.25;14){\tiny $n$}
	\pscurve[linecolor=line,linewidth=1.2pt](1.25;0)(1.15;0)(0.85;15.5)(0.6;15)(0.5;15)
	\pscurve[linecolor=line,linewidth=1.2pt](1.25;1)(1.15;1)(0.85;0.5)(0.6;0)(0.5;0)
	\psarc[linestyle=dotted, linecolor=line,linewidth=1.2pt](0,0){0.85}{2}{14}
   \end{pspicture}\,\,\,\,\,,   
   \qquad\qquad
			e_i \mapsto    \begin{pspicture}[shift=-1.65](-1.75,-1.75)(1.5,1.5)
       	\SpecialCoor
    	\pscircle[fillstyle=ccslope,slopebegin=diskin,slopeend=diskout,linewidth=1.2pt](0,0){1.25}
    	\pscircle[fillstyle=solid, fillcolor=white, linewidth=1.2pt](0,0){0.5}
    	\degrees[17]
     	\rput{0}(1.5;0){\tiny $1$}
	\rput{0}(1.7;9.8){\tiny $i\!-\!1$}
    	\rput{0}(1.6;10.9){\tiny $i$}
	\rput{0}(1.6;11.9){\tiny $i\!+\!1$}
	\rput{0}(1.6;13.3){\tiny $i\!+\!2$}
     	\rput{0}(0.25;0){\tiny $1$}
    	\rput{0}(0.25;10.8){\tiny $i$}
 	\psline[linecolor=line,linewidth=1.2pt](0.5;0)(1.25;0)
	\psline[linecolor=line,linewidth=1.2pt](0.5;10)(1.25;10)
	\psline[linecolor=line,linewidth=1.2pt](0.5;13)(1.25;13)
	\pscurve[linecolor=line,linewidth=1.2pt](0.5;11)(0.6;11)(0.75;11.5)(0.6;12)(0.5;12)
	\pscurve[linecolor=line,linewidth=1.2pt](1.25;11)(1.15;11)(0.9;11.5)(1.15;12)(1.25;12)
	\psarc[linestyle=dotted, linecolor=line,linewidth=1.2pt](0,0){0.85}{1}{9}
	\psarc[linestyle=dotted, linecolor=line,linewidth=1.2pt](0,0){0.85}{14}{16}
   \end{pspicture} 
\end{equation}
for $i=1,\ldots,n$ (with the indices and the labels of the marked points taken modulo $n$).
\end{theorem}
\begin{proof}
{\bf a} and {\bf b} are well known (see, for instance, \cite[\S 1.7]{Mor} and \cite[\S 4]{Mor2}), while
{\bf d} is due to Green \cite[Prop. 2.3.7]{Gre}.\\
{\it Proof of} {\bf c}: A direct check shows that there exists a unique unital algebra homomorphism 
$\phi: \textup{TL}_2\rightarrow \mathcal{TL}_2$ satisfying \eqref{map2}. 

Recall that the set $\mathcal{D}_2$ of affine $(2,2)$-diagrams form a linear basis of $\mathcal{TL}_2$. The affine $(2,2)$-diagrams can be described explicitly as follows.

The affine $(2,2)$-diagram $\phi(\rho^m)$ for $m\in\mathbb{Z}$
is obtained from the identity element of $\mathcal{TL}_2$ by winding the outer boundary counterclockwise by an angle of $m\pi$. It follows that the pairwise distinct affine $(2,2)$-diagrams 
$\phi(\rho^m)$ ($m\in\mathbb{Z}$) form the subset of $\mathcal{D}_2$ consisting of
affine $(2,2)$-diagrams whose arcs all connect the inner boundary with the outer boundary.
The remaining affine $(2,2)$-diagrams are the diagrams of the form
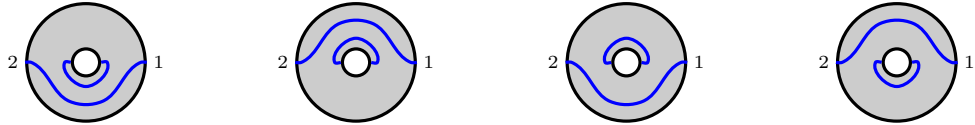
\begin{equation*}
  \begin{pspicture}[shift=-0.8](-1.5,-1.25)(1.5,1.25)
	\SpecialCoor
	\definecolor{line}{rgb}{0,0,1}
	\pscircle[fillstyle=ccslope,slopebegin=diskin,slopeend=diskout,linewidth=1.2pt](0,0){1}
	\pscircle[fillstyle=solid, fillcolor=white, linewidth=1.2pt](0,0){0.25}
	\degrees[8]
	\pscurve[linecolor=line,linewidth=1.2pt](0.25;0)(0.35;0)(0.4;6)(0.35;4)(0.25;4)
	\pscurve[linecolor=line,linewidth=1.2pt](1;0)(0.9;0)(0.7;6.5)(0.7;5.5)(0.9;4)(1;4)
	\rput{0}(1.2;0){\tiny $1$}
	\rput{0}(1.2;4){\tiny $2$}
\end{pspicture}
\qquad\quad
 \begin{pspicture}[shift=-0.8](-1.5,-1.25)(1.5,1.25)
	\SpecialCoor
	\definecolor{line}{rgb}{0,0,1}
	\pscircle[fillstyle=ccslope,slopebegin=diskin,slopeend=diskout,linewidth=1.2pt](0,0){1}
	\pscircle[fillstyle=solid, fillcolor=white, linewidth=1.2pt](0,0){0.25}
	\degrees[8]
	\pscurve[linecolor=line,linewidth=1.2pt](0.25;0)(0.35;0)(0.4;2)(0.35;4)(0.25;4)
	\pscurve[linecolor=line,linewidth=1.2pt](1;0)(0.9;0)(0.7;1.5)(0.7;2.5)(0.9;4)(1;4)
	\rput{0}(1.2;0){\tiny $1$}
	\rput{0}(1.2;4){\tiny $2$}
\end{pspicture}
\qquad\quad
 \begin{pspicture}[shift=-0.8](-1.5,-1.25)(1.5,1.25)
	\SpecialCoor
	\definecolor{line}{rgb}{0,0,1}
	\pscircle[fillstyle=ccslope,slopebegin=diskin,slopeend=diskout,linewidth=1.2pt](0,0){1}
	\pscircle[fillstyle=solid, fillcolor=white, linewidth=1.2pt](0,0){0.25}
	\degrees[8]
	\pscurve[linecolor=line,linewidth=1.2pt](0.25;0)(0.35;0)(0.4;2)(0.35;4)(0.25;4)
	\pscurve[linecolor=line,linewidth=1.2pt](1;0)(0.9;0)(0.7;6.5)(0.7;5.5)(0.9;4)(1;4)
	\rput{0}(1.2;0){\tiny $1$}
	\rput{0}(1.2;4){\tiny $2$}
\end{pspicture}  
\qquad\quad
\begin{pspicture}[shift=-0.8](-1.5,-1.25)(1.5,1.25)
	\SpecialCoor
	\definecolor{line}{rgb}{0,0,1}
	\pscircle[fillstyle=ccslope,slopebegin=diskin,slopeend=diskout,linewidth=1.2pt](0,0){1}
	\pscircle[fillstyle=solid, fillcolor=white, linewidth=1.2pt](0,0){0.25}
	\degrees[8]
	\pscurve[linecolor=line,linewidth=1.2pt](0.25;0)(0.35;0)(0.4;6)(0.35;4)(0.25;4)
	\pscurve[linecolor=line,linewidth=1.2pt](1;0)(0.9;0)(0.7;1.5)(0.7;2.5)(0.9;4)(1;4)
	\rput{0}(1.2;0){\tiny $1$}
	\rput{0}(1.2;4){\tiny $2$}
\end{pspicture}
\end{equation*}
in which $r$ nonintersecting, noncontractible loops 
are inserted for some $r\in\mathbb{Z}_{\geq 0}$.
For $r=2k$ the resulting four types of affine $(2,2)$-diagrams are 
\[
\phi(e_2(e_1e_2)^k),\quad \phi(e_1(e_2e_1)^k),\quad \phi(\rho e_1(e_2e_1)^k), \quad \phi(\rho e_2(e_1e_2)^k).
\]
For $r=2k+1$ they are 
\[
\phi(\rho (e_1e_2)^k),\quad \phi(\rho (e_2e_1)^k), \quad
\phi((e_2e_1)^k),\quad \phi((e_1e_2)^k).
\]
Hence $\phi$ maps the subset
\begin{equation}\label{setTL}
\begin{split}
\{\rho^m\}_{m\in\mathbb{Z}}&\cup 
\{(e_2e_1)^k, \rho (e_2e_1)^k, e_1(e_2e_1)^k, \rho e_1(e_2e_1)^k\}_{k\in\mathbb{Z}_{\geq 0}}\\
&\cup  \{(e_1e_2)^k, \rho(e_1e_2)^k, e_2(e_1e_2)^k, \rho e_2(e_1e_2)^k\}_{k\in\mathbb{Z}_{\geq 0}}
\end{split}
\end{equation}
of $\textup{TL}_2$ bijectively onto the linear basis $\mathcal{D}_2$
of $\mathcal{TL}_2$.
By the defining relations in $\textup{TL}_2$ we see that \eqref{setTL} spans $\textup{TL}_2$.
We conclude that  
$\phi$ is an isomorphism of algebras.
\end{proof}

\begin{remark}\label{isoRem}
By Theorem \ref{equivThm} we now also have a skein-theoretic description 
$\textup{End}_{\mathcal{S}}(n)$ of the $n$th extended affine Temperley-Lieb algebra,
\begin{equation}\label{isoalg}
\textup{TL}_n\simeq\mathcal{TL}_n\simeq\textup{End}_{\cS}(n).
\end{equation}
The skein-theoretic description of the finite Temperley-Lieb algebra is described in
\cite{Kau, Kau2,Lic,Mor}.\\
\end{remark}

\section{The arc insertion functor}\label{section6}

\begin{definition}
The arc insertion functor $\mathcal{I}: \mathcal{S}\rightarrow\mathcal{S}$ is the endofunctor
$\mathcal{I}:=-\times_{\mathcal{S}}1$, defined concretely by
\begin{equation*}
\begin{split}
\mathcal{I}(m)&:=m\times_{\mathcal{S}}1=m+1,\\
\mathcal{I}\bigl([D]\bigr)&:=[D]\times_{\mathcal{S}}\mathbf{1}_1=[D\ast\textup{Id}_1]
\end{split}
\end{equation*}
for $m\in\mathbb{Z}_{\geq 0}$ and for tangle diagrams $D$.
\end{definition}
Let $D$ be an $(m,n)$-tangle diagram and write $D^{ins}=D\ast\textup{Id}_1$, so that
$\mathcal{I}([D])=[D^{ins}]$.
The $(m+1,n+1)$-tangle diagram $D^{ins}$ is obtained from $D$ by inserting an arc in $D$
connecting the inner boundary of $A$ with its outer boundary and going underneath all arcs it meets. 
See Section \ref{section3} for the specific requirements on the location of the endpoints and on the winding of the inserted arc. We give two examples. 
\begin{example}
For the $(0,2)$-tangle diagram $D_1$ and the $(1,3)$-tangle diagram $D_2$
given by
$$ 
D_1:= 
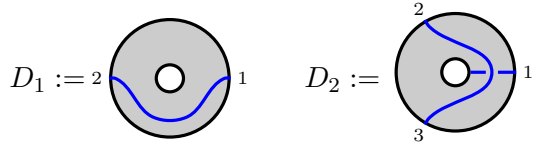
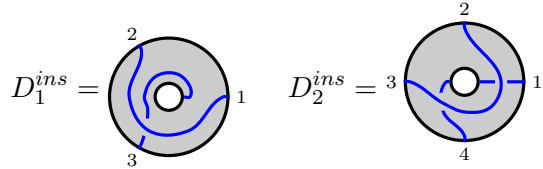
\begin{pspicture}[shift=-1.1](-1.25,-1.25)(1.25,1.25)
	\SpecialCoor
	\pscircle[fillstyle=ccslope,slopebegin=diskin,slopeend=diskout,linewidth=1.2pt](0,0){1}
	\pscircle[fillstyle=solid, fillcolor=white, linewidth=1.2pt](0,0){0.25}
	\degrees[8]
	\pscurve[linecolor=line,linewidth=1.2pt](1;0)(0.9;0)(0.7;6.5)(0.7;5.5)(0.9;4)(1;4)
	\rput{0}(1.2;0){\tiny $1$}
	\rput{0}(1.2;4){\tiny $2$}
\end{pspicture} 
\qquad
D_2:=
\begin{pspicture}[shift=-1](-1.1,-1.25)(1.25,1.25)
    	\SpecialCoor
    	\pscircle[fillstyle=ccslope,slopebegin=diskin,slopeend=diskout,linewidth=1.2pt](0,0){1}
    	\pscircle[fillstyle=solid, fillcolor=white, linewidth=1.2pt](0,0){0.25}
    	\degrees[3]
     	\rput{0}(1.2;0){\tiny $1$}
    	\rput{0}(1.2;1){\tiny $2$}
    	\rput{0}(1.2;2){\tiny $3$}
	\psline[linecolor=line,linewidth=1.2pt](1;0)(0.7;0)
	\psline[linecolor=line,linewidth=1.2pt] (0.25;0)(0.5;0)
	\pscurve[linecolor=line,linewidth=1.2pt](1;1)(0.9;1)(0.6;0)(0.9;2)(1;2)
   \end{pspicture} 
$$
we get
$$
D_1^{ins} =
    \begin{pspicture}[shift=-1](-1,-1)(1,1)
    	\SpecialCoor
	\definecolor{diskmid}{gray}{0.8}
    	\pscircle[fillstyle=ccslope,slopebegin=diskin,slopeend=diskout,linewidth=1.2pt](0,0){1}
    	\pscircle[fillstyle=solid, fillcolor=white, linewidth=1.2pt](0,0){0.25}
    	\degrees[3]
     	\rput{0}(1.2;0){\tiny $1$}
    	\rput{0}(1.2;1){\tiny $2$}
    	\rput{0}(1.2;2){\tiny $3$}
	\pscurve[linecolor=line,linewidth=1.2pt](1;2)(0.9;2)(0.4;1.75)(0.4;1.5)(0.4;1)(0.4;0.5)(0.35;0)(0.25;0)
	\psdot[linecolor=diskmid,dotsize=0.3](0.62; 1.96)
	\pscurve[linecolor=line,linewidth=1.2pt](1;0)(0.9;0)(0.65;2.5)(0.65;2)(0.65;1.5)(0.9;1)(1;1)
   \end{pspicture} 
   \qquad
   D_2^{ins} =
   \begin{pspicture}[shift=-1](-1.3,-1.25)(1.25,1.25)
      	\SpecialCoor
	\definecolor{diskmid}{gray}{0.8}
    	\pscircle[fillstyle=ccslope,slopebegin=diskin,slopeend=diskout,linewidth=1.2pt](0,0){1}
    	\pscircle[fillstyle=solid, fillcolor=white, linewidth=1.2pt](0,0){0.25}
    	\degrees[4]
     	\rput{0}(1.2;0){\tiny $1$}
    	\rput{0}(1.2;1){\tiny $2$}
    	\rput{0}(1.2;2){\tiny $3$}
	\rput{0}(1.2;3){\tiny $4$}
	\psline[linecolor=line,linewidth=1.2pt](1;0)(0.7;0)
	\psline[linecolor=line,linewidth=1.2pt] (0.25;0)(0.5;0)
	\pscurve[linecolor=line,linewidth=1.2pt](1;3)(0.9;3)(0.6;2.6)(0.3;2)(0.25;2)
	\psdot[linecolor=diskmid,dotsize=0.3](0.49;2.47)
	\pscurve[linecolor=line,linewidth=1.2pt](1;1)(0.9;1)(0.6;0)(0.6;3.5)(0.5;3)(0.5;2.5)(0.9;2)(1;2)	
   \end{pspicture} 
$$
\end{example}

For $n\in\mathbb{Z}_{\geq 0}$ consider the unit preserving algebra map
\[
\mathcal{I}_n:=\mathcal{I}\vert_{\cS_n}: \textup{End}_{\cS}(n)\rightarrow\textup{End}_{\cS}(n+1).
\] 
In terms
of the algebra maps $\epsilon_{m,n}$ (see Corollary \ref{algemaps}) we have
$\mathcal{I}_n([D])=\epsilon_{n,1}([D]\otimes \mathbf{1}_1)$. The map $\mathcal{I}_n$ can be interpreted as an
algebra map $\mathcal{I}_n: \textup{TL}_n\rightarrow\textup{TL}_{n+1}$ since
$\textup{TL}_n\simeq\textup{End}_{\cS}(n)$ (see Theorem \ref{presentationThm} and Remark \ref{isoRem}). In the following proposition
we explicitly compute $\mathcal{I}_n$ on the algebraic generators of $\textup{TL}_n$.
\begin{proposition}\label{insertionProp}
\begin{enumerate}
\item[{\bf a.}] $\mathcal{I}_0(X)=t^{\frac{1}{4}}\rho+t^{-\frac{1}{4}}\rho^{-1}$.
\item[{\bf b.}] $\mathcal{I}_1(\rho)=\rho(t^{-\frac{1}{4}}e_1+t^{\frac{1}{4}})$ 
and $\mathcal{I}_1(\rho^{-1})=
(t^{\frac{1}{4}}e_1+t^{-\frac{1}{4}})\rho^{-1}$.
\item[{\bf c.}] For $n\geq 2$ we have
\begin{equation*}
\begin{split}
&\mathcal{I}_n(e_i)=e_i,\qquad i=1,\ldots,n-1,\\
&\mathcal{I}_n(e_n)=(t^{\frac{1}{4}}e_n+t^{-\frac{1}{4}})e_{n+1}(t^{-\frac{1}{4}}e_n+t^{\frac{1}{4}}),\\
&\mathcal{I}_n(\rho)=\rho(t^{-\frac{1}{4}}e_n+t^{\frac{1}{4}}),\\
&\mathcal{I}_n(\rho^{-1})=(t^{\frac{1}{4}}e_n+t^{-\frac{1}{4}})\rho^{-1}.
\end{split}
\end{equation*}
\end{enumerate}
\end{proposition}
\begin{proof} These are direct computations in the skein module.\\ 
{\it Proof of} {\bf a}: We have $X = \begin{pspicture}[shift=-0.8](-1.5,-1)(1.5,1.2)
	\SpecialCoor
	\definecolor{line}{rgb}{0,0,1}
	\pscircle[fillstyle=ccslope,slopebegin=diskin,slopeend=diskout,linewidth=1.2pt](0,0){1}
	\pscircle[fillstyle=solid, fillcolor=white, linewidth=1.2pt](0,0){0.25}
	\pscircle[linecolor=line,linewidth=1.2pt](0,0){0.7}
	\end{pspicture}$ 
	so
	\[
		\mathcal{I}_0(X) =
	\begin{pspicture}[shift=-0.8](-1.5,-1)(1.5,1)
	\SpecialCoor
	\definecolor{line}{rgb}{0,0,1}
	\pscircle[fillstyle=ccslope,slopebegin=diskin,slopeend=diskout,linewidth=1.2pt](0,0){1}
	\pscircle[fillstyle=solid, fillcolor=white, linewidth=1.2pt](0,0){0.25}
	\pscircle[linecolor=line,linewidth=1.2pt](0,0){0.7}
	\psline[linecolor=line,linewidth=1.2pt](1;0)(0.8;0)
	\psline[linecolor=line,linewidth=1.2pt](0.25;0)(0.6;0)	
	\rput{0}(1.2;0){\tiny $1$}
	\end{pspicture}
	=
	t^{\frac{1}{4}}
	\begin{pspicture}[shift=-0.8](-1.5,-1)(1.5,1)
	\SpecialCoor
	\definecolor{line}{rgb}{0,0,1}
	\pscircle[fillstyle=ccslope,slopebegin=diskin,slopeend=diskout,linewidth=1.2pt](0,0){1}
	\pscircle[fillstyle=solid, fillcolor=white, linewidth=1.2pt](0,0){0.25}
	\degrees[8]
	\pscurve[linecolor=line,linewidth=1.2pt](0.25;0)(0.35;0)(0.5;1)(0.6;2)(0.6;3)(0.6;4)(0.6;5)(0.6;6)(0.6;7)(0.8;8)(1;8)
	\rput{0}(1.2;0){\tiny $1$}
\end{pspicture}
+ t^{-\frac{1}{4}}
\begin{pspicture}[shift=-0.8](-1.5,-1)(1.5,1)
	\SpecialCoor
	\definecolor{line}{rgb}{0,0,1}
	\pscircle[fillstyle=ccslope,slopebegin=diskin,slopeend=diskout,linewidth=1.2pt](0,0){1}
	\pscircle[fillstyle=solid, fillcolor=white, linewidth=1.2pt](0,0){0.25}
	\degrees[8]
	\pscurve[linecolor=line,linewidth=1.2pt](0.25;8)(0.35;8)(0.5;7)(0.6;6)(0.6;5)(0.6;4)(0.6;3)(0.6;2)(0.6;1)(0.8;0)(1;0)
	\rput{0}(1.2;0){\tiny $1$}
\end{pspicture}
= t^{\frac{1}{4}} \rho + t^{-\frac{1}{4}}\rho^{-1}
	\]
by the Kauffman skein relation \eqref{kauffman}.\\
{\it Proof of} {\bf b}: We have 
$\rho = \begin{pspicture}[shift=-0.8](-1.5,-1)(1.5,1.2)
	\SpecialCoor
	\definecolor{line}{rgb}{0,0,1}
	\pscircle[fillstyle=ccslope,slopebegin=diskin,slopeend=diskout,linewidth=1.2pt](0,0){1}
	\pscircle[fillstyle=solid, fillcolor=white, linewidth=1.2pt](0,0){0.25}
	\degrees[8]
	\pscurve[linecolor=line,linewidth=1.2pt](0.25;0)(0.35;0)(0.5;1)(0.6;2)(0.6;3)(0.6;4)(0.6;5)(0.6;6)(0.6;7)(0.8;8)(1;8)
	\rput{0}(1.2;0){\tiny $1$}
\end{pspicture}$
	so 
\[
		\mathcal{I}_1(\rho) = 		
		 \begin{pspicture}[shift=-0.8](-1.5,-1)(1.5,1)
	\SpecialCoor
	\definecolor{line}{rgb}{0,0,1}
	\pscircle[fillstyle=ccslope,slopebegin=diskin,slopeend=diskout,linewidth=1.2pt](0,0){1}
	\pscircle[fillstyle=solid, fillcolor=white, linewidth=1.2pt](0,0){0.25}
	\degrees[8]
	\pscurve[linecolor=line,linewidth=1.2pt](0.25;0)(0.35;0)(0.5;1)(0.6;2)(0.6;3)(0.6;4)(0.6;5)(0.6;6)(0.6;7)(0.8;8)(1;8)
	\psline[linecolor=line,linewidth=1.2pt](1;4)(0.7;4)
			\psline[linecolor=line,linewidth=1.2pt](0.25;4)(0.5;4)
	\rput{0}(1.2;4){\tiny $2$}
	\rput{0}(1.2;0){\tiny $1$}
\end{pspicture}
			= \rho(t^{-\frac{1}{4}}e_1+t^{\frac{1}{4}})		
\]
by applying the Kauffman skein relation \eqref{kauffman} to the crossing and rewriting the resulting
expressions in terms of the generators of $\textup{TL}_2$ (compare with the proof of
Theorem \ref{presentationThm}{\bf c}). In a similar way one proves the explicit formula for
$\mathcal{I}_1(\rho^{-1})\in\textup{TL}_2$.\\
{\it Proof of} {\bf c}: The formulas for $\mathcal{I}_n(\rho^{\pm 1})\in\textup{TL}_{n+1}$
are obtained by a similar computation as in {\bf b}.

For $1\leq i<n$, applying the arc insertion functor to $e_i\in\textup{TL}_n$ does not introduce crossings. The resulting $(n+1,n+1)$-affine diagram represents the generator
$e_i$ in $\textup{TL}_{n+1}$, so $\mathcal{I}_n(e_i)=e_{i}$. 

Note that applying the arc insertion functor to $e_n\in\textup{TL}_n$
introduces two crossings. Resolving both crossings with the Kauffman skein relation \eqref{kauffman} and 
expressing the resulting linear combination of four $(n+1,n+1)$-affine diagrams in terms of
the generators of $\textup{TL}_{n+1}$ yield the formula
\[
\mathcal{I}_n(e_n)=t^{\frac{1}{2}}e_ne_{n+1}+t^{-\frac{1}{2}}e_{n+1}e_n+e_{n+1}+e_n=
(t^{\frac{1}{4}}e_n+t^{-\frac{1}{4}})e_{n+1}(t^{-\frac{1}{4}}e_n+t^{\frac{1}{4}}).
\]
\end{proof}

\begin{remark}
The calculation for part \textbf{a} in the proposition above has also been done in \cite[Prop. 2.2]{Le} where a similar skein algebra on the annulus is used to prove centrality of certain skeins.
\end{remark}

\section{Towers of extended affine Temperley-Lieb algebra modules}\label{section7}
In \cite{AlH} the sequence $\{\mathcal{I}_n\}_{n\in\mathbb{Z}_{\geq 0}}$
of algebra maps $\mathcal{I}_n: \textup{TL}_n\rightarrow\textup{TL}_{n+1}$
was used to study affine Markov traces. In \cite{GS} it was used to study fusion
of affine Temperley-Lieb modules.
In the next two sections we use the sequence $\{\mathcal{I}_n\}_{n\in\mathbb{Z}_{\geq 0}}$ 
of algebra maps to introduce the notion
of towers of extended affine Temperley-Lieb modules. We construct examples
that are relevant for understanding the dependence of dense loop models and Heisenberg
XXZ spin-$\frac{1}{2}$ chains on their system size  (cf. \cite{KP, FZJZ,ANS}).

We first introduce some notations. Let $A$ be a $\mathbb{C}$-algebra. Write $\mathcal{C}_A$
for the category of left $A$-modules. Write $\textup{Hom}_A(M,N)$ for the space of morphisms
$M\rightarrow N$ in $\mathcal{C}_A$, which we will call intertwiners.
Suppose that $\eta: A\rightarrow B$ is a (unit preserving) morphism of $\mathbb{C}$-algebras. Write
$\textup{Ind}^{\eta}: \mathcal{C}_A\rightarrow\mathcal{C}_B$ and $\textup{Res}^{\eta}: \mathcal{C}_B\rightarrow\mathcal{C}_A$ for the corresponding induction and restriction functor. Concretely,
if $M$ is a left $A$-module then
\[
\textup{Ind}^{\eta}(M):=B\otimes_AM
\]
with $B$ viewed as right $A$-module by $b\cdot a:=b\eta(a)$ for $b\in B$ and $a\in A$.
If $N$ is a left $B$-module then $\textup{Res}^{\eta}(N)$ is the complex vector space $N$ viewed
as $A$-module by $a\cdot n:=\eta(a)n$ for $a\in A$ and $n\in N$. The restriction functor $\textup{Res}^{\eta}$ is right adjoint to $\textup{Ind}^{\eta}$. If $M$ is a left $A$-module and
$N$ a left $B$-module, then the corresponding linear isomorphism
\[
\textup{Hom}_A\bigl(M,\textup{Res}^\eta(N)\bigr)\overset{\sim}{\longrightarrow}
\textup{Hom}_B\bigl(\textup{Ind}^\eta(M),N\bigr)
\]
is $\phi\mapsto\widehat{\phi}$ with 
$\widehat{\phi}\in\textup{Hom}_B\bigl(\textup{Ind}^\eta(M),N\bigr)$ defined by
\[
\widehat{\phi}\bigl(Z\otimes_Am\bigr):=Z\phi(m)
\]
for $Z\in B$ and $m\in M$.


For a left $\textup{TL}_{n+1}$-module $V_{n+1}$ we use the shorthand notation 
$V_{n+1}^{\mathcal{I}}$ for the left $\textup{TL}_n$-module $\textup{Res}^{\mathcal{I}_n}(V_{n+1})$.

\begin{definition}
We call 
\[
V_0\overset{\phi_0}{\longrightarrow}V_1\overset{\phi_1}{\longrightarrow}V_2
\overset{\phi_2}{\longrightarrow}V_3\overset{\phi_3}{\longrightarrow}\cdots
\]
with $V_n$ a left $\textup{TL}_n$-module
and $\phi_n\in\textup{Hom}_{\textup{TL}_n}\bigl(V_n,V_{n+1}^{\mathcal{I}}\bigr)$ 
a tower of extended affine Temperley-Lieb algebra modules. We will sometimes denote
the tower by $\{(V_n,\phi_n)\}_{n\in\mathbb{Z}_{\geq 0}}$.
\end{definition}
\begin{example}
The interpretation of the extended affine Temperley-Lieb algebras as the endomorphism
spaces of the skein category $\cS$ immediately produces examples of
towers of extended affine Temperley-Lieb algebra modules. 
For example, for $m\in\mathbb{Z}_{\geq 0}$ we have the tower
$\{(V_n^{(m)},\phi_n^{(m)})\}_{n\in\mathbb{Z}_{\geq 0}}$ with
\[
V_n^{(m)}:=\textup{Hom}_{\cS}(m+n,n)
\]
viewed as a left module over $\textup{TL}_n\simeq\textup{End}_{\cS}(n)$ with representation map
\[
\pi_n^{(m)}(Y)Z:=Y\circ Z
\]
for $Y\in\textup{End}_{\cS}(n)$ and $Z\in V_n^{(m)}=\textup{Hom}_{\cS}(m+n,n)$, and with intertwiners 
\[
\phi_n^{(m)}:=\mathcal{I}\vert_{\textup{Hom}_{\cS}(m+n,n)}: V_n^{(m)}\rightarrow V_{n+1}^{(m)}.
\]
There are other intertwiners $V_n^{(m)}\rightarrow V_{n+1}^{(m)}$ one can take here; for instance, $Z\mapsto \mathcal{I}(Z)\circ R$ for some $R\in\textup{End}_{\cS}(m+n+1)$.
 A refinement of this example will play an important role in the construction of the link pattern tower
 in the next section.
\end{example}
In the definition of towers $\{(V_n,\phi_n)\}_{n\in\mathbb{Z}_{\geq 0}}$ of extended affine Temperley-Lieb algebra modules we do not require conditions on the intertwiners $\phi_n$,
in particular allowing trivial intertwiners. The interesting towers are the nondegenerate ones, which
are defined as follows.

\begin{definition}
We say that the tower $\{(V_n,\phi_n)\}_{n\in\mathbb{Z}_{\geq 0}}$ of extended affine Temperley-Lieb algebra modules is nondegenerate if $\widehat{\phi}_n: \textup{Ind}^{\mathcal{I}_n}(V_n)\rightarrow V_{n+1}$ is surjective for all $n\in\mathbb{Z}_{\geq 0}$.
\end{definition} 
In particular, for a nondegenerate tower $\{(V_n,\phi_n)\}_{n\in\mathbb{Z}_{\geq 0}}$ of extended affine Temperley-Lieb algebra modules, the module $V_{n+1}$ is a quotient
of $\textup{Ind}^{\mathcal{I}_n}(V_n)$,
\[
V_{n+1}\simeq\textup{coim}\bigl(\widehat{\phi}_n\bigr).
\]

We give an important example of a nondegenerate tower of extended affine Temperley-Lieb algebra
modules in the next section.

\section{The link pattern tower}\label{section8}
Motivated  by applications to
integrable models in statistical physics \cite{KP,FZJZ,ANS}, in particular to the dense loop model and the Heisenberg XXZ spin$-\frac{1}{2}$ chain, we construct in this section a family of towers of extended affine Temperley-Lieb algebra modules acting on spaces of link patterns on the punctured disc. We use the skein categorical context to
build the tower. 

The composition in the skein category $\mathcal{S}$ turns the hom-space 
$\textup{Hom}_{\mathcal{S}}(m,n)$ into a $\textup{End}_{\mathcal{S}}(n)$-$\textup{End}_{\mathcal{S}}(m)$-bimodule. We regard this as a $\textup{TL}_n$-$\textup{TL}_m$-bimodule structure on
$\textup{Hom}_{\mathcal{S}}(m,n)$ using
the isomorphism $\textup{End}_{\mathcal{S}}(n)\simeq \textup{TL}_n$ from Remark
\ref{isoRem}.
Note that for a left $\textup{TL}_m$-module $W_m$, 
\[
\textup{Hom}_{\mathcal{S}}(m,n)\otimes_{\textup{TL}_m}W_m
\]
is naturally a left $\textup{TL}_n$-module.

For $n=2k$ with $k\in\mathbb{Z}_{\geq 0}$
and $u\in\mathbb{C}$ we define the left $\textup{TL}_{2k}$-module $V_{2k}(u)$ by
\[
V_{2k}(u):=\textup{Hom}_{\mathcal{S}}(0,2k)\otimes_{\textup{TL}_0}\mathbb{C}_0^{(u)},
\]
with $\mathbb{C}_0^{(u)}$ the one-dimensional module over $\textup{TL}_0=\mathbb{C}[X]$
satisfying $X\mapsto u$. 
For $Y\in\textup{Hom}_{\mathcal{S}}(0,2k)$ we write $Y_u$ for the element $Y\otimes_{\textup{TL}_0}1$ in $V_{2k}(u)$.

For $n=2k+1$ with $k\in\mathbb{Z}_{\geq 0}$ 
and $v\in\mathbb{C}^\ast$ we define the left $\textup{TL}_{2k+1}$-module $V_{2k+1}(v)$ by
\[
V_{2k+1}(v):=\textup{Hom}_{\mathcal{S}}(1,2k+1)\otimes_{\textup{TL}_1}\mathbb{C}_1^{(v)},
\]
with $\mathbb{C}_1^{(v)}$ the one-dimensional module over
$\textup{TL}_1=\mathbb{C}[\rho^{\pm 1}]$ satisfying $\rho\mapsto v$. 
For $Z\in\textup{Hom}_{\mathcal{S}}(1,2k+1)$ we write $Z_v$ for the element $Z\otimes_{\textup{TL}_1}1$ in $V_{2k+1}(v)$.

\begin{remark}\label{standardremark}
The left $\textup{TL}_{2k}$-module $V_{2k}(u)$ and the left $\textup{TL}_{2k+1}$-module
$V_{2k+1}(v)$ are examples of the so-called standard $\textup{TL}_N$-modules $\mathcal{W}_{j,z}[N]$ from \cite[\S 4.2]{GS} (the extended affine Temperley-Lieb algebra $\textup{TL}_N$ is denoted by
$\textup{TL}_N^a$ in \cite{GS}). Concretely, writing $u=x+x^{-1}$ with $x\in\mathbb{C}^\ast$, we have
\begin{equation*}
V_{2k}(u)=\mathcal{W}_{0,x}[2k],\qquad
V_{2k+1}(v)=\mathcal{W}_{\frac{1}{2},v}[2k+1].
\end{equation*}
\end{remark}

Next we study towers having the modules $V_{2k}(u)$ and $V_{2k+1}(v)$ as building blocks.
For this we need special elements in the skein modules $\textup{End}_{\mathcal{S}}(0)$,
$\textup{End}_{\mathcal{S}}(1)$ and $\textup{Hom}_{\mathcal{S}}(0,2)$. Let $\emptyset\in
\textup{End}_{\mathcal{S}}(0)$ be the skein class of the empty tangle diagram in $A$
and write $\mathbf{1}:=\mathbf{1}_1$ for the identity morphism in $\textup{End}_{\mathcal{S}}(1)$. Then $V_0(u)=\mathbb{C}\emptyset_u$ and 
$V_1(v)=\mathbb{C}\mathbf{1}_v$. For $V_2(u)$, note that the skein module $\textup{Hom}_{\mathcal{S}}(0,2)$ is a free right 
$\textup{TL}_0=\mathbb{C}[X]$-module with 
$\textup{TL}_0$-basis $\{[c_+], [c_-]\}$, where
\[
c_+ = 
 \begin{pspicture}[shift=-1.1](-1.5,-1.25)(1.5,1.25)
	\SpecialCoor
	\pscircle[fillstyle=ccslope,slopebegin=diskin,slopeend=diskout,linewidth=1.2pt](0,0){1}
	\pscircle[fillstyle=solid, fillcolor=white, linewidth=1.2pt](0,0){0.25}
	\degrees[8]
	\pscurve[linecolor=line,linewidth=1.2pt](1;0)(0.9;0)(0.7;1.5)(0.7;2.5)(0.9;4)(1;4)
	\rput{0}(1.2;0){\tiny $1$}
	\rput{0}(1.2;4){\tiny $2$}
\end{pspicture},
\qquad
c_- =
 \begin{pspicture}[shift=-1.1](-1.5,-1.25)(1.5,1.25)
	\SpecialCoor
	\pscircle[fillstyle=ccslope,slopebegin=diskin,slopeend=diskout,linewidth=1.2pt](0,0){1}
	\pscircle[fillstyle=solid, fillcolor=white, linewidth=1.2pt](0,0){0.25}
	\degrees[8]
	\pscurve[linecolor=line,linewidth=1.2pt](1;0)(0.9;0)(0.7;6.5)(0.7;5.5)(0.9;4)(1;4)
	\rput{0}(1.2;0){\tiny $1$}
	\rput{0}(1.2;4){\tiny $2$}
\end{pspicture}.
\]
In particular, $V_2(u)$ is two-dimensional with linear basis $\{(c_+)_u,(c_-)_u\}$.
Write $U:=t^{\frac{1}{4}}[c_+]+v[c_-]\in\textup{Hom}_{\cS}(0,2)$. In pictures,
\[
U = t^{\frac{1}{4}}
 \begin{pspicture}[shift=-1.1](-1.5,-1.25)(1.5,1.25)
	\SpecialCoor
	\pscircle[fillstyle=ccslope,slopebegin=diskin,slopeend=diskout,linewidth=1.2pt](0,0){1}
	\pscircle[fillstyle=solid, fillcolor=white, linewidth=1.2pt](0,0){0.25}
	\degrees[8]
	\pscurve[linecolor=line,linewidth=1.2pt](1;0)(0.9;0)(0.7;1.5)(0.7;2.5)(0.9;4)(1;4)
	\rput{0}(1.2;0){\tiny $1$}
	\rput{0}(1.2;4){\tiny $2$}
\end{pspicture}
+ v
 \begin{pspicture}[shift=-1.1](-1.5,-1.25)(1.5,1.25)
	\SpecialCoor
	\pscircle[fillstyle=ccslope,slopebegin=diskin,slopeend=diskout,linewidth=1.2pt](0,0){1}
	\pscircle[fillstyle=solid, fillcolor=white, linewidth=1.2pt](0,0){0.25}
	\degrees[8]
	\pscurve[linecolor=line,linewidth=1.2pt](1;0)(0.9;0)(0.7;6.5)(0.7;5.5)(0.9;4)(1;4)
	\rput{0}(1.2;0){\tiny $1$}
	\rput{0}(1.2;4){\tiny $2$}
\end{pspicture}.
\]
\begin{lemma}\label{Step12} Let $u\in\mathbb{C}$ and $v\in\mathbb{C}^\ast$.
\begin{enumerate}
\item[{\bf (i)}] Define the linear map $\phi_0: V_0(u)\rightarrow V_1(v)$ by $\phi_0(\emptyset_u):=\mathbf{1}_v$.
Then 
\begin{equation*}
\textup{Hom}_{\textup{TL}_0}\bigl(V_0(u),V_1(v)^{\mathcal{I}}\bigr)=
\begin{cases}
\mathbb{C}\phi_0\quad &\hbox{ if } u=t^{\frac{1}{4}}v+t^{-\frac{1}{4}}v^{-1},\\
\{0\}\quad &\hbox{ otherwise.}
\end{cases}
\end{equation*}
\item[{\bf (ii)}] 
Let $u=t^{\frac{1}{4}}v+t^{-\frac{1}{4}}v^{-1}$. Define the linear map $\phi_1: V_1(v)\rightarrow V_2(u)$ by $\phi_1(\mathbf{1}_v):=U_u$. Then 
\[
\textup{Hom}_{\textup{TL}_1}\bigl(
V_1(v),V_2(u)^{\mathcal{I}}\bigr)=\mathbb{C}\phi_1.
\]
\end{enumerate}
\end{lemma}
\begin{proof}
{\bf (i)} Note that $\phi_0\in\textup{Hom}_{\textup{TL}_0}\bigl(V_0(u),V_1(v)^{\mathcal{I}}\bigr)$ 
if and only if $\mathcal{I}_0(X)\mathbf{1}_v=u\mathbf{1}_v$. Proposition \ref{insertionProp}{\bf (a)}
gives $\mathcal{I}_0(X)\mathbf{1}_v=(t^{\frac{1}{4}}v+t^{-\frac{1}{4}}v^{-1})\mathbf{1}_v$, hence the result.\\
{\bf (ii)} Take an arbitrary element $Z_u\in V_2(u)$ with $Z\in\textup{Hom}_{\cS}(0,2)$. 
The linear map 
$\chi: V_1(v)\rightarrow V_2(u)$ defined by $\chi(\mathbf{1}_v)=Z_v$ is in 
$\textup{Hom}_{\textup{TL}_1}\bigl(V_1(v),V_2(u)^{\mathcal{I}}\bigr)$ if and only if
\[
\mathcal{I}_1(\rho)Z_u=vZ_u
\]
in $V_2(u)$. By Proposition \ref{insertionProp}{\bf (b)} we have
$\mathcal{I}_1(\rho)=\rho(t^{-\frac{1}{4}}e_1+t^{\frac{1}{4}})$.
A direct computation in $\textup{Hom}_{\cS}(0,2)$ shows that 
\begin{equation}\label{basisaction}
\begin{split}
\rho(t^{-\frac{1}{4}}e_1+t^{\frac{1}{4}})\circ [c_+]&=-t^{-\frac{3}{4}}[c_-],\\
\rho(t^{-\frac{1}{4}}e_1+t^{\frac{1}{4}})\circ [c_-]&=t^{\frac{1}{4}}[c_+]+t^{-\frac{1}{4}}([c_-]\circ X),
\end{split}
\end{equation}
where we have used the loop removal relation \eqref{loopremoval} in the derivation of the first identity. Writing
$m_{\alpha,\beta}:=\alpha (c_+)_u+\beta (c_-)_u\in V_2(u)$ with $\alpha,\beta\in\mathbb{C}$ we
obtain from \eqref{basisaction},
\[
\mathcal{I}_1(\rho)m_{\alpha,\beta}=\rho(t^{-\frac{1}{4}}e_1+t^{\frac{1}{4}})m_{\alpha,\beta}=vm_{\alpha^\prime,\beta^\prime}
\]
with 
\begin{equation*}
\begin{split}
\Bigl(\begin{matrix} \alpha^\prime\\ \beta^\prime\end{matrix}\,\Bigr)=
M\Bigl(\begin{matrix} \alpha\\ \beta\end{matrix}\Bigr),\qquad
M:=
\left(\begin{matrix} 0 & t^{\frac{1}{4}}v^{-1}\\ -t^{-\frac{3}{4}}v^{-1} & 1+t^{-\frac{1}{2}}v^{-2}
\end{matrix}\right).
\end{split}
\end{equation*}
Since $m_{t^{\frac{1}{4}},v}=U_u$ it remains to show that $M$ has eigenvalue $1$ with corresponding eigenspace
$\mathbb{C}\Bigl(\begin{matrix} t^{\frac{1}{4}}\\ v\end{matrix}\Bigr)$. Clearly
$\Bigl(\begin{matrix} t^{\frac{1}{4}}\\ v\end{matrix}\Bigr)$ is an eigenvector of $M$ with eigenvalue $1$.
The characteristic polynomial of $M$ is
\[
p_M(\lambda)=(\lambda-1)(\lambda-t^{-\frac{1}{2}}v^{-2}),
\]
hence the result follows for $t^{-\frac{1}{2}}v^{-2}\not=1$. If $t^{-\frac{1}{2}}v^{-2}=1$ then a direct
check shows that the geometric multiplication of the eigenvalue $1$ of $M$ is still one.
\end{proof}
Note that the intertwiners $\phi_0$ and $\phi_1$ can alternatively be characterized by the formulas
\begin{equation*}
\begin{split}
\phi_0(Y_u)&=\bigl(\mathcal{I}(Y)\bigr)_v,\qquad\qquad\,\,\, Y\in\textup{End}_{\cS}(0),\\
\phi_1(Z_v)&=\bigl(\mathcal{I}(Z)\circ U\bigr)_u,\qquad\,\,\,\,Z\in\textup{End}_{\cS}(1)
\end{split}
\end{equation*}
since $\mathcal{I}(\emptyset)_v=1_v$ and $(\mathcal{I}(\mathbf{1})\circ U)_u=U_u$.

The following theorem shows that $\phi_0$ and $\phi_1$ can be extended to a nondegenerate
tower
\[
V_0(u)\overset{\phi_0}{\longrightarrow}V_1(v)\overset{\phi_1}{\longrightarrow}
V_2(u)\overset{\phi_2}{\longrightarrow}\cdots
\]
of extended affine Temperley-Lieb modules when $u=t^{\frac{1}{4}}v+t^{-\frac{1}{4}}v^{-1}$.

\begin{theorem}\label{intertwinerstower}
Let $v\in\mathbb{C}^\ast$. Set $u:=t^{\frac{1}{4}}v+t^{-\frac{1}{4}}v^{-1}$ and
 let $k\in\mathbb{Z}_{\geq 0}$. 
 \begin{enumerate}
 \item[{\bf (i)}]  There exist unique intertwiners 
 $\phi_{2k}\in\textup{Hom}_{\textup{TL}_{2k}}\bigl(V_{2k}(u),
V_{2k+1}(v)^{\mathcal{I}}\bigr)$ and\\ 
$\phi_{2k+1}\in\textup{Hom}_{\textup{TL}_{2k+1}}\bigl(V_{2k+1}(v),V_{2k+2}(u)^{\mathcal{I}}\bigr)$
satisfying
\begin{equation*}
\begin{split}
\phi_{2k}(Y_u)&:=\bigl(\mathcal{I}(Y)\bigr)_v,\qquad\qquad\,\,\, Y\in\textup{Hom}_{\cS}(0,2k),\\
\phi_{2k+1}(Z_v)&:=\bigl(\mathcal{I}(Z)\circ U\bigr)_u,\qquad\,\,\,\, Z\in\textup{Hom}_{\cS}(1,2k+1).
\end{split}
\end{equation*}
\item[{\bf (ii)}] The tower 
\[
V_0(u)\overset{\phi_0}{\longrightarrow}V_1(v)\overset{\phi_1}{\longrightarrow}
V_2(u)\overset{\phi_2}{\longrightarrow}V_3(v)\overset{\phi_3}{\longrightarrow}\cdots
\]
of extended affine Temperley-Lieb algebra modules is nondegenerate if 
$v^2\not=t^{\frac{1}{2}}$.
\end{enumerate}
\end{theorem}
\begin{proof}
{\bf (i)} If the maps $\phi_{2k}$ and $\phi_{2k+1}$ are well-defined, then they are obviously intertwiners. To prove that $\phi_{2k}$ and $\phi_{2k+1}$ are well-defined
we have to show that $\bigl(\mathcal{I}(Y)\circ\mathcal{I}(X)\bigr)_v=
u\mathcal{I}(Y)_v$ in $V_{2k+1}(v)$ for $Y\in\textup{Hom}_{\cS}(0,2k)$ and 
$\bigl(\mathcal{I}(Z)\circ(\mathcal{I}(\rho)\circ U)\bigr)_u=v\bigl(\mathcal{I}(Z)\circ U\bigr)_u$ in
$V_{2k+2}(u)$ for $Z\in\textup{Hom}_{\cS}(1,2k+1)$. This is analogous to the proof of Lemma
\ref{Step12}.\\
{\bf (ii)}  Consider the tangle diagrams
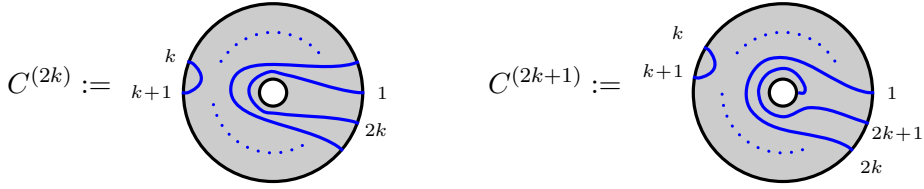
\begin{center}
  {\psset{unit=.8cm}
$C^{(2k)}:=\,\,\,$   \begin{pspicture}[shift=-1.8](-2.25,-1.8)(2.25,1.8)
     	\pscircle[fillstyle=solid, fillcolor=diskin,linewidth=1.2pt](0,0){1.5}
    	\degrees[360]
	\pscircle[fillstyle=solid, fillcolor=white,linewidth=1.2pt](0,0){0.25}
	\rput{0}(1.8;0){\tiny$1$}
	\rput{0}(1.8;-20){\tiny$2k$}
	\rput{0}(1.8;160){\tiny$k$}
	\rput{0}(2;180){\tiny$k\!+\!1$}
	\pscurve[linecolor=line,linewidth=1.2pt](1.5;180)(1.4;180)(1.2;170)(1.4;160)(1.5;160)
	\pscurve[linecolor=line,linewidth=1.2pt](1.5;-20)(1.4;-20)(0.35;280)(0.35;240)(-0.4;-10)(0.35;100)(0.35;60)(1.4;0)(1.5;0)
	\pscurve[linecolor=line,linewidth=1.2pt](1.5;-40)(1.4;-40)(-0.7;-10)(1.4;20)(1.5;20)
	\psarc[linestyle=dotted,linecolor=line,linewidth=1.2pt](0,0){1}{190}{-60}
	\psarc[linestyle=dotted,linecolor=line,linewidth=1.2pt](0,0){1}{40}{150}
   \end{pspicture}    }
\qquad   
  {\psset{unit=.8cm} 
$C^{(2k+1)}:=\,\,\,$ \begin{pspicture}[shift=-1.8](-2.25,-1.8)(2.25,1.8)
     	\pscircle[fillstyle=solid, fillcolor=diskin,linewidth=1.2pt](0,0){1.5}
	\pscircle[fillstyle=solid, fillcolor=white,linewidth=1.2pt](0,0){0.25}
    	\degrees[360]
	\rput{0}(1.8;0){\tiny$1$}
	\rput{0}(2;-20){\tiny$2k\!+\!1$}
	\rput{0}(2;150){\tiny$k$}
	\rput{0}(2;170){\tiny$k\!+\!1$}
	\rput{0}(1.9;-40){\tiny$2k$}
	\pscurve[linecolor=line,linewidth=1.2pt](1.5;170)(1.4;170)(1.2;160)(1.4;150)(1.5;150)
	\pscurve[linecolor=line,linewidth=1.2pt](1.5;-40)(1.4;-40)(-0.6;30)(-0.6;340)(-0.6;290)(1.4;0)(1.5;0)
	\pscurve[linecolor=line,linewidth=1.2pt](1.5;-20)(1.4;-20)(0.5;-30)(0.4;270)(0.4;220)(0.4;180)(0.4;160)(0.4;140)(0.4;120)(0.4;80)(0.37;40)(0.35;0)(0.25;0)
	\psarc[linestyle=dotted,linecolor=line,linewidth=1.2pt](0,0){1}{180}{290}
	\psarc[linestyle=dotted,linecolor=line,linewidth=1.2pt](0,0){1}{30}{140}
   \end{pspicture}   
   }
\end{center}
   
\noindent
respectively. We claim that $V_{2k}(u)=\textup{TL}_{2k}\cdot [C^{(2k)}]_u$ and 
$V_{2k+1}(v)=\textup{TL}_{2k+1}\cdot [C^{(2k+1)}]_v$. 

To prove this we use the matchmaker representation of the 
finite Temperley-Lieb algebra $\textup{TL}_{2k}^{fin}$ (see, e.g., \cite[\S 2.1]{dG}). The finite Temperley-Lieb algebra $\textup{TL}_{2k}^{fin}$ is the subalgebra of $\textup{TL}_{2k}$ generated by $e_1,\ldots,e_{2k-1}$. The representation space 
$M_{2k}$ of the matchmaker representation is the vector space with linear basis the non-crossing perfect matchings of $\{1,\ldots,2k\}$. Such non-crossing perfect matchings are viewed as nonintersecting arcs in a strip with the ordered endpoints $1,\ldots,2k$ positioned on the bottom
line of the strip. We will call such non-crossing perfect matchings link patterns. The $e_j$ acts on link patterns as the matchmaker of $j$ and $j+1$ (see \cite[(1)]{dG}), with the convention that if $j$ and $j+1$ in the link pattern were already matched, then $e_j$ acts by multiplication by the scalar factor $-(t^{\frac{1}{2}}+t^{-\frac{1}{2}})$.

Let $L^{(2k)}\in M_{2k}$
be the link pattern connecting $j$ to $2k+1-j$ for $j=1,\ldots,2k$. By wrapping the link pattern on the annulus
in such a way that $\{1,\ldots,2k\}$ correspond to the marked points $2\xi_{2k}^{j-1}$
($j=1,\ldots,2k$), we get an injective $\textup{TL}_{2k}^{fin}$-module morphism $M_{2k}\hookrightarrow V_{2k}(u)$ mapping $L^{(2k)}$ to $[C^{(2k)}]_u$.
If we in addition insert an arc via $-\ast\textup{Id}_1$ before projecting onto the skein, 
we get an injective $\textup{TL}_{2k}^{fin}$-module morphism
$M_{2k}\hookrightarrow V_{2k+1}(v)$ mapping $L^{(2k)}$ to
$[C^{(2k)}\ast\textup{Id}_1]_v=[C^{(2k+1)}]_v$. 

With these observations and the fact that $\rho\in\textup{TL}_n$ can be used to turn diagrams
in $A$ counterclockwise by an angle of $2\pi/n$, the claim is a consequence of
$M_{2k}=\textup{TL}_{2k}^{fin}\cdot L^{(2k)}$. This in turn is easy to establish using the alternative
description of link patterns in terms of Dyck paths (see, e.g., \cite[\S 2.4]{dGP}).

Now note that
\[
\widehat{\phi}_{2k}\bigl(\mathbf{1}_{2k+1}\otimes_{\textup{TL}_{2k}}[C^{(2k)}]_u\bigr)
=\bigl(\mathcal{I}\bigl([C^{(2k)}]\bigr)\bigr)_v
=[C^{(2k+1)}]_v
\]
hence $\widehat{\phi}_{2k}\in
\textup{Hom}_{\textup{TL}_{2k+1}}\bigl(\textup{Ind}^{\mathcal{I}_{2k}}\bigl(V_{2k}(u)\bigr),V_{2k+1}(v)\bigr)$
is surjective. By a direct computation we have
\begin{equation*}
\begin{split}
\widehat{\phi}_{2k+1}\bigl(e_{k+1}\cdots &e_{2k}e_{2k+1}\otimes_{\textup{TL}_{2k+1}}
[C^{(2k+1)}]_v\bigr)=\\
&=\bigl(e_{k+1}\cdots e_{2k}e_{2k+1}\mathcal{I}\bigl([C^{(2k+1)}]\bigr)U\bigr)_u\\
&=\bigl(t^{\frac{1}{4}}v^2-t^{\frac{3}{4}}\bigr)[C^{(2k+2)}]_u,
\end{split}
\end{equation*}
hence $\widehat{\phi}_{2k+1}\in\textup{Hom}_{\textup{TL}_{2k+2}}\bigl(
\textup{Ind}^{\mathcal{I}_{2k+1}}\bigl(V_{2k+1}(v)\bigr),V_{2k+2}(u)\bigr)$ is surjective
if $v^2\not=t^{\frac{1}{2}}$.
\end{proof}

\begin{remark}
The two skein classes $[C^{(2k)}]$ and $[C^{(2k+1)}]$
play an important role in determining the normalisation of the ground state of the dense loop model (see \cite{ANS,FZJZ,KP}).

\end{remark}

Fix $v\in\mathbb{C}^\ast$ and set $u=t^{\frac{1}{4}}v+t^{-\frac{1}{4}}v^{-1}$ for the remainder of this section. Note that for $n=2k$ the representation space $V_{2k}(u)$ consists of the equivalence classes of
the skein module $\textup{Hom}_{\cS}(0,2k)$ with respect to the equivalence relation obtained as
the linear and transitive closure of the 
{\it noncontractible loop removal relation}
\begin{equation}\label{nonloopremoval}
\begin{pspicture}[shift=-0.8](-1.5,-1)(1.5,1)
	\SpecialCoor
	\definecolor{line}{rgb}{0,0,1}
	\pscircle[fillstyle=ccslope,slopebegin=diskout,slopeend=diskin,linewidth=0.8pt,linestyle=dotted](0,0){1}
	\pscircle[fillstyle=solid, fillcolor=white, linewidth=1.2pt](0,0){0.25}
	\pscircle[linecolor=line,linewidth=1.2pt](0,0){0.7}
	\end{pspicture}
	= (t^{\frac{1}{4}}v+t^{-\frac{1}{4}}v^{-1})
	\begin{pspicture}[shift=-0.8](-1.5,-1)(1.5,1)
	\SpecialCoor
	\definecolor{line}{rgb}{0,0,1}
	\pscircle[fillstyle=ccslope,slopebegin=diskout,slopeend=diskin,linewidth=0.8pt,linestyle=dotted](0,0){1}
	\pscircle[fillstyle=solid, fillcolor=white, linewidth=1.2pt](0,0){0.25}
	\end{pspicture}
\end{equation}
For $n=2k+1$ odd, the representation space $V_{2k+1}(v)$ consists of the equivalence classes of
the skein module $\textup{Hom}_{\cS}(1,2k+1)$ with respect to the equivalence relation obtained as
the linear and transitive closure of the following {\it Dehn twist removal relation}
\begin{equation}\label{Dehnremoval}
\begin{pspicture}[shift=-0.8](-1.5,-1)(1.5,1)
	\SpecialCoor
	\definecolor{line}{rgb}{0,0,1}
	\pscircle[fillstyle=ccslope,slopebegin=diskin,slopeend=diskout,linewidth=1.2pt,linestyle=dotted](0,0){1}
	\pscircle[fillstyle=solid, fillcolor=white, linewidth=1.2pt](0,0){0.25}
	\degrees[8]
	\pscurve[linecolor=line,linewidth=1.2pt](0.25;0)(0.35;0)(0.5;1)(0.6;2)(0.6;3)(0.6;4)(0.6;5)(0.6;6)(0.6;7)(0.8;8)(1;8)
	\rput{0}(1.2;0){\tiny $1$}
\end{pspicture}
=
v\begin{pspicture}[shift=-0.8](-1.5,-1)(1.5,1)
	\SpecialCoor
	\definecolor{line}{rgb}{0,0,1}
	\pscircle[fillstyle=ccslope,slopebegin=diskin,slopeend=diskout,linewidth=1.2pt,linestyle=dotted](0,0){1}
	\pscircle[fillstyle=solid, fillcolor=white, linewidth=1.2pt](0,0){0.25}
	\degrees[8]
	\psline[linecolor=line,linewidth=1.2pt](1;0)(0.25;0)
	\rput{0}(1.2;0){\tiny $1$}
\end{pspicture}
\end{equation}

Let $\widetilde{\mathcal{C}}_{2k}$ be a set of representatives of the
planar isotopy classes of affine $(0,2k)$-diagrams without noncontractible loops. Let 
$\widetilde{\mathcal{C}}_{2k+1}$
be a set of representatives of the planar isotopy classes of the affine $(1,2k+1)$-diagrams 
that are planar isotopic to $D\ast\textup{Id}_1$ for some affine $(0,2k)$-diagram $D$. We will call the inserted arc connecting the inner boundary of $A$ with the outer boundary of $A$ the {\it defect
line} of the affine $(1,2k+1$)-diagram.
Observe that $\widetilde{\mathcal{B}}_{2k}:=\{[D]_u\,\, | \,\, D\in\widetilde{\mathcal{C}}_{2k}\}$ is a linear basis of 
$V_{2k}(u)$ and $\widetilde{\mathcal{B}}_{2k+1}:=\{[D]_v \,\, | \,\, D\in\widetilde{\mathcal{C}}_{2k+1}\}$ 
is a linear basis of $V_{2k+1}(v)$.

\begin{definition} Let $u=t^{\frac{1}{4}}v+t^{-\frac{1}{4}}v^{-1}$.
We call the tower 
\[
V_0(u)\overset{\phi_0}{\longrightarrow}V_1(v)\overset{\phi_1}{\longrightarrow}
V_2(u)\overset{\phi_2}{\longrightarrow} V_3(v)\overset{\phi_3}{\longrightarrow}\cdots
\]
of extended affine Temperley-Lieb algebra modules
the {\textup{link pattern tower}}. We call $v\in\mathbb{C}^\ast$ the 
{\textup{twist weight}} and $t^{\frac{1}{4}}v+
t^{-\frac{1}{4}}v^{-1}$ the {\textup{noncontractible loop weight}} of the link pattern tower.
\end{definition}
Note that the intertwiners $\phi_{2k}$ of the link pattern tower are simply given by the insertion of an arc in the underlying 
$(0,2k)$-tangle diagrams connecting the outer boundary with the inner boundary. This newly inserted arc is the defect line. The intertwiners $\phi_{2k+1}$ in the link pattern tower are more subtle. The intertwiner $\phi_{2k+1}$ acts on the representative of a $(1,2k+1)$-tangle diagram by detaching the defect line from the inner boundary and
reattaching it to the outer boundary in two different ways, corresponding to the two obvious ways that it can pass the hole of the annulus.
The two contributions get different weights $t^{\frac{1}{4}}$ and $v$, respectively. In Theorem \ref{intertwinerstower} we have described the operation $\phi_{2k+1}$ 
as the composition of arc insertion and composing with the linear combination 
$U\in\textup{Hom}_{\mathcal{S}}(0,2)$ of the two basic
$(0,2)$-tangle diagrams $c_+$ and $c_-$.

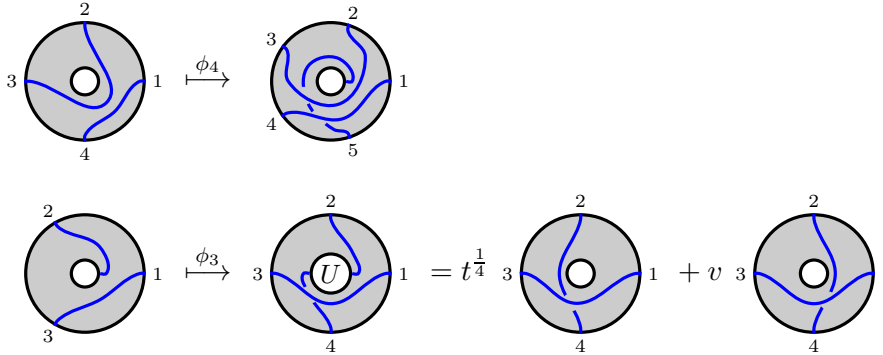
\begin{example}\label{exampleonannulus}
	\begin{align*}
		\begin{pspicture}[shift=-1.4](-1.5,-1.5)(1.5,1.5)
      			\SpecialCoor
    			\pscircle[fillstyle=solid, fillcolor=diskin,linewidth=1.2pt](0,0){1}
			\pscircle[fillstyle=solid, fillcolor=white, linewidth=1.2pt](0,0){0.25}
    			\degrees[4]
     			\rput{0}(1.2;0){\tiny $1$}
    			\rput{0}(1.2;1){\tiny $2$}
    			\rput{0}(1.2;2){\tiny $3$}
			\rput{0}(1.2;3){\tiny $4$}
			\pscurve[linecolor=line,linewidth=1.2pt](1;1)(0.9;1)(-0.5;1.5)(0.9;2)(1;2)
			\pscurve[linecolor=line,linewidth=1.2pt](1;3)(0.9;3)(0.7;3.5)(0.9;0)(1;0)
  		 \end{pspicture}
		& \overset{\phi_4}{\longmapsto}
		\begin{pspicture}[shift=-1.4](-1.5,-1.5)(1.5,1.5)
     			\SpecialCoor
    			\pscircle[fillstyle=solid, fillcolor=diskin,linewidth=1.2pt](0,0){1}
			\pscircle[fillstyle=solid, fillcolor=white, linewidth=1.2pt](0,0){0.25}
    			\degrees[5]
     			\rput{0}(1.2;0){\tiny $1$}
    			\rput{0}(1.2;1){\tiny $2$}
    			\rput{0}(1.2;2){\tiny $3$}
			\rput{0}(1.2;3){\tiny $4$}
			\rput{0}(1.2;4){\tiny $5$}
			\pscurve[linecolor=line,linewidth=1.2pt](1;4)(0.9;4)(0.8;3.8)(0.5;3)(0.45;2)(0.4;1)(0.35;0)(0.25;0)
			\psdot[linecolor=diskmid,dotsize=0.3](0.64; 3.5)
			\psdot[linecolor=diskmid,dotsize=0.3](0.5; 2.9)
			\pscurve[linecolor=line,linewidth=1.2pt](1;1)(0.9;1)(0.7;0.5)(-0.4;1.5)(0.7;2.5)(0.9;2)(1;2)
			\pscurve[linecolor=line,linewidth=1.2pt](1;3)(0.9;3)(0.65;4)(0.9;0)(1;0)
  		 \end{pspicture} \\
		 \begin{pspicture}[shift=-1.4](-1.5,-1.5)(1.5,1.5)
      			\SpecialCoor
    			\pscircle[fillstyle=solid, fillcolor=diskin,linewidth=1.2pt](0,0){1}
			\pscircle[fillstyle=solid, fillcolor=white, linewidth=1.2pt](0,0){0.25}
    			\degrees[3]
     			\rput{0}(1.2;0){\tiny $1$}
    			\rput{0}(1.2;1){\tiny $2$}
    			\rput{0}(1.2;2){\tiny $3$}
			\pscurve[linecolor=line,linewidth=1.2pt](1;1)(0.9;1)(0.5;0.5)(0.35;0)(0.25;0)
			\pscurve[linecolor=line,linewidth=1.2pt](1;2)(0.9;2)(0.5;2.5)(0.9;3)(1;3)
  		 \end{pspicture}
	&\overset{\phi_3}{\longmapsto}
		 \begin{pspicture}[shift=-1.4](-1.5,-1.5)(1.5,1.5)
    			\SpecialCoor
    			\pscircle[fillstyle=solid, fillcolor=diskin,linewidth=1.2pt](0,0){1}
			\pscircle[fillstyle=solid, fillcolor=white, linewidth=1.2pt](0,0){0.35}
    			\degrees[4]
     			\rput{0}(1.2;0){\tiny $1$}
    			\rput{0}(1.2;1){\tiny $2$}
    			\rput{0}(1.2;2){\tiny $3$}
			\rput{0}(1.2;3){\tiny $4$}
			\rput{0}(0;0){\small $U$}
			\pscurve[linecolor=line,linewidth=1.2pt](1;1)(0.9;1)(0.5;0.5)(0.45;0)(0.35;0)
			\pscurve[linecolor=line,linewidth=1.2pt](1;3)(0.9;3)(0.5;2.5)(0.45;2)(0.35;2)
			\psdot[linecolor=diskmid,dotsize=0.3](0.5; 2.5)
			\pscurve[linecolor=line,linewidth=1.2pt](1;2)(0.9;2)(0.5;3)(0.9;0)(1;0)
  		 \end{pspicture}
	 =
		 t^{\frac{1}{4}}
 			 \begin{pspicture}[shift=-1.4](-1.5,-1.5)(1.5,1.5)
    				\SpecialCoor
    				\pscircle[fillstyle=solid, fillcolor=diskin,linewidth=1.2pt](0,0){1}
				\pscircle[fillstyle=solid, fillcolor=white, linewidth=1.2pt](0,0){0.25}
    				\degrees[4]
   			  	\rput{0}(1.2;0){\tiny $1$}
    				\rput{0}(1.2;1){\tiny $2$}
    				\rput{0}(1.2;2){\tiny $3$}
				\rput{0}(1.2;3){\tiny $4$}
				\pscurve[linecolor=line,linewidth=1.2pt](1;1)(0.9;1)(0.35;2)(0.9;3)(1;3)
				\psdot[linecolor=diskmid,dotsize=0.3](0.5; 2.8)
				\pscurve[linecolor=line,linewidth=1.2pt](1;0)(0.9;0)(0.5;2.9)(0.9;2)(1;2)
  			 \end{pspicture}
  		 +v
   			  \begin{pspicture}[shift=-1.4](-1.5,-1.5)(1.5,1.5)
                                 	\SpecialCoor
                                	\pscircle[fillstyle=solid, fillcolor=diskin,linewidth=1.2pt](0,0){1}
				\pscircle[fillstyle=solid, fillcolor=white, linewidth=1.2pt](0,0){0.25}
                                	\degrees[4]
                                 	\rput{0}(1.2;0){\tiny $1$}
                                	\rput{0}(1.2;1){\tiny $2$}
                                	\rput{0}(1.2;2){\tiny $3$}
                            	\rput{0}(1.2;3){\tiny $4$}
                            	\pscurve[linecolor=line,linewidth=1.2pt](1;1)(0.9;1)(0.35;0)(0.9;3)(1;3)
                            	\psdot[linecolor=diskmid,dotsize=0.3](0.5; 3.3)
                            	\pscurve[linecolor=line,linewidth=1.2pt](1;0)(0.9;0)(0.5;3)(0.9;2)(1;2)
   			\end{pspicture} 
		\end{align*}

\end{example}

Let $\mathbb{D}:=\{z\in\mathbb{C} \,\, | \,\, |z|\leq 2\}$
be the unit disc of radius two and $\mathbb{D}^\ast:=\mathbb{D}\setminus \{0\}$. 
Let $\mathcal{L}_{2k}$ be the set of link patterns in $\mathbb{D}^\ast$
connecting the $2k$ marked points $\{2\xi_{2k}^{i-1}\}_{i=1}^{2k}$, i.e., it is 
the set of perfect noncrossing matchings within $\mathbb{D}^\ast$ of the 
marked points $\{2\xi_{2k}^{i-1}\}_{i=1}^{2k}$. For $n=2k+1$ odd, let
$\mathcal{L}_{2k+1}$ be the set of link patterns in $\mathbb{D}$ connecting the $2k+2$ marked points
$\{0\}\cup\{2\xi_{2k+1}^{i-1}\}_{i=1}^{2k+1}$. 
In this context we call the line connecting to $0$ the defect line. Since the defect line 
is now connected to $0$ instead of the hole of the annulus we are losing the information about
the winding of the defect line.
This allows us to realize the link pattern tower
for twist weight $v=1$ on the vector spaces $\mathbb{C}[\mathcal{L}_n]$ with linear basis $\mathcal{L}_n$ as follows.

Consider the map $A\rightarrow\mathbb{D}:=\{z\in\mathbb{C} \,\, | \,\, |z|\leq 2\}$ given by $re^{i\theta}\mapsto 2e^{\frac{2-r}{r-1}}e^{i\theta}$ for $r\in (1,2]$ and mapping $C_i$ onto $0$. 
Note that the map fixes
the outer boundary $C_o$ pointwise.
In this way the set $\widetilde{\mathcal{C}}_{n}$ of affine diagrams in $A$ labelling a basis of 
the $n$th representation space in the link pattern tower is identified with
$\mathcal{L}_n$ for $n\in\mathbb{Z}_{\geq 0}$.
This gives a vector space identification of $L_{n}$ with
$\mathbb{C}[\mathcal{L}_{n}]$.
We now transport the $\textup{TL}_{2k}$-module structure on $V_{2k}(u)$ 
and the $\textup{TL}_{2k+1}$-module structure on $V_{2k+1}(v)$ to 
$\mathbb{C}[\mathcal{L}_{2k}]$ and $\mathbb{C}[\mathcal{L}_{2k+1}]$ respectively through these linear isomorphisms.
It leads to an explicit realization of the link pattern tower with twist weight $v=1$
as a tower $\{(\phi_n,\mathbb{C}[\mathcal{L}_n])\}_{n\in\mathbb{Z}_{\geq 0}}$ of extended affine Temperley-Lieb algebra modules.

Note that the descriptions of the intertwiners $\phi_{2k}$ and $\phi_{2k+1}$ in terms of link patterns are as before: $\phi_{2k}$ is the insertion of a defect line, and $\phi_{2k+1}$ is detaching the defect line from the puncture $0$ and reattaching it to
the outer boundary in two different ways.
Note though that the crucial second description of $\phi_{2k+1}$, in which a second defect line is added first and then the two defect lines are detached from the puncture $0$ and connected to each other in two different ways, requires that one works on the annulus $A$ instead of on the punctured disc $\mathbb{D}^\ast$. 
However, there is an analogue to this on the punctured disc using so-called puncture skein relations \cite{RY}, see Remark \ref{rogeryang} for further details.

\begin{example} Let $v=1$.
\begin{enumerate}
	
	\item Example of the action of $e_2 \in \textup{TL}_3$ on $\C[\mathcal{L}_3]$:
	$$ 
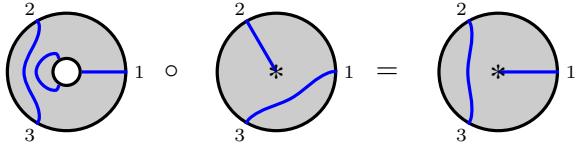
\begin{pspicture}[shift=-1.4](-1.5,-1.5)(1.5,1.5)
     	\SpecialCoor
    	\pscircle[fillstyle=solid, fillcolor=diskin,linewidth=1.2pt](0,0){1}
    	\pscircle[fillstyle=solid, fillcolor=white, linewidth=1.2pt](0,0){0.25}
    	\degrees[3]
	\rput{0}(1.2;0){\tiny $1$}
    	\rput{0}(1.2;1){\tiny $2$}
    	\rput{0}(1.2;2){\tiny $3$}
	\psline[linecolor=line,linewidth=1.2pt](1;0)(0.25;0)
	\pscurve[linecolor=line,linewidth=1.2pt](0.25;1)(0.35;1)(0.5;1.5)(0.35;2)(0.25;2)
	\pscurve[linecolor=line,linewidth=1.2pt](1;1)(0.9;1)(0.7;1.5)(0.9;2)(1;2)
   \end{pspicture}  
   \circ
       \begin{pspicture}[shift=-1.4](-1.5,-1.5)(1.5,1.5)
      	\SpecialCoor
    	\pscircle[fillstyle=solid, fillcolor=diskin,linewidth=1.2pt](0,0){1}
    	\degrees[3]
     	\rput{0}(1.2;0){\tiny $1$}
    	\rput{0}(1.2;1){\tiny $2$}
    	\rput{0}(1.2;2){\tiny $3$}
	\psline[linecolor=line,linewidth=1.2pt](1;1)(0;1)
	\pscurve[linecolor=line,linewidth=1.2pt](1;2)(0.9;2)(0.5;2.5)(0.9;3)(1;3)
	\psdot[dotstyle=asterisk,dotscale=1.5](0,0)
   \end{pspicture}
      =
  \begin{pspicture}[shift=-1.4](-1.5,-1.5)(1.5,1.5)
     	\pscircle[fillstyle=solid, fillcolor=diskin,linewidth=1.2pt](0,0){1}
    	\degrees[3]
     	\rput{0}(1.2;0){\tiny $1$}
    	\rput{0}(1.2;1){\tiny $2$}
    	\rput{0}(1.2;2){\tiny $3$}
	\psline[linecolor=line,linewidth=1.2pt](1;0)(0;0)
	\pscurve[linecolor=line,linewidth=1.2pt](1;1)(0.9;1)(0.5;1.5)(0.9;2)(1;2)
	\psdot[dotstyle=asterisk,dotscale=1.5](0,0)
   \end{pspicture}   
   $$
    	\item Example of the action of $e_2 \in \textup{TL}_4$ on $\C[\mathcal{L}_4]$:
   $$ 
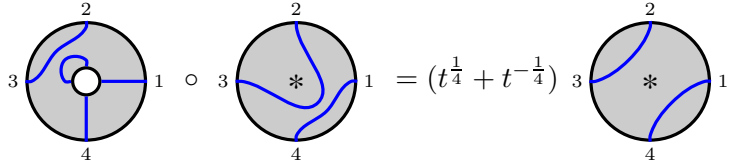
\begin{pspicture}[shift=-1.4](-1.5,-1.5)(1.5,1.5)
     	\SpecialCoor
    	\pscircle[fillstyle=solid, fillcolor=diskin,linewidth=1.2pt](0,0){1}
    	\pscircle[fillstyle=solid, fillcolor=white, linewidth=1.2pt](0,0){0.25}
    	\degrees[4]
	\rput{0}(1.2;0){\tiny $1$}
    	\rput{0}(1.2;1){\tiny $2$}
    	\rput{0}(1.2;2){\tiny $3$}
	\rput{0}(1.2;3){\tiny $4$}
	\psline[linecolor=line,linewidth=1.2pt](1;0)(0.25;0)
	\pscurve[linecolor=line,linewidth=1.2pt](0.25;1)(0.35;1)(0.5;1.5)(0.35;2)(0.25;2)
	\pscurve[linecolor=line,linewidth=1.2pt](1;1)(0.9;1)(0.7;1.5)(0.9;2)(1;2)
	\psline[linecolor=line,linewidth=1.2pt](1;3)(0.25;3)
   \end{pspicture}  
   \circ
       \begin{pspicture}[shift=-1.4](-1.5,-1.5)(1.5,1.5)
     	\SpecialCoor
    	\pscircle[fillstyle=solid, fillcolor=diskin,linewidth=1.2pt](0,0){1}
    	\degrees[4]
     	\rput{0}(1.2;0){\tiny $1$}
    	\rput{0}(1.2;1){\tiny $2$}
    	\rput{0}(1.2;2){\tiny $3$}
	\rput{0}(1.2;3){\tiny $4$}
	\pscurve[linecolor=line,linewidth=1.2pt](1;1)(0.9;1)(-0.5;1.5)(0.9;2)(1;2)
	\pscurve[linecolor=line,linewidth=1.2pt](1;3)(0.9;3)(0.7;3.5)(0.9;0)(1;0)
	\psdot[dotstyle=asterisk,dotscale=1.5](0,0)
   \end{pspicture}
= (t^{\frac{1}{4}}+t^{-\frac{1}{4}})
   \begin{pspicture}[shift=-1.4](-1.5,-1.5)(1.5,1.5)
      	\SpecialCoor
    	\pscircle[fillstyle=solid, fillcolor=diskin,linewidth=1.2pt](0,0){1}
    	\degrees[4]
     	\rput{0}(1.2;0){\tiny $1$}
    	\rput{0}(1.2;1){\tiny $2$}
    	\rput{0}(1.2;2){\tiny $3$}
	\rput{0}(1.2;3){\tiny $4$}
	\pscurve[linecolor=line,linewidth=1.2pt](1;1)(0.9;1)(0.5;1.5)(0.9;2)(1;2)
	\pscurve[linecolor=line,linewidth=1.2pt](1;3)(0.9;3)(0.5;3.5)(0.9;0)(1;0)
	\psdot[dotstyle=asterisk,dotscale=1.5](0,0)
   \end{pspicture}
   $$
   \item Example of the intertwiner $\phi_2$ acting on $\mathbb{C}[\mathcal{L}_2]$:
$$     
   
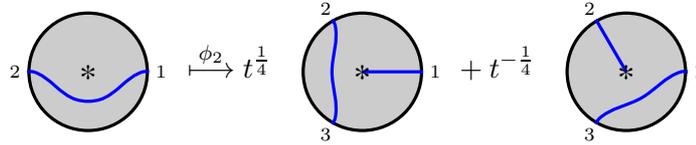
\begin{pspicture}[shift=-1.4](-1.5,-1.5)(1.5,1.5)
	\SpecialCoor
	\pscircle[fillstyle=ccslope,slopebegin=diskin,slopeend=diskout,linewidth=1.2pt](0,0){1}
	\degrees[8]
	\pscurve[linecolor=line,linewidth=1.2pt](1;0)(0.9;0)(0.5;6.5)(0.5;5.5)(0.9;4)(1;4)
	\rput{0}(1.2;0){\tiny $1$}
	\rput{0}(1.2;4){\tiny $2$}
	\psdot[dotstyle=asterisk,dotscale=1.5](0,0)
\end{pspicture} 
\overset{\phi_{2}}{\longmapsto}
   t^{\frac{1}{4}} 
    \begin{pspicture}[shift=-1.4](-1.5,-1.5)(1.5,1.5)
     	\SpecialCoor
    	\pscircle[fillstyle=solid, fillcolor=diskin,linewidth=1.2pt](0,0){1}
    	\degrees[3]
     	\rput{0}(1.2;0){\tiny $1$}
    	\rput{0}(1.2;1){\tiny $2$}
    	\rput{0}(1.2;2){\tiny $3$}
	\psline[linecolor=line,linewidth=1.2pt](1;0)(0;0)
	\pscurve[linecolor=line,linewidth=1.2pt](1;1)(0.9;1)(0.5;1.5)(0.9;2)(1;2)
	\psdot[dotstyle=asterisk,dotscale=1.5](0,0)
   \end{pspicture}
   + t^{-\frac{1}{4}} 
    \begin{pspicture}[shift=-1.4](-1.5,-1.5)(1.5,1.5)
     	\SpecialCoor
    	\pscircle[fillstyle=solid, fillcolor=diskin,linewidth=1.2pt](0,0){1}
    	\degrees[3]
     	\rput{0}(1.2;0){\tiny $1$}
    	\rput{0}(1.2;1){\tiny $2$}
    	\rput{0}(1.2;2){\tiny $3$}
	\psline[linecolor=line,linewidth=1.2pt](1;1)(0;1)
	\pscurve[linecolor=line,linewidth=1.2pt](1;2)(0.9;2)(0.5;2.5)(0.9;3)(1;3)
	\psdot[dotstyle=asterisk,dotscale=1.5](0,0)
   \end{pspicture}   
$$
	\item Example of the intertwiner $\phi_3$ acting on $\mathbb{C}[\mathcal{L}_3]$ (it corresponds to the second example from Example \ref{exampleonannulus} with $v=1$):
	
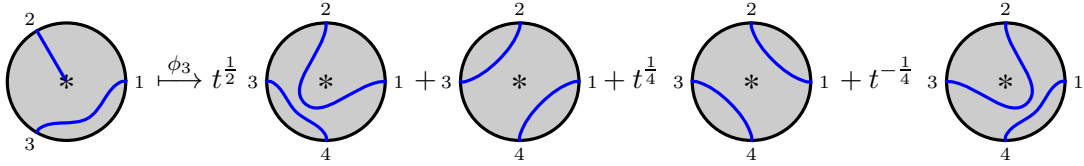
\begin{align*}
   \begin{pspicture}[shift=-1.4](-1.5,-1.5)(1.5,1.5)
      	\SpecialCoor
    	\pscircle[fillstyle=solid, fillcolor=diskin,linewidth=1.2pt](0,0){1}
    	\degrees[3]
     	\rput{0}(1.2;0){\tiny $1$}
    	\rput{0}(1.2;1){\tiny $2$}
    	\rput{0}(1.2;2){\tiny $3$}
	\psline[linecolor=line,linewidth=1.2pt](1;1)(0;1)
	\pscurve[linecolor=line,linewidth=1.2pt](1;2)(0.9;2)(0.7;2.5)(0.9;3)(1;3)
	\psdot[dotstyle=asterisk,dotscale=1.5](0,0)
   \end{pspicture} \!\!
\overset{\phi_{3}}{\longmapsto} 
t^{\frac{1}{2}}\!
   \begin{pspicture}[shift=-1.4](-1.5,-1.5)(1.5,1.5)
     	\SpecialCoor
    	\pscircle[fillstyle=solid, fillcolor=diskin,linewidth=1.2pt](0,0){1}
    	\degrees[4]
     	\rput{0}(1.2;0){\tiny $1$}
    	\rput{0}(1.2;1){\tiny $2$}
    	\rput{0}(1.2;2){\tiny $3$}
	\rput{0}(1.2;3){\tiny $4$}
	\pscurve[linecolor=line,linewidth=1.2pt](1;0)(0.9;0)(-0.5;0.5)(0.9;1)(1;1)
	\pscurve[linecolor=line,linewidth=1.2pt](1;2)(0.9;2)(0.7;2.5)(0.9;3)(1;3)
	\psdot[dotstyle=asterisk,dotscale=1.5](0,0)
   \end{pspicture}\!\!
	+ \!\!\!
	\begin{pspicture}[shift=-1.4](-1.5,-1.5)(1.5,1.5)
     	\SpecialCoor
    	\pscircle[fillstyle=solid, fillcolor=diskin,linewidth=1.2pt](0,0){1}
    	\degrees[4]
     	\rput{0}(1.2;0){\tiny $1$}
    	\rput{0}(1.2;1){\tiny $2$}
    	\rput{0}(1.2;2){\tiny $3$}
	\rput{0}(1.2;3){\tiny $4$}
	\pscurve[linecolor=line,linewidth=1.2pt](1;1)(0.9;1)(0.5;1.5)(0.9;2)(1;2)
	\pscurve[linecolor=line,linewidth=1.2pt](1;3)(0.9;3)(0.5;3.5)(0.9;0)(1;0)
	\psdot[dotstyle=asterisk,dotscale=1.5](0,0)
   \end{pspicture}\!\!\!
	+ t^{\frac{1}{4}}
	\begin{pspicture}[shift=-1.4](-1.5,-1.5)(1.5,1.5)
     	\SpecialCoor
    	\pscircle[fillstyle=solid, fillcolor=diskin,linewidth=1.2pt](0,0){1}
    	\degrees[4]
     	\rput{0}(1.2;0){\tiny $1$}
    	\rput{0}(1.2;1){\tiny $2$}
    	\rput{0}(1.2;2){\tiny $3$}
	\rput{0}(1.2;3){\tiny $4$}
	\pscurve[linecolor=line,linewidth=1.2pt](1;0)(0.9;0)(0.5;0.5)(0.9;1)(1;1)
	\pscurve[linecolor=line,linewidth=1.2pt](1;2)(0.9;2)(0.5;2.5)(0.9;3)(1;3)
	\psdot[dotstyle=asterisk,dotscale=1.5](0,0)
   \end{pspicture}\!\!
	+ t^{-\frac{1}{4}}
	\begin{pspicture}[shift=-1.4](-1.5,-1.5)(1.5,1.5)
      	\SpecialCoor
    	\pscircle[fillstyle=solid, fillcolor=diskin,linewidth=1.2pt](0,0){1}
    	\degrees[4]
     	\rput{0}(1.2;0){\tiny $1$}
    	\rput{0}(1.2;1){\tiny $2$}
    	\rput{0}(1.2;2){\tiny $3$}
	\rput{0}(1.2;3){\tiny $4$}
	\pscurve[linecolor=line,linewidth=1.2pt](1;1)(0.9;1)(-0.5;1.5)(0.9;2)(1;2)
	\pscurve[linecolor=line,linewidth=1.2pt](1;3)(0.9;3)(0.7;3.5)(0.9;0)(1;0)
	\psdot[dotstyle=asterisk,dotscale=1.5](0,0)
   \end{pspicture}
   \end{align*}
   \end{enumerate}
   \end{example}

\begin{remark}
The link pattern tower
$\{(\mathbb{C}[\mathcal{L}_n],\phi_n)\}_{n\in\mathbb{Z}_{\geq 0}}$ with twist weight $v=1$ plays an important role
in the study of the dense loop model on the semi-infinite cylinder \cite{KP,FZJZ,ANS}. The representation space
$\mathbb{C}[\mathcal{L}_n]$ is the state space of the model of system size $n$. 
In \cite{FZJZ} the dense loop model of system size $2k+1$ is related to the dense loop model
of system size $2k$ through the map $\phi_{2k}$. 
The results in this paper allows one to also relate the dense loop model of system size $2k+2$ to the dense loop model
of system size $2k+1$ through the (nontrivial) intertwiner $\phi_{2k+1}$. We will return to this
in \cite{ANS}, in which we also derive recursion relations for associated ground states and for
associated solutions of quantum Knizhnik-Zamolodchikov equations.
\end{remark}

We end the section\footnote{We thank an anonymous referee for pointing us to a possible connection with \cite{RY}.}
 by relating the link pattern tower and the connecting maps to a relative
version of Roger's and Yang's \cite[Def. 2.3]{RY} skein algebra on the punctured disc $\mathbb{D}^*$.
For $n\in\mathbb{Z}_{\geq 0}$ write $\overline{n}\in\{0,1\}$ for the residue of $n$ modulo two and set
\[
s_0:=t^{\frac{1}{4}}+t^{-\frac{1}{4}},\qquad s_1:=1,
\]
which are the noncontractible loop weight and the twist weight of the link pattern tower for $v=1$.
Set
\begin{equation}\label{algebra}
\mathcal{A}:=\bigoplus_{n\in\mathbb{Z}_{\geq 0}}V_n(s_{\overline{n}})
\end{equation}
for the direct sum of the representation spaces of the link pattern tower.
To simplify notations 
we write $\overline{Y}_n$ for the element $Y_{s_{\overline{n}}}\in V_n(s_{\overline{n}})=\textup{Hom}_{\mathcal{S}}(\overline{n},n)\otimes_{\textup{TL}_{\overline{n}}}
\mathbb{C}_{s_{\overline{n}}}$ associated to $Y\in\textup{Hom}_{\mathcal{S}}(\overline{n},n)$.

\begin{proposition}\label{RYconnection}
$\mathcal{A}$ is a graded associative complex algebra with multiplication defined by
\begin{equation*}
\overline{Y}_m\cdot \overline{Z}_n:=
\begin{cases}
(Y\times_{\mathcal{S}}Z)_{m+n},\qquad &\hbox{ if }\, (\overline{m},\overline{n})\not=(1,1),\\
((Y\times_{\mathcal{S}}Z)\circ U)_{m+n},\qquad &\hbox{ if }\, (\overline{m},\overline{n})=(1,1)
\end{cases}
\end{equation*}
for $Y\in\textup{Hom}_{\mathcal{S}}(\overline{m},m)$ and $Z\in\textup{Hom}_{\mathcal{S}}(\overline{n},n)$.
The unit element is $\emptyset_0\in V_0(s_0)$.
\end{proposition}
\begin{proof}
We first show that the product is well-defined. If $(\overline{m},\overline{n})=(0,0)$ then clearly
\[
(\overline{Y\circ X)}_m\cdot \overline{Z}_n=(t^{\frac{1}{4}}+t^{-\frac{1}{4}})\overline{Y}_m\cdot 
\overline{Z}_n=
\overline{Y}_m\cdot (\overline{Z\circ X})_n
\]
since in both the left and right hand side of the equation, the inserted loop around the hole can 
be removed by the scalar factor $t^{\frac{1}{4}}+t^{-\frac{1}{4}}$ using the noncontractible loop removal relation. If $(\overline{m},\overline{n})=(0,1)$ then 
\[(\overline{Y\circ X})_m\cdot \overline{Z}_n=(t^{\frac{1}{4}}+t^{-\frac{1}{4}})\overline{Y}_m\cdot 
\overline{Z}_n
\]
from (the proof of) Proposition \ref{insertionProp}{\bf (a)}, while 
$\overline{Y}_m\cdot (\overline{Z\circ\rho})_n=\overline{Y}_m\cdot \overline{Z}_n$ is a direct consequence of the Dehn twist removal relation
since $v=1$. The case $(\overline{m},\overline{n})=(0,1)$ is checked similarly. For $(\overline{m},\overline{n})=(1,1)$ we have
\begin{equation*}
\begin{split}
(\overline{Y\circ \rho})_m\cdot \overline{Z}_n&=\overline{((Y\circ\rho)\times_{\mathcal{S}}Z)\circ U}_{m+n}\\
&=\overline{(Y\times_{\mathcal{S}}Z)\circ (\rho\times_{\mathcal{S}}\mathbf{1}_1)\circ U}_{m+n}\\
&=\overline{(Y\times_{\mathcal{S}}Z)\circ \mathcal{I}_1(\rho)\circ U}_{m+n}\\
&=\overline{(Y\times_{\mathcal{S}}Z)\circ U}_{m+n}=\overline{Y}_m\cdot \overline{Z}_n,
\end{split}
\end{equation*}
where we used (the proof of) Lemma \ref{Step12}{\bf (ii)} for the fourth equality. In fact, the nontrivial
equality we are using here is
\begin{equation}\label{RIIprime}		
		 \begin{pspicture}[shift=-0.8](-1.5,-1)(1.5,1)
	\SpecialCoor
	\definecolor{line}{rgb}{0,0,1}
	\pscircle[fillstyle=ccslope,slopebegin=diskin,slopeend=diskout,linewidth=1.2pt](0,0){1}
	\pscircle[fillstyle=solid, fillcolor=white, linewidth=1.2pt](0,0){0.35}
	\degrees[8]
	\pscurve[linecolor=line,linewidth=1.2pt](0.35;0)(0.45;0)(0.5;1)(0.6;2)(0.6;3)(0.6;4)(0.6;5)(0.6;6)(0.6;7)(0.8;8)(1;8)
	\psline[linecolor=line,linewidth=1.2pt](1;4)(0.7;4)
			\psline[linecolor=line,linewidth=1.2pt](0.35;4)(0.5;4)
	\rput{0}(1.2;4){\tiny $2$}
	\rput{0}(0;2){\small $U$}
	\rput{0}(1.2;0){\tiny $1$}
\end{pspicture}	
=
 \begin{pspicture}[shift=-0.8](-1.5,-1)(1.5,1)
	\SpecialCoor
	\definecolor{line}{rgb}{0,0,1}
	\pscircle[fillstyle=ccslope,slopebegin=diskin,slopeend=diskout,linewidth=1.2pt](0,0){1}
	\pscircle[fillstyle=solid, fillcolor=white, linewidth=1.2pt](0,0){0.35}
	\degrees[8]
	\psline[linecolor=line,linewidth=1.2pt](1;4)(0.35;4)
			\psline[linecolor=line,linewidth=1.2pt](-0.35;4)(-1;4)
	\rput{0}(1.2;4){\tiny $2$}
	\rput{0}(0;2){\small $U$}
	\rput{0}(1.2;0){\tiny $1$}
\end{pspicture}			
\end{equation}
viewed as an identity in $V_2(t^{\frac{1}{4}}+t^{-\frac{1}{4}})$. In a similar manner, one shows that
$\overline{Y}_m\cdot \overline{Z\circ\rho}_n=\overline{Y}_m\cdot \overline{Z}_n$ if $(\overline{m},\overline{n})=(1,1)$.

Now it remains to show that the product is associative,
\[
(\overline{T}_k\cdot \overline{Y}_m)\cdot \overline{Z}_n=\overline{T}_k\cdot(\overline{Y}_m\cdot 
\overline{Z}_n)
\]
for $T\in\textup{Hom}_{\mathcal{S}}(\overline{k},k)$,
$Y\in\textup{Hom}_{\mathcal{S}}(\overline{m},m)$ and 
$Z\in\textup{Hom}_{\mathcal{S}}(\overline{n},n)$. The only nontrivial case is
$(\overline{k},\overline{m},\overline{n})=(1,1,1)$. Then we have
\begin{equation*}
\begin{split}
(\overline{T}_k\cdot \overline{Y}_m)\cdot \overline{Z}_n&=\overline{((T\times_{\mathcal{S}}Y)\circ U)\times_{\mathcal{S}}
Z}_{k+m+n}\\
&=\overline{((T\times_{\mathcal{S}}Y)\times_{\mathcal{S}}Z)\circ (U\times_{\mathcal{S}}\mathbf{1}_1)
}_{k+m+n}\\
&=\overline{(T\times_{\mathcal{S}}(Y\times_{\mathcal{S}}Z))\circ (\mathbf{1}_1\times_{\mathcal{S}}
U)}_{k+m+n}\\
&=\overline{T}_k\cdot(\overline{Y}_m\cdot \overline{Z}_n),
\end{split}
\end{equation*}
where we used in the third equality that $\overline{U \times_{\mathcal{S}} \mathbf{1}_1} = 
\overline{\mathbf{1}_1\times_{\mathcal{S}} U}$  in $V_3(1)$.
This follows from a direct calculation in the skein, showing that both sides of the equation are equal to
$$ 
t^{\frac{1}{4}} \; 
\begin{pspicture}[shift=-0.8](-1,-1)(1,1)
    	\SpecialCoor
	\pscircle[fillstyle=solid, fillcolor=diskin,linewidth=1.2pt](0,0){1}
	\pscircle[fillstyle=solid, fillcolor=white, linewidth=1.2pt](0,0){0.25}
    	\degrees[3]
     	\rput{0}(1.2;0){\tiny $1$}
    	\rput{0}(1.2;1){\tiny $2$}
    	\rput{0}(1.2;2){\tiny $3$}
	\pscurve[linecolor=line,linewidth=1.2pt](1;0)(0.7;0) (0.25;0)
	\pscurve[linecolor=line,linewidth=1.2pt](1;1)(0.9;1)(0.6;1.5)(0.9;2)(1;2)
   \end{pspicture} \quad
+ t^{\frac{1}{4}} \;
       \begin{pspicture}[shift=-0.8](-1,-1)(1,1)
    	\SpecialCoor
    	\pscircle[fillstyle=solid, fillcolor=diskin,linewidth=1.2pt](0,0){1}
	\pscircle[fillstyle=solid, fillcolor=white, linewidth=1.2pt](0,0){0.25}
    	\degrees[3]
     	\rput{0}(1.2;0){\tiny $1$}
    	\rput{0}(1.2;1){\tiny $2$}
    	\rput{0}(1.2;2){\tiny $3$}
	\pscurve[linecolor=line,linewidth=1.2pt](1;2)(0.9;2)(0.45;1.5)(0.45;1)(0.4;0.5)(0.35;0)(0.25;0)
	\pscurve[linecolor=line,linewidth=1.2pt](1;0)(0.9;0)(0.6;0.5)(0.9;1)(1;1)
   \end{pspicture} \quad
+t^{-\frac{1}{4}} \;  
     \begin{pspicture}[shift=-0.8](-1,-1)(1,1)
    	\SpecialCoor
	\pscircle[fillstyle=solid, fillcolor=diskin,linewidth=1.2pt](0,0){1}
	\pscircle[fillstyle=solid, fillcolor=white, linewidth=1.2pt](0,0){0.25}
    	\degrees[3]
     	\rput{0}(1.2;0){\tiny $1$}
    	\rput{0}(1.2;1){\tiny $2$}
    	\rput{0}(1.2;2){\tiny $3$}
	\pscurve[linecolor=line,linewidth=1.2pt](1;1)(0.9;1)(0.5;0.5)(0.35;0)(0.25;0)
	\pscurve[linecolor=line,linewidth=1.2pt](1;2)(0.9;2)(0.6;2.5)(0.9;3)(1;3)
   \end{pspicture}   
 $$
when viewed as identity in $V_3(1)$ .
\end{proof}
\begin{corollary}\label{corRY}
Let $v=1$. The connecting map $\phi_n: V_n(s_{\overline{n}})\rightarrow V_{n+1}(s_{\overline{n+1}})$
of the link pattern tower is given by
\[
\phi_n(\overline{Y}_n)=\overline{Y}_n\cdot\overline{\mathbf{1}}_1
\]
for $\overline{Y}_n\in V_n(s_{\overline{n}})$.
\end{corollary}
\begin{remark} \label{rogeryang}
We wish to point out the connection between Proposition \ref{RYconnection} and the 
work of Roger and Yang \cite{RY}. We view the representation space $V_n(s_{\overline{n}})$
of the link pattern tower with $v=1$ as the following relative version of 
Roger's and Yang's \cite[Def. 2.3]{RY} skein algebra of arcs and links on 
$\Sigma=\mathbb{D}^*$ with puncture $V:=\{0\}$. In the relative version we consider, besides the puncture, also the set $\{2\xi_n^{j-1} \,\, | \,\, 1\leq j\leq n\}$ of $n$ marked points on the outer boundary of $\mathbb{D}^*$. The associated relative skein module $\mathcal{M}_n$ is a quotient of the vector space generated by the isotopy classes of framed arcs and links in $\mathbb{D}^*\times [0,1]$ such that each pole 
$\{2\xi_{n}^{j-1}\}\times [0,1]$ is met by exactly one endpoint ($1\leq j\leq n$), and 
multiple endpoints may connect to the internal pole $\{0\}\times [0,1]$ at different heights.
The quotient is the linear and transitive closure of the Kauffman skein relation \eqref{kauffman}, the nullhomotopic loop removal relation \eqref{loopremoval}, the noncontractible loop removal relation
\eqref{nonloopremoval} for $v=1$ with the hole shrunk to the puncture (it is called the puncture-framing relation in \cite{RY}), 
and finally the Roger-Yang {\it puncture-skein relation} 
\begin{equation}\label{psr}
\begin{pspicture}[shift=-1.1](-1.5,-1.25)(1.5,1.25)
	\SpecialCoor
	\pscircle[fillstyle=solid,fillcolor=diskin,linewidth=0.8pt,linestyle=dotted, ](0,0){1}
	\degrees[8]
	\psline[linecolor=line,linewidth=1.2pt](1;0)(0;0)
	\psline[linecolor=line,linewidth=1.2pt](1;4)(0.25;4)	
	\psdot[dotstyle=asterisk,dotscale=1.5](0,0)
\end{pspicture}
 =  t^{\frac{1}{8}}\left( t^{\frac{1}{8}}
 \begin{pspicture}[shift=-1.1](-1.5,-1.25)(1.5,1.25)
	\SpecialCoor
	\pscircle[fillstyle=solid,fillcolor=diskin,linewidth=0.8pt,linestyle=dotted, ](0,0){1}
	\degrees[8]
	\psdot[dotstyle=asterisk,dotscale=1.5](0,0)
	\pscurve[linecolor=line,linewidth=1.2pt](1;0)(0.9;0)(0.7;1.5)(0.7;2.5)(0.9;4)(1;4)
\end{pspicture}
+  t^{-\frac{1}{8}}
 \begin{pspicture}[shift=-1.1](-1.5,-1.25)(1.5,1.25)
	\SpecialCoor
	\pscircle[fillstyle=solid,fillcolor=diskin,linewidth=0.8pt,linestyle=dotted, ](0,0){1}
	\degrees[8]
	\psdot[dotstyle=asterisk,dotscale=1.5](0,0)
	\pscurve[linecolor=line,linewidth=1.2pt](1;0)(0.9;0)(0.7;6.5)(0.7;5.5)(0.9;4)(1;4)
\end{pspicture}
\right )
\end{equation}
where at the left, the right curve lies above the left when meeting at the internal pole $\{0\}\times [0,1]$ (the parameters 
$v,q^{\frac{1}{2}}$ in \cite{RY} is set to the specific values $t^{-\frac{1}{8}},t^{-\frac{1}{8}}$ to match up with our conventions). Then we have a natural linear isomorphism
$V_n(s_{\overline{n}})\simeq \mathcal{M}_n$ such that 
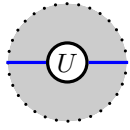
\[
\begin{pspicture}[shift=-1.1](-1.5,-1.25)(1.5,1.25)
	\SpecialCoor
	\pscircle[fillstyle=ccslope,slopebegin=diskin,slopeend=diskout,linewidth=1.2pt,linestyle=dotted, ](0,0){1}
	\pscircle[fillstyle=solid, fillcolor=white, linewidth=1.2pt](0,0){0.35}
	\degrees[8]
	\psline[linecolor=line,linewidth=1.2pt](1;0)(0.35;0)
	\psline[linecolor=line,linewidth=1.2pt](1;4)(0.35;4)
	\rput{0}(0;2){\small $U$}
\end{pspicture}
\]
in $V_{2k}(s_{\overline{2k}})$
corresponds to the left hand side of \eqref{psr} in $\mathcal{M}_{2k}$. With this identification our graded algebra structure on
\[
\mathcal{A}=\bigoplus_{n=0}^{\infty}V_n(s_{\overline{n}})\simeq
\bigoplus_{n=0}^{\infty}\mathcal{M}_n
\]
is the natural relative version of the skein algebra multiplication (cf. the definition of the tensor functor
$\times_{\mathcal{S}}$ from Section \ref{section3}). Note that under this identification 
\eqref{RIIprime} is one of the Reidemeister II' relations from \cite{RY}, and the proof of Proposition
\ref{RYconnection} is a direct generalization of the proof of \cite[Thm. 2.4]{RY}. In fact, 
$\mathcal{M}_0\simeq\mathbb{C}$ is the Roger-Yang skein algebra of arcs and links on $\mathbb{D}^*$ 
with puncture $V=\{0\}$. 

By Corollary \ref{corRY}, we can describe $\phi_{n}$ on $\mathcal{M}_{n}$
as inserting an arc connecting the punctured cylinder to the pole $\{0\}\times [0,1]$ that passes underneath all other arcs and links and then modding out by the Kauffman skein, loop removal and non-contractible loop removal relations, as well as the puncture-skein relation.
\end{remark}

\section{The link pattern tower and fusion}\label{section9}
The algebra maps $\epsilon_{n,m}$ (see Corollary \ref{algemaps})
are used in \cite{GS} to define the following fusion product of extended affine Temperley-Lieb modules. 
\begin{definition}[\cite{GS}]
The fusion product of a left 
$\textup{TL}_n$-module $M_1$ and a left $\textup{TL}_m$-module $M_2$
is the left $\textup{TL}_{n+m}$-module 
\[
M_1\widehat{\times}_fM_2:=\textup{Ind}^{\epsilon_{n,m}}\bigl(M_1\otimes M_2).
\]
\end{definition}
We will show that the consecutive constituents $V_{2k}(u)$ and $V_{2k+1}(v)$ (respectively
$V_{2k+1}(v)$ and $V_{2k+2}(u)$) in the link pattern tower are naturally related by fusion with 
$\textup{TL}_1$-modules. An important role in the analysis is played by the element
$d_n\in\textup{End}_{\mathcal{S}}(n)$ ($n\geq 1$), defined by
\[
d_{n}:=\epsilon_{n-1,1}(\mathbf{1}_{n-1}\otimes\rho)=\mathbf{1}_{n-1}\times_{\mathcal{S}}\rho
\]
with $\rho\in\textup{End}_{\mathcal{S}}(1)$ given by \eqref{Dehntwistrho}. Note that $d_n$
is the skein class of the $(n,n)$-tangle diagram

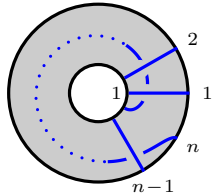
\begin{center}
{\psset{unit=.8cm}
 \begin{pspicture}(-2,-2)(2,2)
    	\pscircle[fillstyle=solid,fillcolor=diskin,linewidth=1.2pt](0,0){1.5}
    	\pscircle[fillstyle=solid, fillcolor=white, linewidth=1.2pt](0,0){0.5}
	\degrees[12]
    \rput{-1}(0;0){ \psplot[algebraic, plotpoints=400,linecolor=line,polarplot=true,linewidth=1.2pt]{0}{0.5 Pi mul }{(2/Pi)*ASIN(x/Pi-1) +(3*Pi-x)/(2*Pi)}}
      \rput{-1}(0;0){ \psplot[algebraic, plotpoints=400,linecolor=line,linestyle=dotted,polarplot=true,linewidth=1.2pt]{0.51 Pi mul}{1.7 Pi mul }{(2/Pi)*ASIN(x/Pi-1) +(3*Pi-x)/(2*Pi)}}
        \rput{-1}(0;0){ \psplot[algebraic, plotpoints=400,linecolor=line,polarplot=true,linewidth=1.2pt]{1.71 Pi mul}{2 Pi mul }{(2/Pi)*ASIN(x/Pi-1) +(3*Pi-x)/(2*Pi)}}
          \psdot[linecolor=diskin,dotsize=0.3](0.8;0)
          \psdot[linecolor=diskin,dotsize=0.3](0.8;1)
           \psdot[linecolor=diskin,dotsize=0.3](1.2;-2)
    \psline[linecolor=line,linewidth=1.2pt](1.5;0)(0.5;0)
    \psline[linecolor=line,linewidth=1.2pt](1.5;1)(0.5;1)
    \psline[linecolor=line,linewidth=1.2pt](1.5;-2)(0.5;-2)
    	\rput{0}(1.8;0){\tiny$1$}
	\rput{0}(1.8;1){\tiny$2$}
	\rput{0}(1.8;-2){\tiny$n\!-\!1$}
	\rput{0}(1.8;-1){\tiny$n$}
	\rput{0}(0.3;0){\tiny$1$}
	
   \end{pspicture}  
}
\end{center}
Furthermore, $d_n\in\textup{TL}_n$ is invertible and 
\begin{equation}\label{dintertwiner}
d_{j+1}\circ\mathcal{I}(Z)=Z\times_{\mathcal{S}}\rho=
\mathcal{I}(Z)\circ d_{i+1}\qquad \forall\, Z\in\textup{Hom}_{\mathcal{S}}(i,j).
\end{equation}
In particular, $d_{n}$ lies in the centralizer of $\mathcal{I}_{n-1}(\textup{TL}_{n-1})$ in 
$\textup{TL}_{n}$.
In terms of the algebraic generators of 
$\textup{TL}_n\simeq\textup{End}_{\mathcal{S}}(n)$, the element $d_{n}$ can be expressed as
\[
d_n=(t^{\frac{1}{4}}e_{n-1}+t^{-\frac{1}{4}})\cdots (t^{\frac{1}{4}}e_1+t^{-\frac{1}{4}})\rho.
\]

Consider now the link pattern tower
\[
V_0(u)\overset{\phi_0}{\longrightarrow}V_1(v)\overset{\phi_1}{\longrightarrow}
V_2(u)\overset{\phi_2}{\longrightarrow} V_3(v)\overset{\phi_3}{\longrightarrow}\cdots
\]
where $u=t^{\frac{1}{4}}v+t^{-\frac{1}{4}}v^{-1}$. Consider the surjective intertwiners
\begin{equation*}
\begin{split}
\pi_{2k}: \textup{Ind}^{\mathcal{I}_{2k}}\bigl(V_{2k}(u)\bigr)&\twoheadrightarrow
V_{2k}(u)\widehat{\times}_fV_1(v),\\
\pi_{2k+1}: \textup{Ind}^{\mathcal{I}_{2k+1}}(V_{2k+1}(v))&\twoheadrightarrow
V_{2k+1}(v)\widehat{\times}_fV_1(v^{-1}),
\end{split}
\end{equation*}
of $\textup{TL}_{2k+1}$-modules and $\textup{TL}_{2k+2}$-modules respectively, defined by 
\begin{equation*}
\begin{split}
\pi_{2k}\big(Y\otimes_{\textup{TL}_{2k}}w_{2k})&:=Y\otimes_{\textup{TL}_{2k}\otimes\textup{TL}_1}
(w_{2k}\otimes 1_v),\\
\pi_{2k+1}\bigl(Z\otimes_{\textup{TL}_{2k+1}}w_{2k+1})&:=Z\otimes_{\textup{TL}_{2k+1}\otimes\textup{TL}_1}
(w_{2k+1}\otimes 1_{v^{-1}})
\end{split}
\end{equation*}
for $Y\in\textup{TL}_{2k+1}$, $w_{2k}\in V_{2k}(u)$ and $Z\in\textup{TL}_{2k+2}$, $w_{2k+1}\in V_{2k+1}(v)$.
\begin{proposition} Let $u=t^{\frac{1}{4}}v+t^{-\frac{1}{4}}v^{-1}$. Let 
\[
V_0(u)\overset{\phi_0}{\longrightarrow}V_1(v)\overset{\phi_1}{\longrightarrow}
V_2(u)\overset{\phi_2}{\longrightarrow} V_3(v)\overset{\phi_3}{\longrightarrow}\cdots
\]
be the link pattern tower. Then the intertwiner $\widehat{\phi}_n$ factors through $\pi_n$. In other words,
there exist unique intertwiners
\begin{equation*}
\begin{split}
\psi_{2k}:\, &V_{2k}(u)\widehat{\times}_fV_1(v)\longrightarrow V_{2k+1}(v),\\
\psi_{2k+1}:\, &V_{2k+1}(v)\widehat{\times}_fV_1(v^{-1})\longrightarrow V_{2k+2}(u)
\end{split}
\end{equation*}
of $\textup{TL}_{2k+1}$-modules and $\textup{TL}_{2k+2}$-modules respectively,
such that $\psi_n\circ\pi_n=\widehat{\phi}_n$ for all $n\in\mathbb{Z}_{\geq 0}$.
\end{proposition}
\begin{proof}
We need to show that 
\begin{equation*}
\begin{split}
\psi_{2k}\bigl(Y\otimes_{\textup{TL}_{2k}\otimes\textup{TL}_1}(w_{2k}\otimes\mathbf{1}_v)\bigr)&:=
Y\phi_{2k}(w_{2k}),\\
\psi_{2k+1}\bigl(Z\otimes_{\textup{TL}_{2k+1}\otimes\textup{TL}_1}(w_{2k+1}\otimes\mathbf{1}_{v^{-1}})\bigr)&:=Z\phi_{2k+1}(w_{2k+1})
\end{split}
\end{equation*}
for $Y\in\textup{TL}_{2k+1}$, $w_{2k}\in V_{2k}(u)$ and $Z\in\textup{TL}_{2k+2}$, $w_{2k+1}\in V_{2k+1}(v)$ are well-defined linear maps. The balancing condition of the tensor product for the first tensor component of the algebra $\textup{TL}_n\otimes\textup{TL}_1$ is respected because of the intertwining properties of $\phi_{2k}$ and $\phi_{2k+1}$. For instance, for $X\in\textup{TL}_{2k}$,
\[
Y\epsilon_{2k,1}(X\otimes\mathbf{1})\phi_{2k}(w_{2k})=Y\mathcal{I}_{2k}(X)\phi_{2k}(w_{2k})=
Y\phi_{2k}(Xw_{2k}).
\]
For the balancing condition of the
tensor product for the second tensor component of $\textup{TL}_n\otimes\textup{TL}_1$
we need to show that
\begin{equation}\label{toprove}
\begin{split}
d_{2k+1}\phi_{2k}(w_{2k})&=v\,\phi_{2k}(w_{2k}),\\
d_{2k+2}\phi_{2k+1}(w_{2k+1})&=v^{-1}\phi_{2k+1}(w_{2k+1})
\end{split}
\end{equation}
 for all $w_{2k}\in V_{2k}(u)$ and $w_{2k+1}\in V_{2k+1}(v)$.
Write $w_{2k}=Y_u$ with $Y\in\textup{Hom}_{\cS}(0,2k)$ and $w_{2k+1}=Z_v$
with $Z\in\textup{Hom}_{\cS}(1,2k+1)$. Then
\[
d_{2k+1}\phi_{2k}(Y_u)=\bigl(d_{2k+1}\mathcal{I}(Y)\bigr)_v=\bigl(\mathcal{I}(Y)\rho\bigr)_v=v\,\phi_{2k}(Y_u),
\]
where we used \eqref{dintertwiner} and the fact that $d_1=\rho\in\textup{TL}_1$ for the second equality. To prove the second equality of \eqref{toprove}, first note that
\[
d_{2k+2}\phi_{2k+1}(Z_v)=\bigl(d_{2k+2}\mathcal{I}(Z)U\bigr)_u
=\bigl(\mathcal{I}(Z)d_2U\bigr)_u
\]
by \eqref{dintertwiner}. Now $d_2=(t^{\frac{1}{4}}e_1+t^{-\frac{1}{4}})\rho=
\mathcal{I}_1(\rho^{-1})\rho^2$ in $\textup{TL}_2$ by Proposition \ref{insertionProp}{\bf (b)}.
Futhermore we have $\rho^2U=U$ in $\textup{Hom}_{\cS}(0,2)$, so we conclude that
\[
d_{2k+2}\phi_{2k+1}(Z_v)=\bigl(\mathcal{I}(Z)\mathcal{I}_1(\rho^{-1})U\bigr)_u=
v^{-1}\bigl(\mathcal{I}(Z)U\bigr)_u=v^{-1}\phi_{2k+1}(Z_v),
\]
where the second step follows from the proof of Lemma \ref{Step12}.
\end{proof}
\begin{corollary} 
Let $u=t^{\frac{1}{4}}v+t^{-\frac{1}{4}}v^{-1}$ with $v^2\not=t^{\frac{1}{2}}$. Let 
\[
V_0(u)\overset{\phi_0}{\longrightarrow}V_1(v)\overset{\phi_1}{\longrightarrow}
V_2(u)\overset{\phi_2}{\longrightarrow} V_3(v)\overset{\phi_3}{\longrightarrow}\cdots
\]
be the link pattern tower. The left $\textup{TL}_{2k+1}$-module $V_{2k+1}(v)$ is a quotient of 
$V_{2k}(u)\widehat{\times}_fV_1(v)$ and the left $\textup{TL}_{2k+2}$-module $V_{2k+2}(u)$ is a quotient of
$V_{2k+1}(v)\widehat{\times}_fV_1(v^{-1})$ for all $k\geq 0$.
\end{corollary}
\begin{proof}
By Theorem \ref{intertwinerstower}{\bf (ii)} the link pattern tower is nondegenerate,
i.e. the $\widehat{\phi}_n$ are surjective for all $n\in\mathbb{Z}_{\geq 0}$. By the previous
proposition we conclude that the intertwiners $\psi_n$ are surjective for all $n\in\mathbb{Z}_{\geq 0}$.
\end{proof}
\begin{example}
Let $u=t^{\frac{1}{4}}v+t^{-\frac{1}{4}}v^{-1}$ with $v^2\not=t^{\frac{1}{2}}$.
Recall the identification of 
\[
V_{2k}(u)=\mathcal{W}_{0,t^{\frac{1}{4}}v}[2k],\qquad
V_{2k+1}(v)=\mathcal{W}_{\frac{1}{2},v}[2k+1]
\]
with the standard modules from \cite{GS}
(see Remark \ref{standardremark}).
Then \cite[(4.26)]{GS} shows that 
\[
V_1(v)\widehat{\times}_fV_1(v^{-1})\simeq V_2(u).
\]
\end{example}
\appendix
\section{Relation to affine Hecke algebras and affine braid groups, and type $B$ presentations}\label{section10}
We first show how the algebra maps 
$\mathcal{I}_n$ can be lifted to extended affine Hecke algebras and to the group algebras of extended affine braid groups. We give the constructions below for $n\geq 3$. The adjustments needed for $n=1,2$ are left to the reader as long as they are obvious.

The affine Temperley-Lieb algebra $\overline{\textup{TL}}_n$ of
type $\widehat{A}_{n-1}$ is the subalgebra of $\textup{TL}_n$ generated by $e_1,e_2,\ldots,e_{n}$, see \cite{FG}. The defining relations of $\overline{\textup{TL}}_n$ are given by the first three lines in \eqref{relTL}. Note that $\mathbb{Z}$ acts on $\overline{\textup{TL}}_n$ by 
algebra automorphisms with $m\in\mathbb{Z}$ acting by  $e_i\mapsto e_{i+m}$ (with the indices modulo $n$). Let
$\textup{TL}^e_n$ be the corresponding crossed product algebra $\mathbb{Z}\ltimes\overline{\textup{TL}}_n$. Note that $\textup{TL}^e_n$ is isomorphic to the algebra generated
by $e_1,\ldots, e_{n},\rho^{\pm 1}$ with defining relations all but the last relation in \eqref{relTL}. It follows that
\[
\textup{TL}_n\simeq\textup{TL}_n^e/\langle \rho^2e_{n-1}-e_1e_2\cdots e_{n-1}\rangle
\]
with $\langle \rho^2e_{n-1}-e_1e_2\cdots e_{n-1}\rangle$ the two-sided ideal generated by
$ \rho^2e_{n-1}-e_1e_2\cdots e_{n-1}$.

In \cite{FG} the affine Temperley-Lieb algebra $\overline{\textup{TL}}_n$ is realized as a quotient of the affine Hecke algebra of type $\widehat{A}_{n-1}$. We recall this here, and give the extension to $\textup{TL}_n^e$.

\begin{definition}
The {\textup{extended affine Hecke algebra}} $H_n$ of type $\widehat{A}_{n-1}$ is the unital complex associative algebra with generators $T_1,T_2,\ldots,T_{n},\rho,\rho^{-1}$ and defining relations
\begin{equation}\label{relH}
\begin{split}
	 &(T_i-t^{-\frac{1}{2}})(T_i+t^{\frac{1}{2}}) = 0,\\
	 &T_iT_j = T_jT_i\qquad\qquad\quad \hbox{ if }\,\, i-j\not=\pm 1,\\
	 &T_iT_{i+1}T_i=T_{i+1}T_iT_{i+1},\\
	 &\rho T_i =  T_{i+1} \rho, \\
	&\rho\rho^{-1}=1=\rho^{-1}\rho,
\end{split}
\end{equation}
where the indices are taken modulo $n$. 
\end{definition}
Note that $T_i\in H_n$ is invertible with inverse
$T_i^{-1}=T_i-t^{-\frac{1}{2}}+t^{\frac{1}{2}}$. The affine Hecke algebra of type 
$\widehat{A}_{n-1}$ is the subalgebra $\overline{H}_n$ of $H_n$ generated by  
$T_1,T_2,\ldots,T_{n}$. The defining relations of $\overline{H}_n$ are given by the first three lines in \eqref{relH}. The extended affine Hecke algebra $H_n$ is isomorphic to the crossed product algebra
$\mathbb{Z}\ltimes\overline{H}_n$, where $m\in\mathbb{Z}$ acts on 
$\overline{H}_n$
by the algebra automorphism $T_i\mapsto T_{i+m}$ (with the indices modulo $n$).
\begin{proposition}
There exists a unique surjective algebra map
$\psi_n: H_n\rightarrow \textup{TL}_n^e$ satisfying $\rho\mapsto \rho$ and $T_i\mapsto e_i+t^{-\frac{1}{2}}$. The kernel of $\psi_n$ is the two-sided ideal in $H_n$ generated by the elements
\begin{equation}\label{kernel}
T_iT_{i+1}T_i-t^{-\frac{1}{2}}T_iT_{i+1}-t^{-\frac{1}{2}}T_{i+1}T_i+t^{-1}T_i+t^{-1}T_{i+1}-t^{-\frac{3}{2}},\qquad
i\in\mathbb{Z}/n\mathbb{Z}.
\end{equation}
\end{proposition}
\begin{proof}
Fan and Green \cite{FG} showed that the kernel of the unique surjective algebra map 
$\overline{\psi}_n: \overline{H}_n\rightarrow\overline{\textup{TL}}_n$ satisfying
$T_i\mapsto e_i+t^{-\frac{1}{2}}$ is generated by
the elements
\eqref{kernel} (see also \cite{GL2}). The proposition now follows since the $\mathbb{Z}$-actions on $\overline{H}_n$
and $\overline{\textup{TL}}_n$ are intertwined by $\overline{\psi}_n$.
\end{proof}

The extended affine braid group
$\mathcal{B}_n$ is the group generated by $\sigma_1,\sigma_2,\ldots,\sigma_{n},\widetilde{\rho}$ with defining relations
\begin{equation}\label{relB}
\begin{split}
	 &\sigma_i\sigma_j = \sigma_j\sigma_i\qquad\qquad\quad \hbox{ if }\,\, i-j\not=\pm 1,\\
	 &\sigma_i\sigma_{i+1}\sigma_i=\sigma_{i+1}\sigma_i\sigma_{i+1},\\
	 &\widetilde{\rho} \sigma_i =  \sigma_{i+1} \widetilde{\rho}, 
\end{split}
\end{equation}
where the indices are taken modulo $n$, see e.g. \cite{GL2}. Recall that $\mathcal{B}_n$
can be realized topologically in terms of $n$ strands in $\mathbb{C}^\ast\times [0,1]$ starting at
$\{(2\xi_n^{j-1},0)\}_{j=1}^n$ and ending at $\{(2\xi_n^{j-1},1)\}_{j=1}^n$,\\

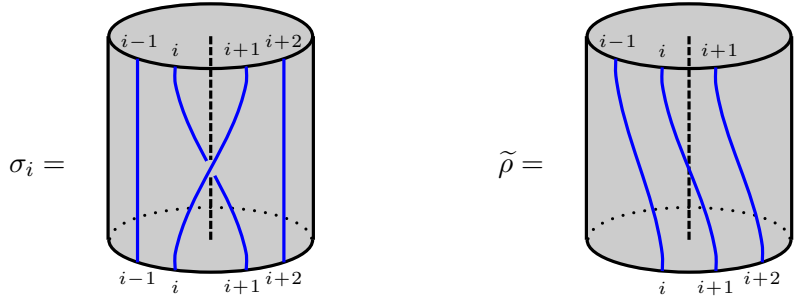
\begin{align*}
\psset{unit=0.45cm}
\sigma_i = 
\begin{pspicture}[shift=-4](8,9.5)
	\pssavepath[linecolor=white]{circle}{\psarc(4,8){3.025}{180}{0}}
        \SpecialCoor
        \pnode(4,8){O}
        \pssavepath{line1}{\psline(O)([angle=-135,nodesep=3.5]O)}
        \pssavepath{line2}{\psline(O)([angle=-110,nodesep=3.3]O)}
        \pssavepath{line3}{\psline(O)([angle=-70,nodesep=3.3]O)}
        \pssavepath{line4}{\psline(O)([angle=-45,nodesep=3.3]O)}
        \psintersect[showpoints,name=p1]{circle}{line1}
        \psintersect[showpoints,name=p2]{circle}{line2}
        \psintersect[showpoints,name=p3]{circle}{line3}
        \psintersect[showpoints,name=p4]{circle}{line4}
        \pssavepath[linewidth=0.01pt,linecolor=white]{line}{\psline([angle=-90,nodesep=5]p11)([angle=90,nodesep=2]p11)([angle=90,nodesep=2.5]p21)([angle=-90,nodesep=4.2]p21)([angle=-90,nodesep=4.2]p31)([angle=90,nodesep=2.5]p31)([angle=90,nodesep=2]p41)([angle=-90,nodesep=5]p41)}
	\psframe*[linecolor=diskin](1,2)(7,8)
	\psellipse*[linecolor=diskin](4,8)(3,1)
	\psellipse*[linecolor=diskin](4,2)(3,1)
	\psline[linewidth = 1.2pt](1,2)(1,8)
	\pssavepath[linewidth=1.2pt]{top}{\psellipse[linewidth=1.2pt](4,8)(3.025,1)}
	\pssavepath[linewidth=1.2pt]{bottom}{\psellipticarc[linewidth=1.2pt](4,2)(3.025,1){180}{0}}
	\psellipticarc[linestyle=dotted,linewidth=1.2pt](4,2)(3,1){0}{180}
	\psline[linewidth=1.2pt](7,2)(7,8)
	\psline[linewidth=1.2pt,linestyle=dashed,dash=3pt 1pt](4,2)(4,8)
  	\psintersect[name=ti]{top}{line}
	\psintersect[name=bi]{bottom}{line}
	\psline[linecolor= line,linewidth=1.2pt](ti1)(bi1)
	\psline[linecolor= line,linewidth=1.2pt](ti4)(bi4)
	\pssavepath[linecolor= line,linewidth=1.2pt]{under}{\pscurve[linecolor= line,linewidth=1.2pt](ti2)([angle=-90,nodesep=0.5]ti2)([angle=90,nodesep=0.5]bi3)(bi3)}
	\pssavepath[linecolor= diskin,linewidth=0.1pt]{over}{\pscurve[linecolor= line,linewidth=1.2pt](ti3)([angle=-90,nodesep=0.5]ti3)([angle=90,nodesep=0.5]bi2)(bi2)}
	\psintersect[name=x]{over}{under}
	\psdot[linecolor=diskin,dotsize=0.5](x1)
	\pscurve[linecolor= line,linewidth=1.2pt](ti3)([angle=-90,nodesep=0.5]ti3)([angle=90,nodesep=0.5]bi2)(bi2)
	\rput{0}([angle=-90,nodesep=0.5]bi1){\tiny $i\!-\!1$}
	\rput{0}([angle=90,nodesep=0.5]ti1){\tiny $\,i\!-\!1$}	
	\rput{0}([angle=-90,nodesep=0.5]bi2){\tiny $i$}
	\rput{0}([angle=90,nodesep=0.5]ti2){\tiny $i$}
	\rput{0}([angle=-90,nodesep=0.5]bi3){\tiny $i\!+\!1\,\,$}
	\rput{0}([angle=90,nodesep=0.5]ti3){\tiny $i\!+\!1\,\,$}
	\rput{0}([angle=-90,nodesep=0.5]bi4){\tiny $i\!+\!2$}
	\rput{0}([angle=90,nodesep=0.5]ti4){\tiny $i\!+\!2$}
	\end{pspicture}
&& \psset{unit=0.45cm}
\widetilde{\rho} =
	\begin{pspicture}[shift=-4](8,9.5)
	\pssavepath[linecolor=white]{circle}{\psarc(4,8){3.025}{180}{0}}
        \SpecialCoor
        \pnode(4,8){O}
        \pssavepath{line1}{\psline(O)([angle=-135,nodesep=3.5]O)}
        \pssavepath{line2}{\psline(O)([angle=-105,nodesep=3.3]O)}
        \pssavepath{line3}{\psline(O)([angle=-75,nodesep=3.3]O)}
        \pssavepath{line4}{\psline(O)([angle=-45,nodesep=3.3]O)}
        \psintersect[showpoints,name=p1]{circle}{line1}
        \psintersect[showpoints,name=p2]{circle}{line2}
        \psintersect[showpoints,name=p3]{circle}{line3}
        \psintersect[showpoints,name=p4]{circle}{line4}
        \pssavepath[linewidth=0.01pt,linecolor=white]{line5}{\psline([angle=-90,nodesep=5]p11)([angle=90,nodesep=2]p11)([angle=90,nodesep=2.5]p21)([angle=-90,nodesep=4.2]p21)([angle=-90,nodesep=4.2]p31)([angle=90,nodesep=2.5]p31)([angle=90,nodesep=2]p41)([angle=-90,nodesep=5]p41)}
         \pssavepath[linewidth=0.01pt,linecolor=white]{linefor3}{\psline([angle=-90,nodesep=4.2]p31)([angle=90,nodesep=2.5]p31)}
	\psframe*[linecolor=diskin](1,2)(7,8)
	\psellipse*[linecolor=diskin](4,8)(3,1)
	\psellipse*[linecolor=diskin](4,2)(3,1)
	\psline[linewidth = 1.2pt](1,2)(1,8)
	\pssavepath[linewidth=1.2pt]{top2}{\psellipse[linewidth=1.2pt](4,8)(3.025,1)}
	\pssavepath[linewidth=1.2pt]{bottom}{\psellipticarc[linewidth=1.2pt](4,2)(3.025,1){180}{0}}
	\psellipticarc[linestyle=dotted,linewidth=1.2pt](4,2)(3,1){0}{180}
	\psline[linewidth=1.2pt](7,2)(7,8)
	\psline[linewidth=1.2pt,linestyle=dashed,dash =3pt 1pt](4,2)(4,8)
  	\psintersect[name=t2i]{top2}{line5}
	\psintersect[name=t3i]{top2}{linefor3}
	\psintersect[name=bi]{bottom}{line5}
	\pscurve[linecolor= line,linewidth=1.2pt](t2i2)([angle=-90,nodesep=0.5]t2i2)([angle=90,nodesep=0.5]bi3)(bi3)
	\pscurve[linecolor= line,linewidth=1.2pt](t2i1)([angle=-90,nodesep=0.5]t2i1)([angle=90,nodesep=0.5]bi2)(bi2)
	\pscurve[linecolor= line,linewidth=1.2pt]([angle=90,nodesep=6]bi3)([angle=90,nodesep=5.5]bi3)([angle=90,nodesep=0.5]bi4)(bi4)
	\rput{0}([angle=90,nodesep=0.5]t2i1){\tiny $\,i\!-\!1$}	
	\rput{0}([angle=-90,nodesep=0.5]bi2){\tiny $i$}
	\rput{0}([angle=90,nodesep=0.5]t2i2){\tiny $i$}
	\rput{0}([angle=-90,nodesep=0.5]bi3){\tiny $i\!+\!1$}
	\rput{0}([angle=89,nodesep=6.5]bi3){\tiny $i\!+\!1$}
	\rput{0}([angle=-90,nodesep=0.5]bi4){\tiny $i\!+\!2$}
	\end{pspicture}
\end{align*}

Given a braid in $\mathcal{B}_n$, project it onto the
cylinder $C_o\times [0,1]$ and map $C_o\times [0,1]$ homeomorphically onto $A$ by collapsing
the wall of the cylinder inwards onto $A\times\{0\}$. This results in an $(n,n)$-tangle diagram in $A$,
which we subsequently interpret as an element in the linear skein $\textup{End}_\cS(n)$.
This defines a surjective algebra map $\mu_n: \mathbb{C}[\mathcal{B}_n]
\rightarrow \textup{End}_\cS(n)$ satisfying
			
	\begin{equation*}		
			\mu_n(\sigma_i) =
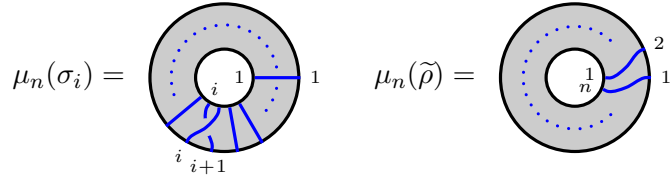
			\begin{pspicture}[shift=-1.4](-1.5,-1.5)(1.5,1.5)
       \definecolor{diskmid}{gray}{0.8}
    	\pscircle[fillstyle=ccslope,slopebegin=diskin,slopeend=diskout,linewidth=1.2pt](0,0){1.25}
    	\pscircle[fillstyle=solid, fillcolor=white, linewidth=1.2pt](0,0){0.5}
    	\degrees[18]
	\rput{0}(1.5;0){\tiny $1$}
	\rput{0}(0.25;0){\tiny $1$}
     	\rput{0}(1.5;12){\tiny $i$}
	\rput{0}(1.5;13){\tiny $i\!+\!1$}
	 \rput{0}(0.25;11.5){\tiny $i$}
	   \psline[linecolor=line,linewidth=1.2pt](0.5;0)(1.25;0)
	\psline[linecolor=line,linewidth=1.2pt](0.5;11)(1.25;11)
	\psline[linecolor=line,linewidth=1.2pt](0.5;14)(1.25;14)
	\psline[linecolor=line,linewidth=1.2pt](0.5;15)(1.25;15)
	\pscurve[linecolor=line,linewidth=1.2pt](1.25;13)(1.15;13)(0.875;12.5)(0.6;12)(0.5;12)
	\psdot[linecolor=diskmid,dotsize=0.3](0.875;12.5)
	\pscurve[linecolor=line,linewidth=1.2pt](1.25;12)(1.15;12)(0.875;12.5)(0.6;13)(0.5;13)
	\psarc[linestyle=dotted, linecolor=line,linewidth=1.2pt](0,0){0.875}{0.5}{10}
	\psarc[linestyle=dotted, linecolor=line,linewidth=1.2pt](0,0){0.875}{16}{17.5}
   \end{pspicture}
	 \qquad
	\mu_n(\widetilde{\rho})= 
	\begin{pspicture}[shift=-1.4](-1.5,-1.5)(1.5,1.5)
    	\pscircle[fillstyle=ccslope,slopebegin=diskin,slopeend=diskout,linewidth=1.2pt](0,0){1.25}
    	\pscircle[fillstyle=solid, fillcolor=white, linewidth=1.2pt](0,0){0.5}
    	\degrees[16]
     	\rput{0}(1.5;0){\tiny $1$}
	\rput{0}(1.5;1){\tiny $2$}
     	\rput{0}(0.25;0.5){\tiny $1$}
	\rput{0}(0.25;14){\tiny $n$}
	\pscurve[linecolor=line,linewidth=1.2pt](1.25;0)(1.15;0)(0.85;15.5)(0.6;15)(0.5;15)
	\pscurve[linecolor=line,linewidth=1.2pt](1.25;1)(1.15;1)(0.85;0.5)(0.6;0)(0.5;0)
	\psarc[linestyle=dotted, linecolor=line,linewidth=1.2pt](0,0){0.85}{2}{14}
   \end{pspicture}
   \end{equation*}
Note that $\mu_n(\widetilde{\rho})=\rho$ and 
$\mu_n(\sigma_i)=t^{\frac{1}{4}}e_i+t^{-\frac{1}{4}}$, where the last equality follows from the Kauffman skein relation \eqref{kauffman}. Note also that 
$\mu_n(\sigma_i^{-1})=t^{-\frac{1}{4}}e_i+t^{\frac{1}{4}}$.

\begin{remark}
Let $\nu_n: \mathbb{C}[\mathcal{B}_n]\rightarrow H_n$ be the surjective algebra map 
satisfying $\nu_n(\widetilde{\rho})=\rho$
and $\nu_n(\sigma_i)=t^{\frac{1}{4}}T_i$, then we have $\psi_n\circ\nu_n=\mu_n$.
\end{remark}
Let $\mathcal{I}^{br}_n: \mathcal{B}_n\rightarrow\mathcal{B}_{n+1}$ be the group
homomorphism that topologically is described by sticking in an additional strand between the $n$th and
the first strand, with the new strand running ``behind" all other strands (but not wrapping around the pole).
For example,

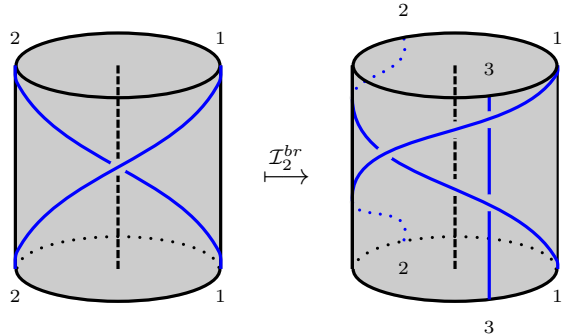
\begin{align*}
\psset{unit=0.45cm}
\begin{pspicture}[shift=-4.5](8,9.5)
	\psframe*[linecolor=diskin](1,2)(7,8)
	\psellipse*[linecolor=diskin](4,8)(3,1)
	\psellipse*[linecolor=diskin](4,2)(3,1)
	\psline[linewidth = 1.2pt](1,2)(1,8)
	\pssavepath[linewidth=1.2pt]{top}{\psellipse[linewidth=1.2pt](4,8)(3.025,1)}
	\pssavepath[linewidth=1.2pt]{bottom}{\psellipticarc[linewidth=1.2pt](4,2)(3.025,1){180}{0}}
	\psellipticarc[linestyle=dotted,linewidth=1.2pt](4,2)(3,1){0}{180}
	\psline[linewidth=1.2pt](7,2)(7,8)
	\psline[linewidth=1.2pt,linestyle=dashed,dash=3pt 1pt](4,2)(4,8)
  	\SpecialCoor
	\pnode(7,8){t1}
	\pnode(1,8){t2}
	\pnode(7,2){b1}
	\pnode(1,2){b2}
	\pssavepath[linecolor= line,linewidth=1.2pt]{under}{\pscurve[linecolor= line,linewidth=1.2pt](t2)([angle=-85,nodesep=0.5]t2)([angle=95,nodesep=0.5]b1)(b1)}
	\pssavepath[linecolor= diskin,linewidth=0.1pt]{over}{\pscurve[linecolor= line,linewidth=1.2pt](t1)([angle=-95,nodesep=0.5]t1)([angle=85,nodesep=0.5]b2)(b2)}
	\psintersect[name=x]{over}{under}
	\psdot[linecolor=diskin,dotsize=0.5](x1)
	\pscurve[linecolor= line,linewidth=1.2pt](t1)([angle=-95,nodesep=0.5]t1)([angle=85,nodesep=0.5]b2)(b2)
	\rput{0}([angle=90,nodesep=0.8]t1){\tiny $1$}
	\rput{0}([angle=90,nodesep=0.8]t2){\tiny $2$}
	\rput{0}([angle=-90,nodesep=0.8]b1){\tiny $1$}
	\rput{0}([angle=-90,nodesep=0.8]b2){\tiny $2$}
	\end{pspicture}
\overset{\mathcal{I}^{br}_2}{\longmapsto}
	\begin{pspicture}[shift=-4.5](8,9.5)
	\pssavepath[linecolor=white,linewidth=0.01pt]{insert}{\psline(5,1)(5,8)}
	\pssavepath[linecolor=white,linewidth=0.01pt]{two}{\psline(2.5,2)(2.5,9)}
	\pssavepath[linecolor=white,linewidth=0.01pt]{top}{\psellipse[linewidth=1.2pt](4,8)(3.025,1)}
	\pssavepath[linecolor=white,linewidth=0.01pt]{left}{\psline(1,2)(1,8)}
	\psframe*[linecolor=diskin](1,2)(7,8)
	\psellipse*[linecolor=diskin](4,8)(3,1)
	\psellipse*[linecolor=diskin](4,2)(3,1)
	\pssavepath[linewidth = 1.2pt]{left}{\psline(1,2)(1,8)}
	\pssavepath[linewidth=1.2pt]{bottom}{\psellipticarc[linewidth=1.2pt](4,2)(3.025,1){180}{0}}
	\pssavepath[linestyle=dotted,linewidth=1.2pt]{bottomback}{\psellipticarc(4,2)(3,1){0}{180}}
	\pssavepath[linewidth=1.2pt]z{right}{\psline(7,2)(7,8)}
	\psline[linewidth=1.2pt,linestyle=dashed,dash=3pt 1pt](4,2)(4,8)
  	\SpecialCoor
	\pnode(7,8){t1}
	\psintersect[name=t2]{two}{top}
	\psintersect[name=b2]{two}{bottomback}
	\pnode(7,2){b1}
	\psintersect[name=it]{insert}{top}
	\psintersect[name=ib]{insert}{bottom}
	\psline[linecolor=line, linewidth=1.2pt](it1)(ib1)
	\psdot[linecolor=diskin,dotsize=0.5](5,6.4)
	\psdot[linecolor=diskin,dotsize=0.5](5,3.95)
	\psdot[linecolor=diskin,dotsize=0.5](4,4.4)
	\psdot[linecolor=diskin,dotsize=0.5](4,6.1)
	\pssavepath[linecolor= line,linewidth=1.2pt]{under}{\pscurve[linecolor= line,linewidth=1.2pt](t21)([angle=-90,nodesep=0.5]t21)(1,7)([angle=95,nodesep=0.25]b1)(b1)}
	\pssavepath[linecolor= diskin,linewidth=0.1pt]{over}{\pscurve[linecolor= line,linewidth=1.2pt](t1)([angle=-95,nodesep=0.25]t1)(1,4)([angle=90,nodesep=0.5]b21)(b21)}
	\psintersect[name=x]{over}{under}
	\psdot[linecolor=diskin,dotsize=0.5](x1)
	\pscurve[linecolor= line,linewidth=1.2pt](t1)([angle=-95,nodesep=0.25]t1)(1,4)([angle=90,nodesep=0.5]b21)(b21)
	\pssavepath[linecolor=diskin,linewidth=0.01pt]{backline1}{\psline(1,7)(t21)}
	\pssavepath[linecolor=diskin, linewidth=0.01pt]{backline2}{\psline(1,4)(b21)}
	\psintersect[name=backpointstop]{backline1}{under}
	\psintersect[name=backpointsb]{backline2}{over}
	 \pstracecurve[linecolor=diskin,linewidth=2pt, istart=1, istop=3]{backpointstop}{under}
	  \pstracecurve[linecolor=line,linestyle=dotted,linewidth=1.2pt, istart=1, istop=3]{backpointstop}{under}
	\pstracecurve[linecolor=diskin,linewidth=2pt, istart=2, istop=4]{backpointsb}{over}
	  \pstracecurve[linecolor=line,linestyle=dotted,linewidth=1.2pt, istart=2, istop=4]{backpointsb}{over}
	\pssavepath[linewidth=1.2pt]{top}{\psellipse[linewidth=1.2pt](4,8)(3.025,1)}
	\pssavepath[linewidth = 1.2pt]{left}{\psline(1,2)(1,8)}
	\rput{0}([angle=90,nodesep=0.8]t1){\tiny $1$}
	\rput{0}([angle=90,nodesep=0.8]t21){\tiny $2$}
	\rput{0}([angle=-90,nodesep=0.8]b1){\tiny $1$}
	\rput{0}([angle=-90,nodesep=0.8]b21){\tiny $2$}
	\rput{0}([angle=90,nodesep=0.8]it1){\tiny $3$}
	\rput{0}([angle=-90,nodesep=0.8]ib1){\tiny $3$}
	\end{pspicture}
\end{align*}
It is the unique group homomorphism satisfying
\begin{equation*}
\begin{split}
&\mathcal{I}_n^{br}(\sigma_i)=\sigma_i,\qquad i=1,\ldots,n-1,\\
&\mathcal{I}_n^{br}(\sigma_n)=\sigma_n\sigma_{n+1}\sigma_n^{-1},\\
&\mathcal{I}_n^{br}(\widetilde{\rho})=\widetilde{\rho}\sigma_n^{-1}.
\end{split}
\end{equation*}
Extending $\mathcal{I}_n^{br}$ linearly to an algebra map $\mathcal{I}_n^{br}: 
\mathbb{C}[\mathcal{B}_n]\rightarrow\mathbb{C}[\mathcal{B}_{n+1}]$, we have 
\[
\mu_{n+1}\circ\mathcal{I}_n^{br}=\mathcal{I}\vert_{\cS_{n}}\circ\mu_n
\]
with $\mu_n: \mathbb{C}[\mathcal{B}_n]\rightarrow\cS_{n}$ the algebra map as defined in the previous section. 

In addition, it is easy to show that there exists a unique unit preserving algebra map
$\mathcal{I}_n^{ha}: H_n\rightarrow H_{n+1}$ satisfying
\begin{equation*}
\begin{split}
&\mathcal{I}_n^{ha}(T_i)=T_i,\qquad i=1,\ldots,n-1,\\
&\mathcal{I}_n^{ha}(T_n)=T_nT_{n+1}T_n^{-1},\\
&\mathcal{I}_n^{ha}(\rho)=t^{-\frac{1}{4}}\rho T_n^{-1},
\end{split}
\end{equation*}
and 
\[
\nu_{n+1}\circ\mathcal{I}_n^{br}=\mathcal{I}_n^{ha}\circ\nu_n
\]
with $\nu_n: \mathbb{C}[\mathcal{B}_n]\rightarrow H_n$ as defined in the previous section.

\begin{remark}
The maps $\mathcal{I}_n^{br}$ and $\mathcal{I}_n$ were constructed before in 
\cite{AlH,GS}.
\end{remark}

We end the section by discussing the relation to the braid group and the affine Temperley-Lieb algebra
of type $B$.
Let $\mathcal{B}_n^B$ be the braid group of type $\text{B}_n$, i.e. the group with generators 
$\sigma_0^B,\ldots,\sigma_{n-1}^B$ and defining relations the braid relations associated to the type $B$ Coxeter diagram
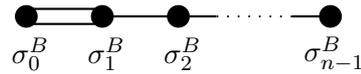
\begin{center}
{\psset{unit=1cm}
 \begin{pspicture}(0,0)(4,1.75)
 \psdot[dotsize=0.3](0,1)
 \psdot[dotsize=0.3](1,1)
 \psdot[dotsize=0.3](2,1)
 \psdot[dotsize=0.3](4,1)
 \psline(1,1)(2,1)
 \psline(2,1)(2.5,1)
 \psline[linestyle=dotted](2.5,1)(3.5,1)
 \psline(3.5,1)(4,1)
 \psline(0,1.1)(1,1.1)
  \psline(0,0.9)(1,0.9)
  \rput{0}(0,0.5){$\;\sigma_0^B$}
    \rput{0}(1,0.5){$\;\sigma_1^B$}
      \rput{0}(2,0.5){$\;\sigma_2^B$}
        \rput{0}(4,0.5){$\;\sigma_{n-1}^B$}
 \end{pspicture}  
}
\end{center}
It is known that $\mathcal{B}_n^B$ is isomorphic to the extended affine braid group $\mathcal{B}_n$,
with the isomorphism given by
\begin{equation}\label{isob}
\begin{split}
	\sigma_0^B &\mapsto \rho \sigma^{-1}_{n-1} \cdots \sigma_1^{-1} \\
	\sigma_i^B &\mapsto \sigma_i \qquad (1\leq i < n),
	\end{split}
\end{equation}
see \cite[Rem. 1.1]{GTW} and references therein. 
We discuss now a similar $B$-type presentation of the extended affine Temperley-Lieb algebra $\textup{TL}_n$.

The $B$-type affine Temperley-Lieb algebra $\textup{TL}_n^B$ is defined as follows (see \cite[Thm. 3.13]{Prz}). 
For $n\geq2$, $\textup{TL}^B_n$ is the unital complex associative algebra with generators $\alpha, \tau, e_1,\cdots, e_{n-1}$ and defining relations
\begin{align*}
&e_i^2 = \bigl(-t^{\frac{1}{2}} - t^{-\frac{1}{2}}\bigr) e_i, \\
&e_i e_j = e_j e_i & &\text{ if } |j-i| \geq 2, \\
& e_i e_{i\pm 1} e_i=e_i,\\
&\tau e_i = e_i \tau & &\text{ if } i>1, \\
&e_1\tau e_1 = \alpha e_1=e_1\alpha,\\
&\tau^2 = -t^{\frac{1}{2}} \alpha \tau-t.
\end{align*}
For $n=0$ and $n=1$ we set $\textup{TL}^B_0:=\C[\alpha]$ and  $\textup{TL}^B_1:=\C[\tau,\tau^{-1}]$. 

Note that $\tau$ is invertible with inverse $\tau^{-1} = -t^{-1}\tau-t^{-\frac{1}{2}}\alpha$. 
Hence $\alpha= -t^{-\frac{1}{2}} \tau - t^{\frac{1}{2}}\tau^{-1}$, and $\alpha$ is central.
For $n=1$ we define $\alpha$ by this formula.

Note that the assignments
\begin{align*}
		\sigma_0^B &\mapsto -t^{-\frac{3}{4}} \tau\\
		\sigma_i^B &\mapsto t^{\frac{1}{4}} e_i + t^{-\frac{1}{4}} \qquad(1 \leq i <n)
\end{align*}
define a surjective algebra map $\mu_n^B: \mathbb{C}[\mathcal{B}_n^B]\rightarrow
\textup{TL}_n^B$. In particular, $\textup{TL}_n^B$ is isomorphic to a quotient of the group algebra $\C[\mathcal{B}^B_n]$. 

Recall the algebra map $\mu_n: \C[\mathcal{B}_n] \rightarrow \textup{TL}_n$ from Section \ref{section7}.  The following result is an algebraic reformulation of \cite[Thm. 3.13(a)]{Prz}, see
also Remark \ref{PrzLink}.
\begin{proposition}\label{Bprop}
There exists a unique isomorphism $\textup{TL}_n^B\overset{\sim}{\longrightarrow}\textup{TL}_n$
of algebras such that the diagram
\[
\setlength{\arraycolsep}{1cm}
\begin{array}{cc}
\Rnode{a}{\C[\mathcal{B}^B_n]} & \Rnode{b}{\C[\mathcal{B}_n]}\\[1.5cm]
\Rnode{c}{\textup{TL}^B_n} & \Rnode{d}{\textup{TL}_n}\\[1.5cm]
\end{array}
\psset{nodesep=5pt,arrows=->,arrowsize=2pt 3}
\ncLine{a}{b}\Aput{\sim}
\ncLine{c}{d}\Aput{\sim}
\psset{nodesep=5pt,arrows=->>, arrowsize=2pt 3}
\ncLine{a}{c}\Bput{\mu^B_n}
\ncLine{b}{d}\Aput{\mu_n}
\]
of algebra maps is commutative, with the isomorphism $\mathbb{C}[\mathcal{B}_n^B]
\overset{\sim}{\longrightarrow}\mathbb{C}[\mathcal{B}_n]$ given by \eqref{isob}.
\end{proposition}
\begin{proof}
We have to show that there exists a well defined algebra map
$f_n: \textup{TL}_n^B\rightarrow \textup{TL}_n$ satisfying
$f_n(e_i)=e_i$ ($1\leq i<n$) and 
\[
f_n(\tau)= -t^{\frac{3}{4}} \rho \mu_n(\sigma_{n-1}^{-1})\cdots\mu_n(\sigma_1^{-1}),
\]
and that there exists a well defined algebra map $g_n: \textup{TL}_n\rightarrow\textup{TL}_n^B$
satisfying $g_n(e_i)=e_i$ ($1\leq i<n$) and 
\[
g_n(\rho)=-t^{-\frac{3}{4}}\tau\mu_n^B(\sigma_1^B)\cdots\mu_n^B(\sigma_{n-1}^B).
\]
We omit the proof as it is a straightforward check that all the algebra relations are respected by 
$f_n$ and $g_n$. For these checks it is convenient to use the presentation of $\textup{TL}_n$ in terms of the generators $\rho^{\pm 1},e_1,\ldots,e_{n-1}$ as given in Remark \ref{restrictedgenerators}.
\end{proof}
\begin{remark}\label{PrzLink}
Combining Proposition \ref{Bprop} with Theorem \ref{presentationThm} and Remark \ref{isoRem}
yields an isomorphism $\textup{TL}_n^B\overset{\sim}{\longrightarrow}\textup{End}_{\mathcal{S}}(n)$
given by 
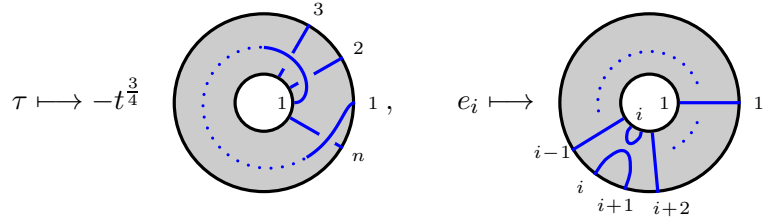
\begin{equation*}
\tau \longmapsto -t^{\frac{3}{4}} 
{\psset{unit=.8cm}
 \begin{pspicture}[shift=-1.9](-2,-2)(2,2)
    	\pscircle[fillstyle=solid,fillcolor=diskin,linewidth=1.2pt](0,0){1.5}
    	\pscircle[fillstyle=solid, fillcolor=white, linewidth=1.2pt](0,0){0.5}
	\degrees[12]       
    \psline[linecolor=line,linewidth=1.2pt](1.5;1)(0.5;1)
    \psline[linecolor=line,linewidth=1.2pt](1.5;2)(0.5;2)
    \psline[linecolor=line,linewidth=1.2pt](1.5;-1)(0.5;-1)
      \psdot[linecolor=diskin,dotsize=0.3](0.8;1)
          \psdot[linecolor=diskin,dotsize=0.3](0.8;2)
           \psdot[linecolor=diskin,dotsize=0.3](1.2;-1)
        \rput{0}(0;0){ \psplot[algebraic, plotpoints=400,linecolor=line,polarplot=true,linewidth=1.2pt]{0}{0.5 Pi mul }{(2/Pi)*ASIN(x/Pi-1) +(3*Pi-x)/(2*Pi)}}
      \rput{0}(0;0){ \psplot[algebraic, plotpoints=400,linecolor=line,linestyle=dotted,polarplot=true,linewidth=1.2pt]{0.51 Pi mul}{1.7 Pi mul }{(2/Pi)*ASIN(x/Pi-1) +(3*Pi-x)/(2*Pi)}}
        \rput{0}(0;0){ \psplot[algebraic, plotpoints=400,linecolor=line,polarplot=true,linewidth=1.2pt]{1.71 Pi mul}{2 Pi mul }{(2/Pi)*ASIN(x/Pi-1) +(3*Pi-x)/(2*Pi)}}
    	\rput{0}(1.8;0){\tiny$1$}
	\rput{0}(1.8;1){\tiny$2$}
	\rput{0}(1.8;2){\tiny$3$}
	\rput{0}(1.8;-1){\tiny$n$}
	\rput{0}(0.3;0){\tiny$1$}	
   \end{pspicture}  
},
\qquad	e_i \longmapsto    \begin{pspicture}[shift=-1.65](-1.75,-1.75)(1.5,1.5)
     	\SpecialCoor
    	\pscircle[fillstyle=ccslope,slopebegin=diskin,slopeend=diskout,linewidth=1.2pt](0,0){1.5}
    	\pscircle[fillstyle=solid, fillcolor=white, linewidth=1.2pt](0,0){0.5}
    	\degrees[17]
     	\rput{0}(1.8;0){\tiny $1$}
	\rput{0}(1.8;9.8){\tiny $i\!-\!1$}
    	\rput{0}(1.8;10.9){\tiny $i$}
	\rput{0}(1.8;11.9){\tiny $i\!+\!1$}
	\rput{0}(1.8;13.3){\tiny $i\!+\!2$}
     	\rput{0}(0.25;0){\tiny $1$}
    	\rput{0}(0.25;10.8){\tiny $i$}
   	\psline[linecolor=line,linewidth=1.2pt](0.5;0)(1.5;0)
	\psline[linecolor=line,linewidth=1.2pt](0.5;10)(1.5;10)
	\psline[linecolor=line,linewidth=1.2pt](0.5;13)(1.5;13)
	\pscurve[linecolor=line,linewidth=1.2pt](0.5;11)(0.6;11)(0.75;11.5)(0.6;12)(0.5;12)
	\pscurve[linecolor=line,linewidth=1.2pt](1.5;11)(1.15;11)(0.9;11.5)(1.15;12)(1.5;12)
	\psarc[linestyle=dotted, linecolor=line,linewidth=1.2pt](0,0){0.85}{1}{9}
	\psarc[linestyle=dotted, linecolor=line,linewidth=1.2pt](0,0){0.85}{14}{16}		
   \end{pspicture} 
\end{equation*}
for $1\leq i<n$. Note that under this isomorphism,
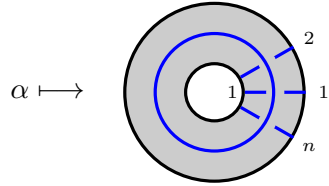
\begin{equation*}
\alpha \longmapsto 
 \begin{pspicture}[shift=-1.9](-2,-2)(2,2)
    	\pscircle[fillstyle=solid,fillcolor=diskin,linewidth=1.2pt](0,0){1.5}
    	\pscircle[fillstyle=solid, fillcolor=white, linewidth=1.2pt](0,0){0.5}
	\degrees[12]       
	 \psline[linecolor=line,linewidth=1.2pt](1.5;0)(0.5;0)
    \psline[linecolor=line,linewidth=1.2pt](1.5;1)(0.5;1)
    \psline[linecolor=line,linewidth=1.2pt](1.5;-1)(0.5;-1)
      \psdot[linecolor=diskin,dotsize=0.3](1;0)
      \psdot[linecolor=diskin,dotsize=0.3](1;1)
          \psdot[linecolor=diskin,dotsize=0.3](1;2)
           \psdot[linecolor=diskin,dotsize=0.3](1;-1)
         \pscircle[linecolor=line,linewidth=1.2pt](0,0){1}
       	\rput{0}(1.8;0){\tiny$1$}
	\rput{0}(1.8;1){\tiny$2$}
	\rput{0}(1.8;-1){\tiny$n$}
	\rput{0}(0.3;0){\tiny$1$}	
   \end{pspicture}  
\end{equation*}
This is the algebra isomorphism $\textup{TL}_n^B\overset{\sim}{\longrightarrow}\textup{End}_{\mathcal{S}}(n)$ from
\cite[Thm. 3.13(a)]{Prz}.
\end{remark}

Using the B type presentation of $\textup{TL}_n$  the algebra maps
$\mathcal{I}_n: \textup{TL}_n\rightarrow\textup{TL}_{n+1}$ takes on a simple form.

\begin{corollary}\label{insertionB}
The algebra maps $\mathcal{I}_n: \textup{TL}_n \rightarrow  \textup{TL}_{n+1}$, viewed as algebra maps  $ \textup{TL}^B_n \rightarrow  \textup{TL}^B_{n+1}$ via the identification  $\textup{TL}^B_n \cong  \textup{TL}_{n}$, satisfies $\tau \mapsto \tau$ and  $e_i \mapsto e_i$
for $1\leq i<n$.
\end{corollary}
Using Corollary \ref{insertionB} in combination with \cite[Thm. 3.13(b)]{Prz} it follows that
the algebra maps $\mathcal{I}_n:\textup{TL}_n\rightarrow\textup{TL}_{n+1}$ are injective.

 \bibliographystyle{plain}
\bibliography{bibliography,personalnotes} 

\end{document}